\documentclass[11pt,a4paper,reqno]{amsart}

\usepackage{mathtools}
\usepackage{babel}
\mathtoolsset{showonlyrefs=true}

\usepackage{amssymb}
\usepackage{amsthm}
\usepackage[all]{xy}
\usepackage{graphicx}
\usepackage[dvipsnames]{xcolor}
\usepackage{enumitem}
\usepackage{tikz}
\usetikzlibrary{matrix}

\usepackage{stix}

\usepackage[final, babel]{microtype}

\usepackage{geometry}

\newcommand{\OO}{\mathcal O}

\newcommand{\eps}{\varepsilon}
\newcommand{\Z}{\mathbb Z}

\newcommand{\N}{\mathbb N}
\newcommand{\F}{\mathbb F}
\newcommand{\Hom}{\operatorname{\bf Hom}}
\newcommand{\End}{\operatorname{\bf End}}
\newcommand{\rad}{\operatorname{rad}}
\newcommand{\GL}{\operatorname{\bf GL}}
\renewcommand{\Im}{\operatorname{Im}}
\newcommand{\Ker}{\operatorname{Ker}}
\newcommand{\Tr}{\operatorname{Tr}}
\newcommand{\id}{\operatorname{id}}
\newcommand{\add}{\operatorname{add}}
\newcommand{\soc}{\operatorname{soc}}

\newcommand{\IBr}{\operatorname{IBr}}

\newcommand{\Aut}{\operatorname{\bf Aut}}
\newcommand{\Out}{\operatorname{\bf Out}}
\newcommand{\Outcent}{\operatorname{\bf Outcent}}

\newcommand{\opp}{\operatorname{op}}
\newcommand{\rank}{\operatorname{rank}}

\newcommand{\Pic}{{\operatorname{\bf Pic}}}
\newcommand{\PicS}{\operatorname{\bf Pic}^{\scriptscriptstyle \mathbf K}}
\newcommand{\PicO}{\operatorname{\bf Pic}^{\scriptscriptstyle{0}}}
\newcommand{\TrPic}{{\operatorname{\bf TrPic}}}
\newcommand{\Picent}{{\operatorname{\bf Picent}}}
\newcommand{\Autcent}{{\operatorname{\bf Autcent}}}
\newcommand{\TrPicent}{{\operatorname{\bf TrPicent}}}

\newcommand{\silt}{\textrm{{\bf{silt}}-}}

\newcommand{\tilt}{\textrm{{\bf{tilt}}-}}
\newcommand{\tsilt}{\textrm{{\bf{t-silt}}-}}
\newcommand{\proj}{\textrm{{\bf{proj}}-}}
\newcommand{\Proj}{\textrm{{\bf{Proj}}-}}
\newcommand{\Mod}{\textrm{{\bf{Mod}}-}}
\newcommand{\projL}{\textrm{-\bf{proj}}}
\newcommand{\modC}{\textrm{{\bf{mod}}-}}
\newcommand{\LambdaTw}{\Lambda_{\rm tw}}
\newcommand{\GammaTw}{\Gamma_{\rm tw}}
\newcommand{\Gro}{\operatorname{\bf K_0}}

\newcommand{\Isom}{\operatorname{\bf Isom}}

\newcommand{\upbr}[1]{^{\scriptscriptstyle(#1)}}

\usepackage{relsize}
\newcommand{\xbullet}{\bullet}

\renewcommand{\mathfrak}{\mathcal}

\theoremstyle{definition}
\newtheorem{defi}{Definition}[section]

\theoremstyle{plain}
\newtheorem{thm}[defi]{Theorem}
\newtheorem{lemma}[defi]{Lemma}
\newtheorem{corollary}[defi]{Corollary}

\newtheorem{prop}[defi]{Proposition}

\newtheorem*{thm*}{Theorem}

\newtheorem{innercustomthm}{Theorem}
\newenvironment{customthm}[1]
{\renewcommand\theinnercustomthm{#1}\innercustomthm}
{\endinnercustomthm}

\theoremstyle{remark}
\newtheorem{remark}[defi]{Remark}

\newtheorem*{setup}{Set-up}

\newcommand{\sqq}[1]{(\!(#1)\!)}
\newcommand{\sbb}[1]{[\![#1]\!]}

\newcommand{\mm}{m_\xbullet}
\newcommand{\cc}{c_\xbullet}
\renewcommand{\tt}{t_\xbullet}

\newcommand*\xoverline[1]{%
	\hbox{%
		\vbox{%
			\hrule height 0.5pt 
			\kern0.3ex
			\hbox{%
				\kern-0.0em
				\ensuremath{#1}%
				\kern-0.0em
			}%
		}%
	}%
}

\usepackage[colorlinks,breaklinks,allcolors=Blue,bookmarks=false,hypertexnames=true]{hyperref} 

\title{Bijections of silting complexes and derived Picard groups}

\author{Florian Eisele}
\address{Department of Mathematics, City, University of London, London EC1V 0HB, United Kingdom}
\email{florian.eisele@city.ac.uk}
\thanks{This research was supported by EPSRC grant EP/T004592/1.}

\newtheoremstyle{named}{}{}{\itshape}{}{\bfseries}{.}{.5em}{\thmnote{#3}}
\theoremstyle{named}

\renewcommand{\leq}{\leqslant}
\renewcommand{\geq}{\geqslant}


\begin{document}

\begin{abstract}
	We introduce a method that produces a bijection between the posets $\silt{A}$ and $\silt{B}$ formed by the isomorphism classes of basic silting complexes  over finite-dimensional $k$-algebras $A$ and $B$,  by lifting $A$ and $B$ to two $k\sbb{X}$-orders which are isomorphic as rings. We apply this to a class of algebras generalising Brauer graph and weighted surface algebras, showing that their silting posets are multiplicity-independent in most cases. Under stronger hypotheses we also prove the existence of large multiplicity-independent subgroups in their derived Picard groups as well as multiplicity-invariance of $\TrPicent$.  As an application to the modular representation theory of finite groups we show that if $B$ and $C$ are blocks with $|\IBr(B)|=|\IBr(C)|$ whose defect groups are either both cyclic, both dihedral or both quaternion, then the posets $\tilt{B}$ and $\tilt{C}$ are isomorphic (except, possibly, in the quaternion case with  $|\IBr(B)|=2$) and $\TrPicent(B)\cong\TrPicent(C)$ (except, possibly, in the quaternion and dihedral cases with $|\IBr(B)|=2$).
\end{abstract}

\maketitle


\setcounter{tocdepth}{1}
\tableofcontents


\section{Introduction}

The notion of a \emph{silting complex} over a finite-dimensional algebra was first introduced by Keller and Vossieck \cite{KellerVossieck}, and is closely related to Rickard's stronger notion of a \emph{tilting complex} \cite{RickardOne}. It was later discovered by Aihara and Iyama \cite{AiharaIyamaSiltingMutation} that silting complexes have a well-behaved mutation theory, which kindled a wider interest in this class objects. One of the most obvious problems to consider in this context would be their classification over a given algebra $A$. For most algebras a classification of either silting or tilting complexes is entirely out of reach (although there are exceptions, e.g. \cite{AiharaMizuno}), but two-term silting (and tilting) complexes are much more accessible thanks to Adachi, Iyama and Reiten's theory of \emph{$\tau$-tilting modules} \cite{AdachiIyamaReiten}, and classification results in this area include Brauer tree and Nakayama algebras \cite{ZonarevaTwoTerm, BrauerGraphTauTilt, AdachiTauTiltingNakayama}, algebras of dihedral, semi-dihedral and quaternion type \cite{TauRigid}, and more \cite{IyamaZhang, MizunoClassTauTilt}. The present paper aims to generalise some of the results of \cite{TauRigid}, which was joint work of Janssens, Raedschelders and the author, to silting complexes of arbitrary length.

The main tool of \cite{TauRigid} was a one-to-one correspondence between two-term silting complexes over $A$ and $A/zA$, where $A$ is a finite-dimensional algebra over a field $k$ and $z\in \rad(Z(A))$ is arbitrary. Unsurprisingly, this fails for complexes of length greater than two, the problem being that neither complexes nor the morphisms between them lift from $A/zA$ to $A$ in general. A key idea of the present paper is to consider a $k\sbb{X}$-order $\Gamma$, by which we mean a $k\sbb{X}$-algebra which is free and finitely generated as a $k\sbb{X}$-module, such that $A\cong\Gamma / X\Gamma$. It is known \cite{RickardLiftTilting} that pre-silting complexes over $A$ always lift to $\Gamma$, even uniquely, and the converse holds as well (see Proposition~\ref{prop reduction silting}). Taking into account that silting complexes also need to generate, we obtain that $\silt{A}$ and a certain set $\tsilt{\Gamma}\supseteq \silt{\Gamma}$ (see Definition~\ref{def tsilt}) are in bijection. Now one just needs to realise that $\Gamma$ can be turned into a $k\sbb{X}$-order in many different ways, and therefore has many different ``reductions modulo $X$'', while $\tsilt{\Gamma}$ (like $\silt{\Gamma}$) only depends on the structure of $\Gamma$ as a ring. Hence, if $\iota_1,\iota_2:\ k\sbb{X} \hookrightarrow Z(\Gamma)$ are two different ways of turning $\Gamma$ into a $k\sbb{X}$-order, with reductions modulo $X$ being finite-dimensional $k$-algebras $A$ and $B$, then we get a diagram
\begin{equation}
	\xymatrix{
		&&\tsilt{\Gamma} \ar@{<->}[rd]^\sim \ar@{<->}[ld]_\sim \\
		\silt{A} \ar@{=}[r]&\silt{\Gamma/\iota_1(X)\Gamma} \ar@{<->}[rr]^\sim && \silt{\Gamma/\iota_2(X)\Gamma} & \ar@{=}[l] \silt{B}.
	}
\end{equation}
These bijections do not alter the terms of complexes in a non-trivial way, but the effect of lifting and subsequent reduction on differentials is less straightforward. The relationship between $A$ and $B$ is not obvious either, and in particular they may have different $k$-dimensions since the $k\sbb{X}$-rank of $\Gamma$ depends on the chosen $k\sbb{X}$-algebra structure. This principle, formally stated in Corollary~\ref{corollary silting bijection}, is quite versatile since we are not imposing any structural restrictions on the $k\sbb{X}$-order $\Gamma$, and it should have applications beyond what we do in the present article.

Now we need to identify some families of algebras that arise as the reduction of a single ring~$\Gamma$ with respect to different $k\sbb{X}$-algebra structures. A first example are Brauer tree and certain Brauer graph algebras, which Gnedin \cite{GnedinRibbonOrders} showed to have lifts, called ``\emph{Ribbon graph orders}'', whose ring structure is manifestly independent of the \emph{multiplicities} involved. In the present article we will define a much larger class of algebras, comprising Brauer graph algebras and \emph{weighted surface algebras} as defined by Erdmann and Skowro\'{n}ski \cite{ErdmannSkonrowski1, ErdmannSkonrowski2}. This class is similar to what is outlined under the heading ``the general context'' in \cite{ErdmannSkonrowski2}. We call an algebra in this class a \emph{generalised weighted surface algebra}, denoted $\Lambda(Q,f,\mm,\cc,\tt,\mathcal Z)$. We then construct lifts of these algebras to $k\sbb{X}$-orders provided the \emph{multiplicities} are big enough, and use that to establish \emph{multiplicity-independence} of $\silt{\Lambda(Q,f,\mm,\cc,\tt,\mathcal Z)}$.

Our main results, which are Theorems~\ref{thm main general},~\ref{thm trpic q3k}~and~\ref{thm Brauer graph}, are stated in terms of these \emph{generalised weighted surface algebras} and \emph{twisted Brauer graph algebras}, which the reader may not be familiar with. What we will do in this introduction is state explicit versions of these results for \emph{Brauer graph algebras} and certain \emph{blocks} of group algebras of finite groups,  two widely studied classes of algebras which were actually the main intended application. Since these algebras are symmetric, the notions of silting and tilting coincide and we even get correspondences of tilting complexes.

Recall that a  \emph{Brauer graph} is a finite undirected graph $G$ equipped with a cyclic order on the set of half-edges incident to $v$ for each vertex $v\in G_0$. Once we assign a \emph{multiplicity} to each vertex by means of a function $\mm: G_0\longrightarrow \Z_{>0}$, we can define the  \emph{Brauer graph algebra} $A(G,\mm)$, which is symmetric and special biserial. See \cite{SchrollBGA} for a survey on these algebras.
\begin{customthm}{A}[Tilting bijections in Brauer graph algebras]\label{thm bga tilting correspondence}
	Let $k$ be an algebraically closed field and let $G$ be a Brauer graph with two sets of multiplicities $\mm\upbr{1},\mm\upbr{2}:\ G_0 \longrightarrow \Z_{>0}$. Assume that either
	\begin{enumerate}
		\item $\operatorname{char}(k)=2$, or
		\item $G$ is bipartite, or
		\item $\mm\upbr{1}$ and $\mm\upbr{2}$ only take values $\geq 2$.
	\end{enumerate}
	Then there is a bijection between the isomorphism classes of (pre-)tilting complexes over $A(G,\mm\upbr{1})$ and those over $A(G,\mm\upbr{2})$, inducing a poset isomorphism
	\begin{equation}
		\tilt{A(G,\mm\upbr{1})}\stackrel{\sim}{\longleftrightarrow} \tilt{A(G,\mm\upbr{2})}.
	\end{equation}
\end{customthm}

\begin{customthm}{B}[Tilting bijections in blocks]\label{thm block tilting correspondence}
	Let $k$ be an algebraically closed field of characteristic ${p>0}$, and let $A$ and $B$ be blocks of group algebras of finite groups defined over $k$ such that $|\IBr(A)|=|\IBr(B)|$. Assume that the defect groups of $A$ and $B$ are
	\begin{enumerate}
		\item cyclic groups of orders $p^a$ and $p^b$, or
		\item  dihedral groups of orders $2^a$ and $2^b$, or
		\item  quaternion groups of orders $2^a$ and $2^b$ and $|\IBr(A)|=|\IBr(B)|\neq 2$,
	\end{enumerate}
	for arbitrary $a,b\geq 1$ (in the cyclic case) or $a,b\geq 3$ (in the other two cases). Then there is a bijection between the isomorphism classes of (pre-)tilting complexes over $A$ and those over $B$, inducing a poset isomorphism
	\begin{equation}
		\tilt{A}\stackrel{\sim}{\longleftrightarrow} \tilt{B}.
	\end{equation}
\end{customthm}
It should be noted that the $k\sbb{X}$-orders used to prove Theorem~\ref{thm bga tilting correspondence} do not usually have semisimple $k\sqq{X}$-span and are not canonical in any way. By contrast, the lifts used in Theorem~\ref{thm block tilting correspondence} happen to be orders in semisimple $k\sqq{X}$-algebras, and in many ways look like equicharacteristic versions of block algebras over an extension $\OO$ of the $p$-adic integers. This fact is exploited in the second half of the paper, which is devoted to derived Picard groups.

The \emph{derived Picard group} of a finite-dimensional $k$-algebra $A$, denoted $\TrPic_k(A)$, is the group of \emph{standard auto-equivalences} of $\mathcal D^b(A)$. This object was first considered by Yekutieli \cite{YekutieliDPic} and Rouquier and Zimmermann \cite{RouquierZimmermann}.  The group $\TrPic_k(A)$ is a \emph{locally algebraic group} \cite{YekutieliDPicLocallyAlgebraic}, whose identity component is $\PicO_k(A)$, the identity component of the ordinary Picard group. What really put derived Picard groups into the limelight was the discovery by Seidel and Thomas \cite{SeidelThomasBraidGroup} that elements called ``spherical twists'' satisfy Braid relations and give rise to embeddings of Braid groups into the derived Picard groups of certain \emph{dg}-versions of Brauer tree algebras (a similar result was obtained independently in \cite{RouquierZimmermann}). This led to a number of papers proving the existence and faithfulness of Braid actions on derived categories of Brauer tree algebras, using that to fully determine their derived Picard groups \cite{ZimmermannTrPicBrauerTree, SchapsZYBraidAction, MuchtadiBraidFaithful,  ZvonarevaBrauerStar, VolkovZvonareva}. Of course there are results for other classes of algebras as well \cite{MiyachiYekutieliTrPicHereditary, DerivedDiscreteCatsI,  NordskovaVolkov}.

The group $\TrPic_k(A)$ acts on $\silt{A}$ and $\tilt{A}$, with kernel containing $\PicO_k(A)$. Therefore Theorems~\ref{thm bga tilting correspondence}~\and~\ref{thm block tilting correspondence} should have some implications for derived Picard groups. The group $\PicO_k(A)$ is certainly not multiplicity-independent, but it might have a multiplicity-independent complement or supplement.  The rough idea is that for a $k\sbb{X}$-order $\Gamma$ which reduces to an algebra $A$ we have a group homomorphism~${\TrPic_{k\sbb{X}}(\Gamma)\longrightarrow \TrPic_k(A)}$. The group $\TrPic_{k\sbb{X}}(\Gamma)$ still depends on the $k\sbb{X}$-algebra structure on $\Gamma$, though, and in general the aforementioned homomorphism is neither surjective (even modulo $\Pic_k(A)$) nor injective. Therefore we must rely on a number of favourable properties of our chosen lift~$\Gamma$ to prove Theorems~\ref{thm intro trpic bga}~and~\ref{thm intro trpic blocks} below. Our results give a conceptual explanation of multiplicity-independence results such as \cite{ZimmermannTrPicBrauerTree,MuchtadiBraidFaithful}, but more importantly also cover cases where there are no ``spherical objects'' in~$\mathcal D^b(A)$ that would allow to apply the results of \cite{SeidelThomasBraidGroup}. We also obtain very elegant multiplicity-independence results for $\TrPicent(A)$, the subgroup of $\TrPic_k(A)$ acting trivially on $Z(A)$.

\begin{customthm}{C}\label{thm intro trpic bga}
	Let $k$ be an algebraically closed field and let $G$ be a Brauer graph with two sets of multiplicities $\mm\upbr{1},\mm\upbr{2}:\ G_0 \longrightarrow \Z_{>0}$. Assume $G$ is a simple graph, and if $\operatorname{char}(k)\neq 2$ then also assume that $G$ is bipartite. Then
	\begin{equation}
		\TrPicent(A(G,\mm\upbr{1})) \cong \mathcal \TrPicent(A(G,\mm\upbr{2})).
	\end{equation}
	If the multiplicities have the property that $m\upbr{1}_v=m\upbr{1}_w$ if and only if $m\upbr{2}_v=m\upbr{2}_w$ for all $v,w\in G_0$, then there are subgroups  $\mathcal H_i \leq \TrPic_k(A(G,\mm\upbr{i}))$ for $i\in\{1,2\}$, such that $\mathcal H_1\cong \mathcal H_2$ and
	\begin{equation}
		\TrPic_k(A(G,\mm\upbr{i})) = \Pic_k(A(G,\mm\upbr{i})) \cdot \mathcal H_i \quad\textrm{for $i\in\{1,2\}$.}
	\end{equation}
\end{customthm}

The subgroups $\mathcal H_i$ can be identified with a certain subgroup $\mathcal H$ of finite index in the derived Picard group of some $k\sbb{X}$-order $\Gamma$, and in that setting it is reasonable to identify $\mathcal H=\mathcal H_1=\mathcal H_2$, as we do in Theorems~\ref{thm trpic q3k}~and~\ref{thm Brauer graph}. The latter, together with Proposition~\ref{prop tw bga eq bga}, implies Theorem~\ref{thm intro trpic bga}. It is worth mentioning that in certain situations one can modify the  decomposition $\Pic_k(A(G,\mm\upbr{i})) \cdot \mathcal H_i$ slightly to obtain a semi-direct product decomposition (e.g. in Proposition~\ref{prop semidirect brauer tree}).

\begin{customthm}{D}\label{thm intro trpic blocks}
	Let $k$ be an algebraically closed field of characteristic ${p>0}$, and let $B_1$ and $B_2$ be blocks of group algebras of finite groups defined over $k$ such that $|\IBr(B_1)|=|\IBr(B_2)|$. Assume that the defect groups of $B_1$ and $B_2$ are
	\begin{enumerate}
		\item cyclic groups of orders $p^a$ and $p^b$, or
		\item  dihedral groups of orders $2^a$ and $2^b$ and $|\IBr(B_1)|=|\IBr(B_2)|\neq 2$, or
		\item  quaternion groups of orders $2^a$ and $2^b$ and $|\IBr(B_1)|=|\IBr(B_2)|\neq 2$,
	\end{enumerate}
	for arbitrary $a,b\geq d_0$, where we set $d_0=1$ in the cyclic case and $d_0 = 3$ in the other two cases. Then
	\begin{equation}
		\TrPicent(B_1) \cong \TrPicent(B_2).
	\end{equation}
	If $a=b=d_0$ or $a,b>d_0$ then there are $\mathcal H_i\leq \TrPic_k(B_i)$ for $i\in\{1,2\}$ such that $\mathcal H_1\cong \mathcal H_2$ and
	\begin{equation}
		\TrPic_k(B_i) = \Pic_k(B_i) \cdot \mathcal H_i \quad\textrm{for $i\in\{1,2\}$.}
	\end{equation}
\end{customthm}

Again one can do slightly better in some cases. For blocks of quaternion defect we get $\TrPic_k(B_i)=\PicS_k(B_i)\rtimes \mathcal H_i$, where $\PicS_k(B_i)$ is the subgroup of the Picard group that fixes the isomorphism classes of all simple modules. In the cyclic defect case we even have $\TrPic_k(B_i)=\PicS_k(B_i)\times \mathcal H_i$ (which, however, already follows from \cite{VolkovZvonareva}).

It would be interesting to see if there are other families of blocks which arise as the various reductions modulo $X$ of a ring $\Gamma$ endowed with different $k\sbb{X}$-algebra structures. The proof of Theorem~\ref{thm intro trpic blocks} uses only one such ring~$\Gamma$ for each $p$, type of defect group and number of simple modules. In this framework it would therefore be possible to formulate much more radical finiteness conjectures. Given the importance of derived equivalences in modular representation theory it would also be interesting to see if one can use $k\sbb{X}$-orders to construct derived equivalences between whole families of blocks.

\subsection*{Relation to other work} After a first preprint of this paper had appeared online, the author was informed that W. Gnedin has also, independently, studied bijections of silting posets. In particular, Gnedin has a version of Propositions~\ref{prop reduction silting}~and~\ref{prop silti to silt} as well as Corollary~\ref{corollary silting bijection}, mentioned in an Oberwolfach report~\cite{GnedinOWR}, which will appear in an upcoming paper. Gnedin can also show a version of Theorem~\ref{thm bga tilting correspondence} without restrictions on the multiplicities \cite{GnedinPrivateComm}.

\subsection*{Conventions} Modules are right modules by default. For a quiver $Q$ we let  $Q_0$ denote its set of vertices and $Q_1$ its set of arrows. Two-sided ideals in a ring $A$ are denoted by $(\ldots)_A$, or $(\ldots)$ when the choice of $A$ is unambiguous. If $R$ is a discrete valuation ring with field of fractions $K$, then we call an $R$-algebra which is free and finitely-generated as an $R$-module an \emph{$R$-order}. Given a $K$-algebra $A$ we say that $\Lambda \subseteq A$ is an \emph{$R$-order in $A$} if $\Lambda$ is an $R$-order which also spans $A$ as a vector space.

\section{Preliminaries}

\subsection{Silting and tilting complexes} Let $A$ be a ring for which  the Krull-Schmidt theorem holds in $\mathcal D^b(A)$, which is true for example if $A$ is an algebra over a field or an order over a complete discrete valuation ring.
\begin{defi}\label{defi silting}
	A complex $T^\xbullet \in \mathcal K^b(\proj{A})$ is called
	\begin{enumerate}
		\item \emph{pre-silting} if $\Hom_{\mathcal D^b(A)}(T^\xbullet, T^\xbullet[i])=0$ for all $i>0$,
		\item \emph{pre-tilting} if $\Hom_{\mathcal D^b(A)}(T^\xbullet, T^\xbullet[i])=0$ for all $i\neq 0$.
	\end{enumerate}
	A pre-silting (or pre-tilting) complex $T^\xbullet$ with the property  $\operatorname{thick}(T^\xbullet)=\mathcal K^b(\proj{A})$ is called \emph{silting} (or \emph{tilting}).  Define
	\begin{equation}
		\begin{gathered}
			\silt{A} = \{ \textrm{ basic silting complexes over $A$ } \} \big/ \textrm{ isomorphism},\\
			\tilt{A} = \{ \textrm{ basic tilting complexes over $A$ } \} \big/ \textrm{ isomorphism}.
		\end{gathered}
	\end{equation}
\end{defi}
We will sometimes refer to tilting complexes as defined above as \emph{one-sided tilting complexes}. It was shown by Rickard \cite{RickardDerEqDerFun} that, under mild hypotheses on $A$, every one-sided tilting complex is the restriction of a two-sided tilting complex as defined in the subsection below. The set $\silt{A}$ is partially ordered by defining  (see \cite[Definition 2.10]{AiharaIyamaSiltingMutation})
\begin{equation}\label{eqn def partial order}
	S^\xbullet \geq T^\xbullet \textrm{ if and only if } \Hom_{\mathcal D^b(A)}(S^\xbullet, T^\xbullet[i])=0 \textrm{ for all $i>0$.}
\end{equation}
Of course, this restricts to a partial order on $\tilt{A}$. There is also a mutation theory for $\silt{A}$ (see \cite{AiharaIyamaSiltingMutation}), while mutation in $\tilt{A}$ is not always possible. In general silting complexes tend to be better behaved than tilting complexes. Fortunately, most algebras we are interested in are symmetric, and for symmetric algebras the notions of silting and tilting coincide.

\begin{prop}
	If $A$ is a symmetric algebra over a field or a symmetric order over a complete discrete valuation ring then partial silting and partial tilting complexes coincide. \qed
\end{prop}
For algebras this is well-known, see for instance \cite[Example 2.8]{AiharaIyamaSiltingMutation} (or \cite[Lemma 3.1]{HoshinoKato} for a proof of the relevant version of Auslander-Reiten duality). For $R$-orders we can use the fact that pre-silting and pre-tilting complexes coincide over the reduction to the residue field of $R$, and then use Proposition~\ref{prop reduction silting} in conjunction with \cite[Proposition 3.1 and Theorem 3.3]{RickardLiftTilting} to get the same over the order. See also \cite{ZimmermannTiltedSymm}.

\subsection{Derived Picard groups} Let $R$ be a commutative ring, and let $A$ and $B$ be $R$-algebras that are projective as $R$-modules.
We call an object $X^\xbullet\in \mathcal D^b(A^{\opp}\otimes_R B)$ \emph{invertible} if there is a $Y^\xbullet\in \mathcal D^b(B^{\opp}\otimes_R A)$ such that $X^\xbullet \otimes_B^{\mathbf L} Y^\xbullet\cong A$ in $\mathcal D^b(A^{\opp}\otimes_R A)$ and   $Y^\xbullet \otimes_A^{\mathbf L} X^\xbullet\cong B$ in $\mathcal D^b(B^{\opp}\otimes_R B)$.  We call $X^\xbullet$ a \emph{two-sided tilting complex} if  $X^\xbullet$ is invertible and restricts to a bounded complex of projective left $A$-modules and to a bounded complex of projective right $B$-modules. As mentioned in \cite[Definition~4.2]{RickardDerEqDerFun} every invertible complex is isomorphic to a two-sided tilting complex.  If $X^\xbullet$ is a two-sided tilting complex, then the derived tensor product ``$X^\xbullet\otimes_B^{\mathbf L}-$'' and the ordinary tensor product of complexes ``$X^\xbullet\otimes_B-$'' coincide, which is why we will not need to use left derived tensor products in the remainder of this article.
\begin{defi}
	The \emph{derived Picard group} of $A$ is defined as
	\begin{equation}
		\TrPic_R(A)=\left \{  \textrm{ two-sided tilting complexes in $\mathcal D^b(A^{\opp}\otimes_RA)$ }  \right\} \bigg/ \textrm{isomorphism}.
	\end{equation}
	The product in this group is induced by ``$-\otimes_A=$''.
\end{defi}

For basic properties of derived Picard groups refer to \cite{RouquierZimmermann}. We will make extensive use of the fact that a two-sided  tilting complex $X^\xbullet$ induces an isomorphism
\begin{equation}
	\gamma_X:\ Z(A) \stackrel{\sim}{\longrightarrow} Z(B).
\end{equation}
This can be seen in a number of ways. For instance, the functor $X^{-1}\otimes_A - \otimes_A X$ sends the $A$-$A$-bimodule $A$ to the $B$-$B$-bimodule $B$, and therefore induces a homomorphism between the endomorphism rings of these bimodules, which are $Z(A)$ and $Z(B)$, respectively. In particular $\TrPic_R(A)$ acts on the centre of~$A$. It also acts on the Grothendieck group $\Gro(A)$, as we will see in the next subsection.
\begin{defi}
	Define the following subgroups of $\TrPic_R(A)$:
	\begin{equation}
		\begin{spreadlines}{0.8em}
			\begin{gathered}
				\TrPicent(A)= \big\{ X^\xbullet \in \TrPic_R(A)\ | \ \gamma_X =\id_{Z(A)} \big\},\\
				\Pic_R(A)= \big\{ X^\xbullet \in \TrPic_R(A) \ | \ \textrm{$X^\xbullet$ is isomorphic to an $A$-$A$-bimodule} \big\},\\
				\PicS_R(A) = \big\{ M \in \Pic_R(A) \ | \ \textrm{$P\otimes_AM\cong P$ for all projective $A$-modules $P$}  \big\},\\
				\Picent(A) = \Pic_R(A)\cap \TrPicent(A).
			\end{gathered}
		\end{spreadlines}
	\end{equation}
	If $R$ is a field and $A$ is finite-dimensional and basic then $\Pic_R(A)\cong \Out_R(A)$, which is the group of \emph{outer automorphisms} of $A$, and $\Picent(A)\cong \Outcent(A)$, the group of \emph{outer central automorphisms}.
\end{defi}

In the present paper we will often consider different $R$-algebra structures on the same ring $A$. This raises some questions regarding well-definedness, which the following remark addresses.
\begin{remark}\label{remark detect trpic linearity}
	\begin{enumerate}
		\item Let ${{A}}$ be a ring and let $R,S\subset Z({{A}})$ be two commutative subrings such that ${{A}}$ is projective both as an $R$-module and as an $S$-module. Then, technically, the elements of $\TrPic_R({{A}})$ are represented by complexes whose terms are $R$-linear bimodules, and those of $\TrPic_S({{A}})$ by complexes whose terms are $S$-linear bimodules. However, we can always embed
		      \begin{equation}
			      \TrPic_R({{A}}) , 	\TrPic_S({{A}})  \hookrightarrow \mathcal D^b({{A}}^{\opp}\otimes_{\mathbb Z} {{A}}),
		      \end{equation}
		      showing that expressions like ``$	\TrPic_R({{A}}) \cap 	\TrPic_S({{A}}) $'' are well-defined. If one wants to stick closer to the setting of \cite{RickardDerEqDerFun} one could embed into $\TrPic_k({{A}})$, provided there is a common subfield $k\subset R,S$ (which is the case in all examples we are interested in). In either case, these are embeddings because an element $X^\xbullet \in \TrPic_R({{A}})$ is trivial if and only if $H^i(X^\xbullet)=0$ for all $i\neq 0$ and $H^0(X^\xbullet) \cong {{A}}$ as an ${{A}}$-${{A}}$-bimodule (this follows, for example, from \cite[Proposition 2.3]{RouquierZimmermann}). This can be detected in $ \mathcal D^b({{A}}^{\opp}\otimes_{\mathbb Z} {{A}})$.
		\item If the automorphism $\gamma_{Y}\in \Aut_R(Z({{A}}))$ induced by a $Y^\xbullet\in\TrPic_R({{A}})$ happens to be $S$-linear as well, then there is an $X^\xbullet \in \TrPic_S({{A}})$ such that $X^\xbullet \cong Y^\xbullet $ in  $\mathcal D^b({{A}}^{\opp}\otimes_\Z {{A}})$. To find $X^\xbullet$ first pick a two-sided tilting complex  $Z^\xbullet\in \mathcal D^b({{A}}^{\opp} \otimes_S {{A}}')$, for a suitable $S$-algebra ${{A}}'$, whose restriction to the left is the restriction to the left of $Y^\xbullet$.  Then, by \cite[Proposition 2.3]{RouquierZimmermann}, we have $Z^\xbullet\otimes_{{{A}}'} {_\alpha {{A}}} \cong  Y^\xbullet $ in $\mathcal D^b({{A}}^{\opp}\otimes_\Z {{A}})$ for some ring isomorphism $\alpha:\ {{A}}' \longrightarrow {{A}}$. Note that, technically, \cite[Proposition 2.3]{RouquierZimmermann}  asks for ${{A}}$ and $A'$ to be projective over a commutative base ring. In all cases we are interested in we could in principle use a field $k$ for this. But the proof of \cite{RouquierZimmermann} goes through regardless since in our situation the existence of inverses of $Z^\xbullet$ and $Y^\xbullet$ is guaranteed by other means. It follows that $\alpha|_{Z({{A}}')} \circ \gamma_{Z} = \gamma_{Y}$ (where $\gamma_{Z}:\ Z(A)\longrightarrow Z(A')$ is induced by $Z^\xbullet$), which shows that $\alpha|_{Z({{A}}')}$ must be $S$-linear, implying that ${_{\alpha}{{A}}}$ is an $S$-linear bimodule. We can therefore choose $X^\xbullet = Z^\xbullet\otimes_{{{A}}'} {_\alpha {{A}}}$.
	\end{enumerate}
\end{remark}

\subsection{Grothendieck groups and derived equivalences}
The contents of this subsection are standard (except perhaps Lemma~\ref{lemma action centre}), but both definitions and notation vary a lot across the literature. Assume $\Lambda$ is an $R$-algebra, free and finitely-generated as an $R$-module, where $R$ is either a field or a complete discrete valuation ring. By $\Gro(\Lambda)$ we denote the \emph{Grothendieck group} of $\Lambda$, which is spanned by symbols $[P]$, where $P$ is a finitely-generated projective $\Lambda$-module, and $[P'']=[P]+ [P']$ whenever there is a short exact sequence $0\rightarrow P\rightarrow P'' \rightarrow P'\rightarrow 0$. This free abelian group is equipped with a bilinear form
\begin{equation}
	(-,=)_\Lambda: \ \Gro(\Lambda) \times \Gro(\Lambda) \longrightarrow \Z:\ (\,[P], [Q]\,) \mapsto \rank_R \Hom_\Lambda(P,Q).
\end{equation}
$\Gro(\Lambda)$ comes with a \emph{distinguished basis} consisting of the symbols $[P]$ for indecomposable projective modules $P$.

Similar to the above we also get a Grothendieck group $\Gro(\mathcal K^b(\proj \Lambda))$, which is isomorphic to $\Gro(\Lambda)$ by means of the Euler characteristic $\chi([C^\xbullet]) = \sum_i (-1)^i [C^i]$. In general, we have
\begin{equation}
	([S^\xbullet], [T^\xbullet])_\Lambda = \sum_{i\in\Z} (-1)^i \rank_R \Hom_{\mathcal K^b(\proj \Lambda)}(S^\xbullet, T^\xbullet[i]).
\end{equation}
See \cite[Chapter III.1]{HappenTriangCatsInRep} for a reference (but note that Happel defines $\Gro(A)$ as $\Gro(\modC A)$). It follows that if $X^\xbullet\in \mathcal D^b(\Lambda^{\opp}\otimes_R \Gamma)$ is a two-sided tilting complex (where $\Gamma$ is another $R$-algebra), then
\begin{equation}
	([P\otimes_\Lambda X^\xbullet], [Q\otimes_\Lambda X^\xbullet ])_\Gamma = ([P], [Q])_\Lambda
\end{equation}
for any two finitely-generated projective $\Lambda$-modules $P$ and $Q$. When the form $(-,=)_\Lambda$ is symmetric (e.g. for orders in semisimple algebras) this means that a derived equivalence induces an isometry between Grothendieck groups.

\begin{defi}[``Decomposition map'']\label{defi decomp}
	Let $R$ be a complete discrete valuation ring with residue field $k=R/\pi R$ and field of fractions $K$. Let $\Lambda$ be an $R$-order in a semisimple $K$-algebra~$A$, and set $\bar{\Lambda}=\Lambda/\pi\Lambda$. Then we can define a $\Z$-linear map
	\begin{equation}
		D_\Lambda:\ \Gro(\bar{\Lambda})\longrightarrow \Gro(A)
	\end{equation}
	given by the composition of the canonical isomorphism between $\Gro(\bar \Lambda)$ and $\Gro(\Lambda)$ followed by the map induced by ``$K\otimes_R -$''.
\end{defi}

These maps satisfy the identity $(-,=)_{\bar{\Lambda}}= (D_\Lambda(-),D_\Lambda(=))_A$. We could call $D_\Lambda$ a \emph{decomposition map} and its matrix a \emph{decomposition matrix}, but these terms usually refer to the adjoint of $D_\Lambda$ and its matrix. We will therefore refrain from using this terminology.

\begin{prop}\label{prop compat center decomp}
	Let $R$ be a complete discrete valuation ring with residue field $k=R/\pi R$ and field of fractions $K$. Moreover, let $\Lambda$ and $\Gamma$ be $R$-orders in semisimple $K$-algebras $A$ and $B$, respectively. Write $\bar \Lambda$ and $\bar \Gamma$ for the reductions of $\Lambda$ and $\Gamma$ modulo $\pi$. If $X^\xbullet\in \mathcal D^b(\Lambda^{\opp}\otimes_R \Gamma)$ is a two-sided tilting complex, then there is a commutative diagram
	\begin{equation}\label{eqn decomp tilt square}
		\xymatrix{
		\Gro(\bar \Lambda) \ar[r]^{\varphi_{\bar X}} \ar[d]^{D_\Lambda} & \Gro(\bar \Gamma) \ar[d]^{D_\Gamma}\\
		\Gro(A) \ar[r]^{\varphi_{KX}} & \Gro(B).
		}
	\end{equation}
	where $\varphi_{KX}$ and $\varphi_{\bar X}$ are the isometries induced by the functor ``$-\otimes_{\Lambda} X^\xbullet$''. Let $V_1,\ldots, V_n$ and $W_1,\ldots, W_n$ denote representatives for the simple $A$- and $B$-modules, and let $\eps_1,\ldots, \eps_n$ and $\eps'_1,\ldots,\eps'_n$ denote the corresponding primitive idempotents in $Z(A)$ and $Z(B)$.  The following hold:
	\begin{enumerate}
		\item There are a $\sigma \in S_n$ and signs $\tau: \ \{ 1,\ldots,n\} \longrightarrow \{\pm 1\}$ such that
		      \begin{equation}
			      \varphi_{KX}([V_i]) = \tau(i) \cdot  [W_{\sigma(i)}].
		      \end{equation}
		\item There is an isomorphism
		      \begin{equation}
			      \gamma_{KX}:\ Z(A) \longrightarrow Z(B)\quad \textrm{ such that } \gamma_{KX}(\eps_i)=\eps'_{\sigma(i)}\quad \textrm{for all $i\in \{1,\ldots,n\}$}
		      \end{equation}
		      and $\gamma_{KX}$ restricts to the isomorphism $\gamma_X:\ Z(\Lambda)\longrightarrow Z(\Gamma)$ induced by $X^\xbullet$.\qed
	\end{enumerate}
\end{prop}

In the situation of Proposition~\ref{prop compat center decomp} the isomorphism between the centres of $\Lambda$ and $\Gamma$ induced by the two-sided tilting complex $X^\xbullet$ is determined by the isomorphism between the centres of $A$ and $B$ induced by $K\otimes_R X^\xbullet$. If $A$ and $B$ are split, then this is even determined by  the induced map on Grothendieck groups. However, we will also deal with orders where $Z(\bar{\Lambda})$ is bigger than the reduction of $Z(\Lambda)$ modulo $\pi$. The following lemma helps determining the induced isomorphism between the centres of $\bar{\Lambda}$ and $\bar{\Gamma}$ in those cases.

\begin{lemma}\label{lemma action centre}
	Let $k$ be an algebraically closed field. Let $A$ be a basic finite-dimensional symmetric $k$-algebra, and let $e_1,\ldots,e_n\in A$ ($n\in\N$) denote a full system of orthogonal primitive idempotents. Let $t:\ A \longrightarrow k$ be a symmetrising form. Assume moreover that $\soc(Z(A))=\soc(A)$, and for all $1\leq i \leq n$ let $s_i \in e_i \soc(A)e_i$ be the unique element such that $t(s_i)=1$. Let  $X^\xbullet\in \TrPic_k(A)$ be a two-sided tilting complex, and define $C\in\GL_n(\Z)$ such that
	\begin{equation}
		[e_iA\otimes_A X^\xbullet] = \sum_{j=1}^n C_{i,j} \cdot [e_jA] \quad\textrm{for all $1\leq i \leq n$}
	\end{equation}
	holds in $\Gro(A)$. If $\gamma_X:\ Z(A)\longrightarrow Z(A)$ is the automorphism induced by $X^\xbullet$, then
	\begin{equation}
		\langle \gamma_X(s_i)\rangle_k = \langle \sum_{j=1}^n (C^{-1})_{j,i} \cdot s_j \rangle_k\quad \textrm{for all $1\leq i\leq n$.}
	\end{equation}
\end{lemma}
\begin{proof}
	For each projective $A$-module $P$ set $t_P=t\circ \Tr_P$, where $\Tr_P$ is the composition of the natural isomorphism $\End_A(P)\longrightarrow P\otimes_A\Hom_A(P,A)$ and the evaluation map ${P\otimes_A\Hom_A(P,A)\longrightarrow A}$. Note that $t_P$ defines a symmetrising form on $\End_A(P)$. Then define forms
	\begin{equation}
		t_{e_iA\otimes_A X^\xbullet}(\alpha) = \sum_{j\in\Z} (-1)^j \cdot t_{e_iA\otimes_AX^j}( \alpha^j)= t_{\bigoplus_je_iA\otimes_AX^j} \big(\sum_{j\in\Z} (-1)^j\cdot \alpha^j \big)
	\end{equation}
	for   $1\leq i \leq n$, where $\alpha \in \End_{\mathcal D^b(A)}(e_iA\otimes_A X^\xbullet)$. Note that this sum is finite and it is well-defined (the rightmost expression shows it vanishes on null-homotopic maps).  

	Let $\tau_L,\tau_R:\ Z(A)\longrightarrow  \End_{\mathcal D^b(A)}(e_iA\otimes_A X^\xbullet)$ be the maps sending $z\in Z(A)$ to the endomorphism induced by left and right multiplication  by $z$, respectively (we use the same name for all $i$). By definition, $\tau_L(z)$ and $\tau_R(\gamma_X(z))$ are equal (as elements of $\End_{\mathcal D^b(A)}(e_iA\otimes_A X^\xbullet)$). Now note that $t_{e_iA\otimes_AX^j}(\tau_R(s_l))$ for $1\leq i,l\leq n$  and $j\in\Z$ counts the number of times $e_lA$ occurs as a summand of $e_iA\otimes_AX^j$ and therefore $t_{e_iA\otimes_A X^\xbullet}(\tau_R(s_l)) = C_{i,l}$. We get
	\begin{equation}\label{eqn sjojsoi}
		t_{e_iA\otimes_A X^\xbullet}\left(\sum_{j=1}^n (C^{-1})_{j,l}\cdot \tau_R(s_j)\right) = (\id_{n\times n})_{i,l}
	\end{equation}
	for all $1\leq i,l\leq n$. At the same time $t_{e_iA\otimes_A X^\xbullet}(\tau_L(s_l))=0$ when $i\neq l$ since then $s_le_i=0$. If we had $t_{e_iA\otimes_A X^\xbullet}(\tau_L(s_i))=0$ as well (for some $i$), then $t_{e_iA\otimes_A X^\xbullet}(\tau_L(\soc(A)))=t_{e_iA\otimes_A X^\xbullet}(\tau_R(\soc(A)))=0$, which is impossible by equation~\eqref{eqn sjojsoi}. Therefore there is a multiple $s'_i$ of $s_i$ such that
	\begin{equation}
		t_{e_iA\otimes_A X^\xbullet}(\tau_L(s_l'))=t_{e_iA\otimes_A X^\xbullet}(\tau_R(\gamma_X(s_l')))=(\id_{n\times n})_{i,l}.
	\end{equation}
	In particular, given a linear combination  of the elements $\gamma_X(s_l')$ for $1\leq l\leq n$, applying $t_{e_iA\otimes_A X^\xbullet}\circ\tau_R$ recovers the coefficient of $\gamma_X(s_i')$. By considering equation~\eqref{eqn sjojsoi} again we see that
	\begin{equation}
		\gamma_X(s_l')=\sum_{j=1}^n (C^{-1})_{j,l} \cdot s_j
	\end{equation}
	which proves the claim.
\end{proof}

\section{Generalised weighted surface algebras}
\label{section gen weighted surf}

In this section we will introduce a class of algebras which was first studied by Erdmann and Skowro\'{n}ski  \cite{ErdmannSkonrowski1, ErdmannSkonrowski2} to get a unified description of both Brauer graph algebras and the algebras of dihedral, semi-dihedral and quaternion type classified in \cite{TameClass}. The paper \cite{ErdmannSkonrowski2} already lays out a framework that allows unified treatment of most of these algebras. We will make a slightly more general definition that also encompasses the socle deformations studied in \cite{ErdmannSkowronskiGenDihedral, ErdmannSkonrowskiGenQuatType}. In particular, the class of \emph{generalised weighted surface algebras} defined below contains all Brauer graph algebras and their socle deformations, as well as almost all algebras from \cite{TameClass} up to derived equivalence. In this section $k$ denotes an arbitrary field.

\begin{setup}
	The combinatorial data to specify a \emph{generalised weighted surface algebra} consists of the following:
	\begin{enumerate}
		\item A finite $2$-regular quiver $Q$. Such a quiver comes with an involution
		      \begin{equation}
			      \bar{\phantom{X}}:\ Q_1\longrightarrow Q_1
		      \end{equation}
		      on its set of arrows such that $\bar \alpha$ (for $\alpha\in Q_1$) is the unique arrow in $Q_1\setminus \{\alpha\}$ sharing its source with $\alpha$.
		\item A permutation
		      \begin{equation}
			      f:\ Q_1 \longrightarrow Q_1
		      \end{equation}
		      such that the source of $f(\alpha)$ is the target of $\alpha$ for all $\alpha\in Q_1$. This defines another permutation
		      \begin{equation}
			      g:\ Q_1 \longrightarrow Q_1:\ \alpha\mapsto \xoverline{f(\alpha)}.
		      \end{equation}
		      For $\alpha \in Q_1$ define
		      \begin{equation}
			      n_\alpha=|\alpha\langle g \rangle| \quad \textrm{(the cardinality of the $g$-orbit of $\alpha$)}.
		      \end{equation}
		\item Functions on $g$-orbits
		      \begin{equation}
			      \mm:\ Q_1 / \langle g\rangle \longrightarrow \Z_{> 0}
			      \quad \textrm{ and }\quad
			      \cc:\ Q_1 / \langle g\rangle \longrightarrow k^\times
		      \end{equation}
		      representing \emph{multiplicities} and certain scalars occurring in socle relations, respectively.
		\item A function on $f$-orbits
		      \begin{equation}
			      \tt:\ Q_1 / \langle f	\rangle \longrightarrow kQ:\ \alpha \mapsto t_{\alpha,0} + t_{\alpha,1} \cdot \alpha,
		      \end{equation}
		      with $t_{\alpha,0}\in\{0,1\}$ and $t_{\alpha,1}\in k$. We allow $t_{\alpha,0}=1$ only if $f^3(\alpha)=\alpha$ and $m_{\bar \alpha} n_{\bar \alpha}\geq 2$ , and we allow  $t_{\alpha,1}\neq 0$ only if $f(\alpha)=\alpha$ and $m_{\bar{\alpha}}n_{\bar{\alpha}}\geq 3$ (in which case $\alpha$ is a loop and $g(\alpha)=\bar\alpha$). We will write ``$t_\alpha\equiv 0$'' if $t_{\alpha,0}=0$ and ``$t_\alpha\equiv 1$'' if $t_{\alpha,0}=1$.
		\item A set $\mathcal Z\subseteq \left\{\alpha g(\alpha) f(g(\alpha)), \alpha f(\alpha) g(f(\alpha)) \ | \ \alpha\in Q_1\right\}$ of additional relations.
	\end{enumerate}
\end{setup}

The function $\tt$ is  not present in \cite{ErdmannSkonrowski2}, but here we want to allow algebras with mixed special biserial and quaternion relations (to include algebras of semi-dihedral type). The purpose of $\tt$ is to control the shape of the relation the monomial $\alpha f(\alpha)$ is involved in. Informally, $t_\alpha\equiv 0$ corresponds to a ``special biserial relation'', $t_\alpha\equiv 1$ to a ``quaternion type relation'', and a non-zero $t_{\alpha,1}$ corresponds to a ``socle deformation''. Note that $t_\alpha\neq 0$ is only possible if $n_{\bar{\alpha}}m_{\bar\alpha}\geq 2$. We should also point out that the conditions in the specification of $\tt$ need to be satisfied for all $\alpha\in Q_1$, it is not sufficient to verify them for a transversal of the $f$-orbits.

The set $\mathcal Z$ is also not present in \cite{ErdmannSkonrowski2}.  The idea is that working with completed path algebras as in \cite{LadkaniQuaternion} seems more elegant, and it will be necessary anyway once we consider lifts to $k\sbb{X}$ later. But \cite{LadkaniQuaternion} excludes some cases of small quivers with small multiplicities which occur as block algebras (an application we have in mind). The above set-up allows us to throw in the relations of the form $\alpha g(\alpha) f(g(\alpha))$ from \cite{ErdmannSkonrowski1, ErdmannSkonrowski2} if needed. But there is also the case of Proposition~\ref{prop semimimple q3k} where we explicitly do not want these relations, hence why we do not include them by default.

\begin{defi}[Additional notation]
	Given the data above we define
	\begin{equation}
		\begin{gathered}
			B_\alpha = \alpha g(\alpha) g^{2}(\alpha) \cdots g^{n_\alpha m_\alpha-1}(\alpha) \quad \textrm{(a circular path of length $n_\alpha m_\alpha$)},\\ A_\alpha = \alpha g(\alpha) g^{2}(\alpha) \cdots g^{n_\alpha m_\alpha-2}(\alpha)  \quad \textrm{(a path of length $n_\alpha m_\alpha-1$)},
		\end{gathered}
	\end{equation}
	for all $\alpha\in Q_1$. If $m_\alpha n_\alpha=1$ then $A_\alpha=s(\alpha)$ is the source of $\alpha$ (this will not matter in the sequel, though).
\end{defi}

When given  $Q$, $f$, $\mm$, $\cc$, $\tt$ and $\mathcal Z$ as above we will always use the notations ``$g$'', ``$\xoverline{\phantom{x}}$'', ``$n_\alpha$'', ``$B_\alpha$'' and ``$A_\alpha$'' without explicit reintroduction. We will avoid the use of the notation ``$B_\alpha$'' and ``$A_\alpha$'' where it might be ambiguous (e.g. when we are dealing with more than one multiplicity function).  In the following definition, $\widehat{kQ}$ denotes the completion of $kQ$ with respect to the ideal generated by $Q_1$, and we use a horizontal bar to indicate completions of ideals.

\begin{defi}[``Generalised weighted surface algebras'']\label{def weighted surface}
	Let $Q$, $f$, $\mm$, $\cc$, $\tt$ and $\mathcal Z$ be as above. Define ${\Lambda=\Lambda(Q,f,\mm, \cc, \tt, \mathcal Z)}$ as
	\begin{equation}
		\Lambda=\widehat{kQ}/ \xoverline{(\alpha f(\alpha)-c_{\bar \alpha}A_{\bar \alpha}t_{\alpha},\ c_{\alpha}B_{ \alpha} -c_{\bar \alpha}B_{\bar \alpha},\ \mathcal Z \ | \ \alpha\in Q_1 )}.
	\end{equation}
	We call $\Lambda$ a \emph{generalised weighted surface algebra} if
	\begin{enumerate}
		\item $\dim_k\Lambda =\sum_{\alpha\langle g \rangle \in Q_1/\langle g\rangle} m_\alpha n_\alpha^2$, and
		\item
		      for all $\alpha\in Q_1$ with $m_\alpha n_\alpha\geq 2$ the relation
		      \begin{equation}\label{eqn cond socle}
			      \left( \alpha g(\alpha) \cdots g^{en_\alpha-1}(\alpha) + \bar{\alpha} g(\bar{\alpha})\cdots g^{en_{\bar{\alpha}}-1}(\bar{\alpha}) \right)A_\alpha=0
		      \end{equation}
		      holds in $\Lambda$ for all $e\geq 2$ (if $n_\alpha=1$) or $e\geq 1$ (if $n_\alpha>1$).
	\end{enumerate}
\end{defi}

Our goal is not per se to extend the combinatorial description of \cite{ErdmannSkonrowski2, ErdmannSkownrowskiCorrigendum} (while that would be interesting, it would also be a sizeable undertaking unrelated to the ideas presented in this paper).   Therefore the preceding definition includes as axioms only the key properties we want from our algebras. Let us now give sufficient criteria for when ${\Lambda(Q,f,\mm, \cc, \tt, \mathcal Z)}$ is a generalised weighted surface algebra. For $\tt=1$ this is contained in \cite{ErdmannSkonrowski2}.

\begin{prop}\label{prop weighted surface cond}
	\begin{enumerate}
		\item The algebra $\Lambda=\Lambda(Q,f,\mm, \cc, \tt, \{ \alpha f(\alpha) g(f(\alpha))  \ | \ \alpha \in Q_1\})$ is a generalised weighted surface algebra if $n_\alpha m_\alpha \geq 3$ for all $\alpha\in Q_1$ with $t_{\bar \alpha}\equiv 1$.

		\item The algebra $\Lambda=\Lambda(Q,f,\mm, \cc, \tt, \mathfrak Z)$  (for any admissible $\mathfrak Z$) is a generalised weighted surface algebra if $n_\alpha m_\alpha \geq 4$ for all $\alpha\in Q_1$ with $t_{\bar{\alpha}}\equiv 1$.
	\end{enumerate}
	In both cases all elements $\alpha f(\alpha) g(f(\alpha))$ and $\alpha g(\alpha) f(g(\alpha))$ for $\alpha\in Q_1$ become zero in $\Lambda$.
\end{prop}
\begin{proof}
	\begin{enumerate}
		\item Write $t_\alpha=t_{\alpha,0}+t_{\alpha,1}\cdot\alpha$ for $\alpha\in Q_1$.  We have
		      \begin{equation}\label{eqn yuayuyu}
			      \alpha g(\alpha) f(g(\alpha)) \equiv c_{\overline{g(\alpha)}}t_{g(\alpha),0}\cdot \alpha A_{\overline{g(\alpha)}}
			      + c_{\overline{g(\alpha)}}t_{g(\alpha),1}\cdot \alpha B_{\overline{g(\alpha)}} \equiv c_{\overline{g(\alpha)}}t_{g(\alpha),0}\cdot \alpha A_{\overline{g(\alpha)}}
		      \end{equation}
		      modulo the relations of $\Lambda$. This uses the fact that if $t_{\alpha,1}\neq 0$ then $\alpha=f(\alpha)$ and therefore $g(\alpha)=\bar \alpha$, which entails $A_{\bar{\alpha}}\alpha=B_{\bar{\alpha}}$. If $t_{g(\alpha),0}=0$, then the right hand side of \eqref{eqn yuayuyu} is zero. If $t_{g(\alpha),0}=1$ then our assumption implies $m_{\overline{g(\alpha)}}n_{\overline{g(\alpha)}}\geq 3$, which means that $A_{\overline{g(\alpha)}}=A_{f(\alpha)}$ contains the initial subword $f(\alpha)g(f(\alpha))$. Hence the right hand side of \eqref{eqn yuayuyu} is zero modulo the relations of $\Lambda$ in this case as well.

		      It follows that any path of length at least three  containing a subword of the form $\alpha f(\alpha)$ either becomes zero in $\Lambda$, or is of the form $\alpha f(\alpha) f^2(\alpha)$, which can be rewritten as $c_{\bar\alpha}t_{\alpha,0}B_{\bar{\alpha}}$. Using the relation $c_\alpha B_\alpha = c_{\bar \alpha}B_{\bar \alpha}$ it also follows that each $B_\alpha$ lies in the socle of $\Lambda$, which implies the condition from equation~\eqref{eqn cond socle} we need to show. Hence $\Lambda$ is spanned by the initial subwords of the elements $B_\alpha$ for $\alpha\in Q_1$. We should also note that we have in fact shown that $\widehat{kQ}/ (\alpha f(\alpha)-c_{\bar \alpha}A_{\bar \alpha}t_{\alpha},\ c_{\alpha}B_{ \alpha} -c_{\bar \alpha}B_{\bar \alpha},\ \alpha f(\alpha) g(f(\alpha)) \ | \ \alpha\in Q_1)$ is finite-dimensional without taking the completion, which shows that the ideal we are modding out here is in fact already complete.

		      Note that the assumption $m_\alpha n_\alpha\geq 3$ whenever $t_{\bar{\alpha}}\equiv 1$ implies  that all paths involved in the defining relations of $\Lambda$ have length at least two. A linear combination of paths involving paths of length less than two can therefore not become zero modulo the relations of $\Lambda$. Any non-trivial linear dependence between initial subwords of one or more of the $B_{\alpha}$  involving only paths of length $\geq 2$ implies that one of the $B_{\alpha}$ is zero in $\Lambda$ (by multiplying such a linear dependence by another monomial, turning the shortest occurring initial subword of some $B_\alpha$ into $B_\alpha$ itself whilst annihilating all other terms). Hence we just need to show that all of the $B_{\alpha}$ are non-zero in $\Lambda$, since it will then follow that the proper initial subwords of the $B_\alpha$ together with one element of the pair $\{B_\alpha, B_{\bar \alpha}\}$ for each $\alpha$ form a basis of $\Lambda$, which counting reveals to have size $\sum_{\alpha\langle g \rangle} m_\alpha n_\alpha^2$.

		      Let us now show that, in fact, all $B_\alpha$ are non-zero.  To this end, let us consider the ideal in $A=kQ/(\alpha g(\alpha) f(g(\alpha)),\ \alpha f(\alpha) g(f(\alpha)), B_\alpha\alpha, B_\alpha\bar{\alpha}, \bar{\alpha}B_\alpha \ | \ \alpha\in Q_1)$ (rather than $\widehat{kQ}$) given by $I=(\alpha f(\alpha)-c_{\bar \alpha}A_{\bar \alpha}t_{\alpha},\ c_{\alpha}B_{ \alpha} -c_{\bar \alpha}B_{\bar \alpha} \ | \ \alpha\in Q_1)_A$. Since the defining relations of $A$ are monomial, the algebra $A$ has a $k$-basis consisting of all paths not containing any of the relations as a subpath. By multiplying by all possible monomials from the left and from the right we can  then write down  $k$-vector space  generators of $I$:
		      \begin{equation}
			      \begin{gathered}
				      \alpha f(\alpha)-c_{\bar \alpha}t_{\alpha,0}A_{\bar \alpha}-c_{\bar \alpha}t_{\alpha,1}B_{\bar \alpha}, \ \alpha f(\alpha)f^2(\alpha)-c_{\alpha}t_{\alpha,0}B_{\alpha}, \\
				      \alpha f(\alpha)f^2(\alpha)-c_{\bar \alpha}t_{\alpha,0}B_{\bar \alpha}, \alpha f(\alpha) \cdots f^i(\alpha),
				      \ c_\alpha B_\alpha-c_{\bar \alpha}B_{\bar \alpha},
			      \end{gathered}
		      \end{equation}
		      running over all $i\geq 3$ and $\alpha\in Q_1$. Moreover, all paths involved in these elements are linearly independent in $A$. It follows that  ${I\cap \langle B_\alpha \ | \ \alpha\in Q_1\rangle_k}$ is equal to ${\langle c_\alpha B_\alpha-c_{\bar{\alpha}}B_{\bar{\alpha}} \ | \ \alpha\in Q_1\rangle_k}$, which contains none of the $B_\alpha$. Therefore each $B_\alpha$ is non-zero in $A/I$, which implies that it is non-zero in $\Lambda$.

		\item We can show that in this case the paths $\alpha f(\alpha)g(f(\alpha))$ become zero in $\Lambda$ even though we did not explicitly add them as relations. Once that is done, we can apply the first part of this proposition.

		      Let us assume that we have a path of length $n\geq 3$ containing the subword $\alpha f(\alpha)g(f(\alpha))$, and let us assume $t_{ \alpha} \neq 0$ and therefore $f^3(\alpha)=\alpha$ (otherwise the path is zero in $\Lambda$ anyway). Then this path is equivalent to the sum of (if $t_\alpha\equiv 1$) a multiple of  a path of length $n+n_{\bar \alpha} m_{\bar \alpha} -3>n$  with subword $A_{\bar \alpha} g(f(\alpha))$, and (if $t_{\alpha,1}\neq 0$) a multiple of a path of length $n+n_{\bar \alpha} m_{\bar \alpha} -2>n$ with subword $B_{\bar{\alpha}} g(f(\alpha))$. The terminal subword of length three of  $A_{\bar \alpha} g(f(\alpha))$ is $g^{n_{\bar \alpha}m_{\bar \alpha}-3}(\bar \alpha)g^{n_{\bar \alpha}m_{\bar \alpha}-2}(\bar \alpha) f(g^{n_{\bar \alpha}m_{\bar \alpha}-2}(\bar \alpha))$, since $g(f(\alpha))=\xoverline{f^2(\alpha)}$, and $g(f^{2}(\alpha))\neq f(f^2(\alpha))=\alpha$, which implies $f^2(\alpha) = g^{m_{\bar\alpha} n_{\bar \alpha}-1}(\bar \alpha)$ and therefore $\xoverline{f^2(\alpha)} = \xoverline{g^{m_{\bar \alpha} n_{\bar \alpha}-1}(\bar \alpha)}= f(g^{{m_{\bar \alpha} n_{\bar \alpha}-2}}(\bar \alpha ))$. The element $B_{\bar \alpha} g(f(\alpha))$ is equal to $B_{\bar\alpha}\bar\alpha$, assuming $f(\alpha)=\alpha$. Modulo the relations of $\Lambda$ this is a multiple of $B_\alpha\bar{\alpha}$, which has the same length and terminates in  $g^{n_{\alpha}m_{\alpha}-2}(\alpha)g^{n_{\alpha}m_{\alpha}-1}(\alpha) f(g^{n_{\alpha}m_{ \alpha}-1}(\alpha))$.

		      Similarly, if we have a path of length ${n\geq 3}$ containing the subword $\alpha g(\alpha)f(g(\alpha))$ (with ${t_{{g(\alpha)}}\neq 0}$), then this path can be rewritten as a sum of multiples of paths of greater length containing a subword of the form $\beta f(\beta) g(f(\beta))$.  It thus follows that any path of length~${\geq 3}$ containing a subword of the form $\alpha f(\alpha)g(f(\alpha))$ can be rewritten as a path of arbitrarily large length, and such paths converge to zero in the completed path algebra $\widehat{kQ}$. Hence the original path was zero in~$\Lambda$.  \qedhere
	\end{enumerate}
\end{proof}

\begin{prop}[Alternative presentations]\label{prop alternative pres}
	Assume $\Lambda=\Lambda(Q,f,\mm, \cc, \tt, \mathcal Z)$ is a generalised weighted surface algebra in the sense of Definition~\ref{def weighted surface}, and assume $m_{\alpha}n_{\alpha} \geq 4$ for all $\alpha$ with $t_{\bar \alpha}\equiv 1$. For each $\alpha\in Q_1$ let $C_\alpha$ and $C_\alpha'$ be linear combinations of paths each containing a subword of the form $\beta g(\beta) f(g(\beta))$ or $\beta f(\beta) g(f(\beta))$, and $C_{\alpha}''$ a linear combination of paths each containing a subword $\beta f(\beta)$ with $t_\beta=0$. If $\alpha$ and $g(\alpha)$ lie in the same $g$-orbit then also assume that all paths involved in $C_\alpha'$ and $C_\alpha''$ have length equal to or bigger than that of $B_\alpha$.  Then
	\begin{equation}\label{eqn alt pres}
		\Lambda = \widehat{kQ}/ \xoverline{(\alpha f(\alpha)-c_{\bar \alpha}A_{\bar \alpha}t_{\alpha} + C_{\alpha},\ c_{\alpha}B_{ \alpha} -c_{\bar \alpha}B_{\bar \alpha} + C_{\alpha}'+C_{\alpha}'',\ \mathcal Z \ |\ \alpha \in Q_1)}.
	\end{equation}
\end{prop}
\begin{proof}
	By the proof of the second part of Proposition~\ref{prop weighted surface cond} we know that the $C_\alpha$, $C_\alpha'$ and $C_\alpha''$ become zero in $\Lambda$. Hence, the ideal being factored out on the right hand side of \eqref{eqn alt pres} is contained in the defining ideal of $\Lambda$. We can also eliminate the $C_\alpha''$ from the presentation, by replacing the subwords $\beta f(\beta)$ by $-C_{\beta}$. Now we can use the same argument as in the proof of the second part of Proposition~\ref{prop weighted surface cond} (which is where we need the length condition) to show that the ideal on the right hand side of \eqref{eqn alt pres} contains all elements of the form $\beta g(\beta) f(g(\beta))$ and $\beta f(\beta) g(f(\beta))$, which shows that it contains all $C_\alpha$'s and $C_{\alpha}'$'s. Therefore it is also included in the defining ideal of $\Lambda$, proving equality.
\end{proof}

\begin{prop}\label{prop semimimple q3k}
	Assume $\operatorname{char}(k)=2$. Let $Q$ be the quiver
	\begin{equation}
		\xygraph{
		!{<0cm,0cm>;<1.5cm,0cm>:<0cm,1cm>::}
		!{(0,0) }*+{\bullet_{1}}="a"
		!{(2,0) }*+{\bullet_{2}}="b"
		!{(1,-2) }*+{\bullet_{3}}="c"
		"b" :@/^/^{\beta_2} "a"
		"a" :@/^/^{\alpha_1} "b"
		"b" :@/^/^{\alpha_2} "c"
		"c" :@/^/^{\beta_3} "b"
		"c" :@/^/^{\alpha_3} "a"
		"a" :@/^/^{\beta_1} "c"
		}
	\end{equation}
	and let $f$ be the permutation $(\alpha_1\alpha_2\alpha_3)(\beta_1\beta_2\beta_3)$ (i.e. $n_\gamma=2$ for all $\gamma\in Q_1$). Set $t_\gamma= 1$, $m_\gamma = 1$ and $c_\gamma=c$ for some fixed $c\in k^\times$, for all $\gamma \in Q_1$. Then
	\begin{equation}
		\Lambda = \Lambda(Q,f,\mm, \cc, \tt, \emptyset)
	\end{equation}
	is a generalised weighted surface algebra.
\end{prop}
\begin{proof}
	This is easy to verify. First note that for all $e\geq 1$ and all $\gamma\in Q_1$
	\begin{equation}
		\gamma g(\gamma) \cdots g^{en_\gamma-1}(\gamma) + \bar{\gamma} g(\bar{\gamma})\cdots g^{en_{\bar{\gamma}}-1}(\bar{\gamma}) = B_\gamma^e+B_{\bar{\gamma}}^e.
	\end{equation}
	Now $c_\gamma^e=c_{\bar\gamma}^e=-c_{\bar\gamma}^e$ gives that the above is zero in $\Lambda$, implying the condition in equation~\eqref{eqn cond socle}.

	We claim that $\Lambda$ is a split-semisimple $k$-algebra, and more specifically
	\begin{equation}\label{eqn lambda semisimple}
		\Lambda \cong k\oplus k \oplus k\oplus M_3(k),
	\end{equation}
	which, once shown, will imply the dimension condition immediately.

	Let $E_1,E_2,E_3$ denote the respective unit elements in the first three summands on the right hand side of \eqref{eqn lambda semisimple}, and let $E_4(i,j)$ denote the $(i,j)$-matrix unit in the rightmost summand in \eqref{eqn lambda semisimple}. One can easily verify that
	\begin{equation}\label{eqn map semisimple}
		\Lambda \longrightarrow  k\oplus k \oplus k\oplus M_3(k):\ e_i \mapsto E_i+E_4(i,i),\ \alpha_i \mapsto c\cdot E_{4}(i,\sigma(i)),\ \beta_i \mapsto c\cdot E_4(i,\sigma^{-1}(i))
	\end{equation}
	defines a homomorphism (where $\sigma=(1,2,3)\in S_3$), by checking that the images satisfy the defining relations of $\Lambda$. It follows that $\dim_k\Lambda\geq 12$. On the other hand, $\Lambda$ is spanned by paths along $g$-orbits of length at most two, since any path involving a $\gamma f(\gamma)$ for $\gamma\in Q_1$, and likewise any path of length greater than two, can be rewritten as a multiple of a shorter path. Hence $\dim_k \Lambda$ is at most $12$, which implies that the map in~\eqref{eqn map semisimple} is an isomorphism.
\end{proof}

\section{Lifts for twisted Brauer graph algebras}

In this section we will have a look at a class of algebras closely related to Brauer graph algebras, but with slightly more well-behaved lifts to $k\sbb{X}$-orders. These lifts were first studied in \cite{GnedinRibbonOrders}. Note that Brauer graph algebras are just generalised weighted surface algebras in the sense of Definition~\ref{def weighted surface} where $\cc$ is constant equal to one and  $\tt$ is constant equal to zero. By $k$ we again denote an arbitrary field.
\begin{defi}[{see \cite[Definiton 3.7]{GnedinRibbonOrders}}]\label{def twisted bga}
	Let $Q$ and $f$ be as in \S\ref{section gen weighted surf}, and assume we are also given a map ${\mm:\ Q_1/\langle g \rangle \longrightarrow \Z_{>0}}$. We define the \emph{twisted Brauer graph algebra}
	\begin{equation}
		\LambdaTw(Q,f,\mm) = kQ/(\alpha f(\alpha), \ B_{\alpha}+B_{\alpha'}  \ | \ \alpha \in Q_1),
	\end{equation}
	where $B_\alpha=\alpha g(\alpha)\cdots g^{n_\alpha m_\alpha-1}(\alpha)$ for all $\alpha\in Q_1$, as before.
\end{defi}

Let us quickly explain how to associate a Brauer graph to the pair $(Q,f)$. The Brauer graph encodes the exact same information as $(Q,f)$, and we will often switch back and forth between the two (with a preference for Brauer graphs in the statement of results, since that is the standard way of parametrising Brauer graph algebras).

\begin{defi}\label{def brauer graph}
	The \emph{Brauer graph} associated with $(Q,f)$ is the (undirected) graph whose vertices are indexed by the $g$-orbits on $Q_1$ and whose edges are indexed by the vertices in $Q_0$. The edge corresponding to $e\in Q_0$ connects the vertices corresponding to $\alpha \langle g\rangle$ and $\bar\alpha \langle g\rangle$, where $\alpha$ and $\bar{\alpha}$ are the two arrows whose source is $e$.  The map $g$ induces a cyclic order on the half-edges incident to a vertex. More specifically, half-edges can be encoded as pairs $(e,\alpha)\in Q_0\times Q_1$ where $e$ is the source of $\alpha$. If $(e,\alpha)$ corresponds to a half-edge, then the half-edge  following $(e,\alpha)$ in the cyclic order around $\alpha\langle g\rangle$ is defined to be $(s(g(\alpha)), g(\alpha))$, where $s(g(\alpha))$ denotes the source of $g(\alpha)$.
\end{defi}

There are some circumstances under which twisted Brauer graph algebras are isomorphic to their untwisted counterparts. This was studied in \cite{GnedinRibbonOrders}, as were the lifts of these algebras given in Proposition~\ref{prop biserial order} below. In the present paper we want to deal with the class of algebras defined in \S\ref{section gen weighted surf}, which properly contains Brauer graph algebras but not all twisted Brauer graph algebras. However, if an algebra we are interested in happens to be a twisted Brauer graph algebra, then the lift provided in Proposition~\ref{prop biserial order} below will have nicer properties than the lifts constructed in Proposition~\ref{prop lift weighted surface}, and we can prove slightly stronger results using them.

\begin{prop}[see {\cite[Proposition 3.16]{GnedinRibbonOrders}}]\label{prop tw bga eq bga}
	If $\operatorname{char}(k)=2$ or if $k=\bar k$ and the Brauer graph of $(Q,f)$ is bipartite, then
	\begin{equation}
		\LambdaTw(Q,f,\mm)\cong \Lambda(Q,f,\mm, \cc, \tt, \emptyset)
	\end{equation}
	where $c_\alpha=1$ and $t_\alpha=0$ for all $\alpha\in Q_1$.
\end{prop}

\begin{prop}[{see ``Ribbon Graph Orders'' in \cite{GnedinRibbonOrders}}]\label{prop biserial order}
	Let $Q$ and $f$ be as in \S\ref{section gen weighted surf}, and assume we are given a map ${\mm:\ Q_1/\langle g \rangle \longrightarrow \Z_{>0}}$. Define the $k$-algebra
	\begin{equation}
		\GammaTw(Q,f)= \widehat{kQ}/\xoverline{(\alpha f(\alpha) \ | \ \alpha \in Q_1)}
	\end{equation}
	and the element
	\begin{equation}
		z=\sum_{\alpha\in Q_1} \alpha g(\alpha)\cdots g^{m_\alpha n_\alpha-1}(\alpha).
	\end{equation}
	The $k$-algebra $\GammaTw(Q,f)$ becomes a $k\sbb{X}$-order by letting $X$ acts as $z$ (and extending this to the completion). We will denote this $k\sbb{X}$-order by
	\begin{equation}
		\Gamma = \GammaTw(Q,f, \mm).
	\end{equation}
	The following hold:
	\begin{enumerate}
		\item $	\GammaTw(Q,f, \mm)/X	\GammaTw(Q,f, \mm) \cong 	\LambdaTw(Q,f, \mm)$.
		\item The $k\sqq{X}$-algebra $k\sqq{X}\otimes_{k\sbb{X}}\Gamma$ is semisimple.
		\item The simple $k\sqq{X}\otimes_{k\sbb{X}}\Gamma$-modules are labelled by the $g$-orbits  of arrows $\alpha \langle g \rangle \in Q_1/\langle g \rangle$.
		\item 	The endomorphism algebra of the simple $k\sqq{X}\otimes_{k\sbb{X}}\Gamma$-module labelled by $\alpha \langle g\rangle$ is~$k\sqq{X^{1/m_\alpha}}$.
		\item
		      If $D_{e, \alpha\langle g\rangle}\in \{0,1,2\}$ for $e\in Q_0$ and $\alpha \in Q_1$ is chosen such that $e$ is the source of exactly $D_{e, \alpha\langle g\rangle}$ arrows in the orbit $\alpha\langle g\rangle$, then
		      \begin{equation}
			      D_\Gamma([e\cdot\Gamma/X\Gamma]) = \sum_{\alpha\langle g \rangle \in Q_1/\langle g\rangle} D_{e, \alpha\langle g\rangle}\cdot [V_{\alpha\langle g \rangle}]
		      \end{equation}
		      for all $e\in Q_0$, where $V_{\alpha\langle g\rangle}$ denotes the simple $k\sqq{X}\otimes_{k\sbb{X}}\Gamma$-module labelled by $\alpha\langle g\rangle$.
	\end{enumerate}
\end{prop}
\begin{proof}
	$kQ/(\alpha f(\alpha) \ | \ \alpha \in Q_1)$ has a basis consisting of the vertices of $Q$ and paths along $g$-orbits, i.e. $\alpha g(\alpha) \cdots g^i(\alpha)$ for $i\geq 0$. This follows from the fact that the ideal $(\alpha f(\alpha) \ | \ \alpha \in Q_1)$  is spanned by the paths which have a subword of the form $\alpha f(\alpha)$ for $\alpha \in Q_1$.

	As a $k[z]$-module, $kQ/(\alpha f(\alpha) \ | \ \alpha \in Q_1)$ is clearly spanned by initial subwords (including length zero) of $\alpha g(\alpha)\cdots g^{m_\alpha n_\alpha-1}(\alpha)$ for the various $\alpha \in Q_1$. In particular, it is finitely generated. Moreover, if we fix a subset $S_1\subset Q_1$ such that for each $\alpha\in Q_1$ exactly one of the two arrows $\alpha$ and $\bar\alpha$ is contained in $S_1$, then the proper initial subwords (including length zero) of $\alpha g(\alpha)\cdots g^{n_\alpha m_\alpha-1}(\alpha)$ for the various $\alpha\in Q_1$ together with the elements $\alpha g(\alpha)\cdots g^{n_\alpha m_\alpha-1}(\alpha)$ for $\alpha\in S_1$ form a free generating set for the $k[z]$-module $kQ/(\alpha f(\alpha) \ | \ \alpha \in Q_1)$, showing that it is a $k[z]$-order. This carries over to the completion $\Gamma$, turning it into a free $k\sbb{z}$-module (or $k\sbb{X}$-module). The fact that $\GammaTw(Q,f, \mm)$ reduces to $\LambdaTw(Q,f, \mm)$ is clear by definition.

	To show that $B=k\sqq{X}\otimes_{k\sbb{X}}\Gamma$ is semisimple, let us first write down (non-unital) embeddings of  $k\sqq{X^{1/m_\alpha}}$ into this algebra, which will turn out to correspond to the simple components of the centre. The image of $1$ under such an embedding gives a central idempotent $\eps_\alpha$ in $B$, and we will write down a $k\sqq{X^{1/m_\alpha}}$-basis of  $\eps_\alpha B$ whose elements multiply like matrix units, proving semisimplicity and shape of the decomposition matrix in one go.

	Now, to an orbit $\alpha\langle g \rangle$ associate the (non-unital) embedding
	\begin{equation}
		\iota_{\alpha}:\ k\sbb{X^{1/m_\alpha}} \longrightarrow Z(B) :\  \sum_{i=0}^\infty c_i (X^{1/m_\alpha})^i \mapsto \frac{1}{X} 	\sum_{\beta\in \alpha\langle g \rangle} \sum_{i=0}^{\infty} c_i \beta g(\beta) \cdots g^{n_\alpha (m_\alpha+i) -1}(\beta).
	\end{equation}
	Under this map, $1$ gets mapped to ${\eps_\alpha=\frac{1}{X}	\sum_{\beta\in \alpha\langle g \rangle} \beta g(\beta) \cdots g^{n_\alpha m_\alpha -1}(\beta)}$, and  this turns $\eps_\alpha\Gamma$ into a $k\sbb{X^{1/m_\alpha}}$-algebra.  One easily checks that $\eps_{\alpha}$ and $\eps_\beta$ are orthogonal idempotents when $\alpha$ and $\beta$ lie in distinct $g$-orbits. One also checks that $k\sqq{X^{1/m_\alpha}}\otimes_{k\sbb{X^{1/m_\alpha}}}\eps_\alpha \Gamma$ has $k\sqq{X^{1/m_\alpha}}$-basis
	\begin{equation}
		e(\alpha)_{i,j} = \frac{1}{X^{1/m_\alpha}}g^{i-1}(\alpha)g^{i}(\alpha) \cdots g^{n_\alpha + j - 2}(\alpha)\quad\textrm{for $1\leq i,j \leq n_{\alpha}$}.
	\end{equation}
	One verifies that $e(\alpha)_{i,j}e(\alpha)_{j,l}=e(\alpha)_{i,l}$ for any $1\leq i,j,l\leq n_\alpha$. That is, the  $e(\alpha)_{i,j}$ multiply like matrix units, thus proving that $\eps_\alpha B\cong M_{n_\alpha}(k\sqq{X^{1/m_\alpha}})$. Moreover, the image in $\eps_\alpha\Gamma$ of the idempotent in $\Gamma$ attached to a vertex $v\in Q_0$ is given by $\sum_{\alpha} e(\alpha)_{1,1}$, where $\alpha$ runs over all arrows whose source is $v$ (note that  $e(\alpha)_{i,i}=e(g^{i-1}(\alpha))_{1,1}$, so if two $\alpha$'s in that sum lie in the same $g$-orbit, then the corresponding idempotents lie in the same set of matrix units). This shows that the map $D_\Gamma$ is as claimed, which completes the proof.
\end{proof}
The above can be used to describe the order $\Gamma$ in slightly greater detail. For example, if for each $\alpha\in Q_1$ the arrows $\alpha$ and $\bar\alpha$ lie in different $g$-orbits, then the $k\sbb{X^{1/m_\alpha}}$-orders $\eps_\alpha \Gamma$ are hereditary with basis $e(\alpha)_{i,j}$ for $i\leq j$ and $X^{1/m_\alpha} e(\alpha)_{i,j}$ for $i>j$. This is perfectly analogous to the lifts of Brauer tree algebras to $p$-adic discrete valuation rings that occur in blocks (also called ``Green orders'' \cite{RoggenkampGreenOrders}).

\section{Lifts for generalised weighted surface algebras}

In this section we will lift arbitrary generalised weighted surface algebras with sufficiently large multiplicities to $k\sbb{X}$-orders.  Such a lift will be a pullback of a lift as in Proposition~\ref{prop biserial order} and an order in a generalised weighted surface algebra over $k\sqq{X}$ (which is typically non-semisimple). We will see that these lifts will usually be independent of the multiplicities as rings, a fact which we then use to construct a bijection of silting complexes in Theorem~\ref{thm main general}. The latter is one of the main results of the present article. In this section, $k$ again denotes an arbitrary field.

\begin{prop}\label{prop pullback}
	Let $Q$ be a finite quiver, and define  $A_0=\widehat{kQ}$ as well as $A=\widehat {k[X]Q}$, the completion of the path algebra $k[X]Q$ with respect to the ideal $(X, Q_1)_{k[X]Q}$. Consider maps $r_{i}:\ (X,Q_1)_A \longrightarrow A$ given by
	\begin{equation}
		r_i(T)=r_{i,0}+\sum_{j=1}^{\infty} T^j \cdot r_{i,j}\qquad\textrm{ for $1\leq i \leq m$},
	\end{equation}
	where the $r_{i,j}$ are elements of $A_0$ and $m\in \N$. For a given $T\in (X,Q_1)_A$ let $\mathfrak R(T)=\left\{r_1(T),\ldots,r_m(T)\right\}$. 

	Let $\mathfrak I$ be an ideal in $A_0$ and let $z\in (Q_1)_{A_0}$ be chosen such that $z$ is central in $A_0/\xoverline{(\mathfrak R(0))_{A_0}}$.  Define
	\begin{equation}
		B = A_0/\xoverline{(\mathfrak R(z), \mathfrak I \cdot z)_{A_0}},
	\end{equation}
	and assume that this is a finite-dimensional algebra. Assume moreover that all of the following hold:
	\begin{enumerate}
		\item $\dim_k B=\dim_k A_0/\xoverline{(\mathfrak R(0), \mathfrak I)_{A_0}} + \dim_k A_0/\xoverline{(\mathfrak R(0), z)_{A_0}}$.
		\item
		      For each $1\leq i \leq m$ there is a $\widehat z_i\in (Q_1)_{A_0}$ such that
		      \begin{equation}\label{eqn zi hat}
			      \widehat{z}_i^j \cdot r_{i,j}\in \xoverline{(\mathfrak R(X), \mathcal I)_A}\quad \textrm{and} \quad
			      \widehat{z}_i^j \cdot r_{i,j} + \xoverline{(\mathfrak R (0))_{A_0}} = z^j \cdot r_{i,j} + \xoverline{(\mathfrak R (0))_{A_0}}
		      \end{equation} for all $1\leq i\leq m$ and  $j\geq 1$, and also
		      \begin{equation} \label{eqn zi hat ideal}
			      \xoverline{(r_1(\widehat z_1),\ldots, r_m(\widehat z_m), \mathfrak I\cdot z, X)_A} = \xoverline{(\mathfrak R(z), \mathfrak I\cdot z, X)_A}.
		      \end{equation}
		\item $A/\xoverline{(\mathcal R(0), X+z)_A}$ and $A/\xoverline{(\mathfrak R(X), \mathfrak I)_A}$ are free and finitely generated as $k\sbb{X}$-modules.
	\end{enumerate}
	Let $\Gamma$ be the pullback
	\begin{equation}\label{eqn pullback}
		\xymatrix{
			A/\xoverline{(\mathfrak R(X), \mathfrak I)_A} \ar@{->>}[r] & A/\xoverline{(\mathfrak R(0), \mathfrak I, X+z)_A}\\ \Gamma	 \ar[u]\ar[r] & A/\xoverline{(\mathfrak R(0), X+z)_A}, \ar@{->>}[u]
		}
	\end{equation}
	where the maps into the top right term are the natural surjections. Then $\Gamma$ is a $k\sbb{X}$-order such that~$\Gamma/X\Gamma \cong B$.
\end{prop}
\begin{proof}
	Our second assumption implies in particular that $z^j\cdot r_{i,j}\in \xoverline{(\mathfrak R(0), \mathcal I)_{A_0}}$, since substituting $X=0$ in the first part of the condition yields  $\widehat z_i^j\cdot r_{i,j}\in \xoverline{(\mathfrak R(0), \mathcal I)_{A_0}}$. This shows that $\mathfrak R(z),\mathfrak R(-z)\subseteq \xoverline{(\mathfrak R(0), \mathfrak I)_{A}}$. Hence, the topmost horizontal arrow in the pullback diagram is well-defined, and the top right term, $A/\xoverline{(\mathfrak R(0), \mathfrak I, X+z)_A}$, is isomorphic to a quotient of $B$ as a $k$-algebra, and therefore is finite-dimensional.

	The $k\sbb{X}$-algebras $A/\xoverline{(\mathfrak R(X), \mathfrak I)_A}$ and $A/\xoverline{(\mathfrak R(0), X+z)_A}$ are $k\sbb{X}$-orders by assumption, and their reductions modulo $X$ are $A_0/\xoverline{(\mathfrak{R}(0), \mathfrak I)_{A_0}}$ and $A_0/\xoverline{(\mathfrak R(0),z)_{A_0}}$, respectively. From our assumptions it follows that the dimensions of these quotients sum up to $ \dim_k(B)$, from which it follows that $\Gamma$ has $k\sbb{X}$-rank $\dim_k(B)$. It therefore suffices to show that  $\Gamma/X\Gamma$ is a quotient of $B$. Note that $\xoverline{(\mathfrak R(X), \mathfrak I)_A}+\xoverline{(\mathfrak R(0), X+z)_A}=\xoverline{(\mathfrak R(0), \mathfrak I, X+z)_A}$.  This shows that ${\Gamma = A/\xoverline{(\mathfrak R(X), \mathfrak I)_A}\cap\xoverline{(\mathfrak R(0), X+z)_A}}$. Hence we need to show that
	\begin{equation}\label{eqn intersect}
		\xoverline{(\mathfrak R(X), \mathfrak I)_A}\cap\xoverline{(\mathfrak R(0), X+z)_A} + (X)_A \supseteq (\mathfrak R(z), \mathfrak I \cdot z)_A + (X)_A,
	\end{equation}
	which will imply the analogous inclusion for the completion as well (note that $\xoverline{(X)_A}=(X)_A$, so the left hand side is actually complete). We have $(\mathfrak I \cdot z)_A +(X)_A= \mathfrak I \cdot (X+z)_A+(X)_A$, and $\mathfrak I\cdot (X+z)_A$ is clearly contained in the intersection on the left hand side.	Moreover, for all $1\leq i \leq m$ we have
	\begin{equation}
		\begin{array}{rll}
			r_i(X+\widehat z_i) & \equiv r_i(\widehat z_i) & \mod (X)_A,                                      \\
			r_i(X+\widehat z_i) & \equiv  r(X+z)\equiv 0   & \mod \xoverline{(\mathfrak R(0), X+z)_A},        \\
			r_i(X+\widehat z_i) & \equiv  r(X) \equiv 0    & \mod \xoverline{(\mathfrak R(X),\mathfrak I)_A}.
		\end{array}
	\end{equation}
	In the last two lines we are using the assumptions of equation~\eqref{eqn zi hat}.

	In conclusion, each $r_i(X+\widehat z_i)$ lies in the intersection on the left hand side of \eqref{eqn intersect}, and at the same time reduces to $r(\widehat z_i)$ modulo $X$. Now use the assumption of equation~\eqref{eqn zi hat ideal} to see that the left hand side of \eqref{eqn intersect}  actually contains the right hand side of \eqref{eqn intersect}. We have thus shown that $\Gamma/X\Gamma$ is a quotient of $B$, and therefore, by virtue of dimensions, is isomorphic to $B$.
\end{proof}

While perhaps not obvious at first glance, the main point of Proposition~\ref{prop pullback} is that the pullback diagram \eqref{eqn pullback} is often completely independent of  the element $z$ if considered as a pullback diagram of rings. The bottom right term is isomorphic to $A_0/\xoverline{(\mathfrak R(0))_{A_0}}$ as a ring, and the top right term is isomorphic to $A_0/\xoverline{(\mathfrak R(0), \mathfrak I)_{A_0}}$. Neither of these depend on $z$, nor does the map between them. There is a hidden dependence on $z$ in the topmost horizontal arrow, but we will see below that this dependence disappears in most cases of interest. The point is that most admissible choices for $z$ give rise to the same ring $\Gamma$ lifting $B=B(z)$, it is only the $k\sbb{X}$-algebra structure on this ring which varies.

\begin{prop}\label{prop lift weighted surface}
	Let $Q$, $f$, $\mm$, $\cc$, $\tt$ and $\mathfrak Z$ be as in \S\ref{section gen weighted surf}. Assume that
	\begin{equation}
		\Lambda(Q,f,\mm, X\cdot \cc, \tt, \mathfrak Z) \quad\textrm{defined over $k\sqq{X}$}
	\end{equation}
	is a generalised weighted surface algebra in the sense of Definition~\ref{def weighted surface} (over $k\sqq{X}$ rather~than~$k$). On top of that let
	\begin{equation}
		\mm':\ Q_1/\langle g \rangle \longrightarrow \Z_{>0}
	\end{equation}
	be a function such that  $m'_{\alpha}n_{\alpha}\geq 2$ for all $\alpha\in Q_1$ with $t_{\bar{\alpha}}\equiv 1$  (which implies  $(m_\alpha+m'_{\alpha})n_\alpha \geq 4$ in those cases).

	Then there exists a $k\sbb{X}$-order
	\begin{equation}
		\Gamma=\Gamma(Q,f,\mm, \cc, \tt, \mathfrak Z; \mm')
	\end{equation}
	such that
	\begin{equation}
		\Gamma/X\Gamma \cong \Lambda(Q, f, \mm + \mm', \cc, \tt,\mathcal Z)\quad \textrm{defined over $k$}.
	\end{equation}
\end{prop}
\begin{proof}
	Let us retain the notation $A=\widehat{k[X] Q}$ and $A_0=\widehat{kQ}$. Define the following: 
	\begin{equation}
		\mathfrak R(T) = \{\underbrace{\alpha f(\alpha) - T\cdot c_{\bar \alpha} A_{\bar \alpha}t_{\alpha}}_{=r_\alpha(X)}\ | \ \alpha \in Q_1\} \cup \mathcal Z
		\quad\textrm{ and }
		\mathfrak I = (c_{\alpha} B_{\alpha} - c_{\bar \alpha} B_{\bar \alpha}),
	\end{equation}
	where the $A_\alpha$'s and $B_\alpha$'s are defined with respect to the multiplicity function $\mm$. Moreover, define
	\begin{equation}
		z= \sum_{\alpha \in Q_1} \alpha g(\alpha)\cdots g^{m'_{\alpha}n_\alpha -1}(\alpha),
	\end{equation}
	as well as an element
	\begin{equation}
		\widehat z_\alpha = \alpha g(\alpha)\cdots g(\alpha)^{m'_{\bar \alpha}n_{\alpha}-1} +  \bar{\alpha} g(\bar{\alpha})\cdots g^{m'_{\bar \alpha}n_{\bar \alpha}-1}(\bar{\alpha})
	\end{equation}
	for each $\alpha \in Q_1$ (note the use of ``$m'_{\bar{\alpha}}$'' instead of ``$m'_{\alpha}$'' in the first summand). We need to check that the assumptions of Proposition~\ref{prop pullback} are satisfied. Note that the $\alpha \in Q_1$ replace the indices $1\leq i \leq m$ from Proposition~\ref{prop pullback} (technically we should also include the elements of $\mathcal Z$, but they are irrelevant for the verification). We have the following:
	\begin{enumerate}
		\item $A_0/\xoverline{(\mathfrak R(0), \mathfrak I)}$ is isomorphic to the generalised weighted surface algebra $\Lambda(Q, f, \mm, \cc, 0, \emptyset)$ (i.e. all relations of special biserial type, rendering the relations in $\mathcal Z$ redundant). Therefore it has dimension $\sum_{\alpha\langle g \rangle} m_\alpha n_{\alpha}^2$. The algebra
		      $A_0/\xoverline{(\mathfrak R(0), z)}$
		      is exactly the twisted Brauer graph algebra $\LambdaTw(Q,f,\mm')$ as in Definition~\ref{def twisted bga}. It is a special biserial algebra spanned by the initial subwords of the elements $\alpha g(\alpha)\cdots g^{n_\alpha m'_\alpha-1}(\alpha)$ for $\alpha \in Q_1$, with  the added relation that $\alpha g(\alpha)\cdots g^{n_\alpha m'_\alpha-1}(\alpha)=-\bar \alpha g(\bar \alpha)\cdots g^{n_{\bar\alpha} m'_{\bar \alpha}-1}(\bar\alpha)$ (and such an element lies in the socle). Counting elements in the given basis yields that this algebra has dimension $\sum_{\alpha\langle g \rangle} m'_\alpha n_{\alpha}^2$.  Hence the sum of the dimensions is equal to
		      \begin{equation}
			      \sum_{\alpha\langle g \rangle\in Q_1/\langle g \rangle} (m_\alpha + m'_\alpha) n_{\alpha}^2,
		      \end{equation}
		      which is precisely the dimension of $\Lambda(Q, f, \mm + \mm', \cc, \tt, \mathfrak Z)$. Now
		      \begin{equation}
			      A_0/\xoverline{(\mathfrak R(z),\mathfrak I\cdot z)}
		      \end{equation}
		      is another presentation of $\Lambda(Q, f, \mm + \mm', \cc, \tt, \mathfrak Z)$ by Proposition~\ref{prop alternative pres}, which shows the dimension condition is satisfied. Let us quickly outline why the assumptions of Proposition~\ref{prop alternative pres} are fulfilled.  What could potentially go wrong is that $zA_{\bar \alpha}$ or $zB_{\bar \alpha}$ could contain a summand $\alpha \bar\alpha$ for some $\alpha$ with $t_\alpha \neq 0$. Now, this can only happen if $n_\alpha=1$, which implies $\alpha\neq f(\alpha)$, and therefore $t_\alpha \in \{0,1\}$. If $t_\alpha=0$ then there is no issue. If $t_\alpha=1$ then $f^3(\alpha)=\alpha$, which shows that $\alpha$ and $f^2(\alpha)$ have the same target. Then $g(f(\alpha))$ cannot also have the same target (which is at the same time the source of $f(\alpha)$), which forces $g^2(f(\alpha))\neq f(\alpha)$, that is, $n_{f(\alpha)}=n_{\bar{\alpha}}>2$. But then $A_{\bar{\alpha}}$ has length $\geq 2$, and there is no issue.
		\item
		      $k\sqq{X}\otimes_{k\sbb{X}} A/\xoverline{(\mathfrak R(X), \mathfrak I)_{A}}$ is the algebra $\Lambda(Q,f,\mm, X\cdot \cc, \tt, \mathfrak Z)$ defined over $k\sqq{X}$, which is a generalised weighted surface algebra in the sense of Definition~\ref{def weighted surface} by assumption. Also,
		      \begin{equation}
			      A/\xoverline{(\mathfrak R(X), \mathfrak I, X)_{A}}\cong A_0 /\xoverline{(\mathfrak R(0), \mathfrak I)_{A_0}}
		      \end{equation}
		      is the generalised weighted surface algebra  $\Lambda(Q,f,\mm, \cc, 0,\emptyset)$ defined over $k$ (here all $t_\alpha$'s are zero, which means that this automatically satisfies Definition~\ref{def weighted surface}). Since the aforementioned two generalised weighted surface algebras have the same dimension (over $k\sqq{X}$ and $k$, respectively), it follows that $A/\xoverline{(\mathfrak R(X), \mathfrak I)}$ is in fact a $k\sbb{X}$-order (which will also be useful  on its own further below), and therefore any element which becomes zero upon tensoring with $k\sqq{X}$ (that is, zero in $\Lambda(Q,f,\mm, X\cdot \cc, \tt, \mathfrak Z)$ over $k\sqq{X}$) was already zero in $A/\xoverline{(\mathfrak R(X), \mathfrak I)}$.

		      Since $m'_{\bar \alpha}>1$ whenever $n_{\bar \alpha}=1$ and $t_\alpha\equiv 1$ by assumption,  we can now  use  the second part of Definition~\ref{def weighted surface} to get that $\widehat z_\alpha A_{\bar \alpha} =0$ in  $A/\xoverline{(\mathfrak R(X), \mathfrak I)_A}$ whenever $t_{\alpha}\neq 0$ (note that $t_{\alpha}\neq 0$ and $t_\alpha\equiv 0$ implies $n_{\bar{\alpha}}>1$, so Definition~\ref{def weighted surface} applies in all cases where $t_\alpha\neq 0$), and therefore in particular $\widehat z_{\alpha} r_{\alpha,1}\in \xoverline{(\mathfrak R(X), \mathfrak I)_A}$ in the notation of Proposition~\ref{prop pullback} for all $\alpha\in Q_1$ (if $t_\alpha=0$ then $r_{\alpha,1}=0$, so we have indeed covered all cases). We also have that both $\widehat z_\alpha A_{\bar \alpha}$ and  $z A_{\bar{\alpha}}$ become equal to $\bar{\alpha}g(\bar{\alpha})\cdots g^{m_{\bar \alpha}'n_{\bar{\alpha}}-1}(\bar{\alpha})  A_{\bar \alpha}$ modulo $\xoverline{(\mathfrak R(0))_{A_0}}$ if $t_\alpha\neq 0$, which implies the other condition from equation~\eqref{eqn zi hat} the $\widehat z_\alpha$ need to satisfy. Moreover,  by Proposition~\ref{prop alternative pres} (which applies for the same reasons as earlier) we have that both
		      \begin{equation}
			      A/\xoverline{(\mathfrak R(z), \mathfrak I \cdot z, X)_A}\quad \textrm{and} \quad A/\xoverline{(r_\alpha(\widehat z_\alpha), \mathfrak I \cdot z, X \ | \ \alpha \in Q_1)_A}
		      \end{equation}
		      are isomorphic to $\Lambda(Q,f,\mm + \mm', \cc, \tt,\mathfrak Z)$, which gives the condition from equation~\eqref{eqn zi hat ideal}.
		\item We have already seen further up that $A/\xoverline{(\mathfrak R(X), \mathfrak I)_A}$ is a $k\sbb{X}$-order. By Proposition~\ref{prop biserial order} the algebra  $A/\xoverline{(\mathfrak R(0), X+z)_A}$ is also a $k\sbb{X}$-order, since it is isomorphic to $A_0/\xoverline{(\alpha f(\alpha) \ | \ \alpha \in Q_1)_A}$ with $X$ acting as $-z$.
	\end{enumerate}
	It follows that all conditions of Proposition~\ref{prop pullback} are satisfied, which gives us a $k\sbb{X}$-order $\Gamma$ as a pullback of  $A/\xoverline{(\mathfrak R(X), \mathfrak I)_A}$ and  $A/\xoverline{(\mathfrak R(0), X+z)_A}$ with $\Gamma/X\Gamma \cong \Lambda(Q, f, \mm + \mm', \cc, \tt,\mathfrak Z)$, as claimed.
\end{proof}

\begin{prop}\label{prop multiplicity independence}
	Assume we are in the situation of Proposition~\ref{prop lift weighted surface}, and assume we are given another map $\mm'':\ Q_1/\langle g \rangle\longrightarrow  \Z_{>0}$ subject to the same conditions as $\mm'$. Then the orders $\Gamma(Q,f,\mm, \cc, \tt,\mathfrak Z; \mm')$ and $\Gamma(Q,f,\mm, \cc, \tt, \mathfrak Z; \mm'')$  defined in Proposition~\ref{prop lift weighted surface} are isomorphic as rings if either one of the following holds:
	\begin{enumerate}
		\item\label{prop muly indep part 1} $m'_\alpha > m_{\alpha}$ and $m''_\alpha > m_{\alpha}$ for all $\alpha \in Q_1$ for which $m'_\alpha\neq m''_\alpha$,
		\item $m_\alpha n_\alpha \geq 3$ for all $\alpha\in Q_1$ with $t_{\bar{\alpha}}\equiv 1$, and $\{ \alpha f(\alpha) g(f(\alpha))  \ | \ \alpha \in Q_1\} \subseteq \mathfrak Z$,
		\item $m_\alpha n_\alpha \geq 4$ for all $\alpha \in Q_1$  with $t_{\bar{\alpha}}\equiv 1$ ,
		\item  $\textrm{char}(k)=2$,  $Q$, $f$, $\mm$, $\cc$ as well as $\tt$ are as in Proposition~\ref{prop semimimple q3k}, and $\mathfrak Z =\emptyset$. 
	\end{enumerate}
\end{prop}
\begin{proof}
	As a pullback diagram of rings, the diagram~\eqref{eqn pullback} is isomorphic to the diagram
	\begin{equation}\label{eqn pullback rings}
		\xymatrix{
			A/\xoverline{(\mathfrak R(X), \mathfrak I)_A} \ar@{->>}[r]^{\varphi} & A_0/\xoverline{(\mathfrak R(0), \mathfrak I)_{A_0}}\\ \Gamma	 \ar[u]\ar[r] & A_0/\xoverline{(\mathfrak R(0))_{A_0}} \ar@{->>}[u]_{\nu}
		}
	\end{equation}
	where the map $\nu$ is the natural surjection, and the map $\varphi$ is induced by the map from $A$ to $A_0$ which sends $X$ to $-z$ (using the notation of Proposition~\ref{prop pullback})  and induces the identity on all arrows, extended to the completion.  For both $\mm'$ and $\mm''$ we get a diagram as in \eqref{eqn pullback rings}, with two different maps $\varphi$, say $\varphi_1$ and $\varphi_2$, but otherwise identical. The maps $\varphi_i$ are actually determined by the image of $X$. Clearly the two diagrams will be isomorphic if both $\varphi_1$ and $\varphi_2$ send $X$ to $0$, but this is not always the case. We will show below to what extent the map $\varphi$ in \eqref{eqn pullback rings} is independent of $\mm'$ (or $\mm''$) under our four alternative assumptions.

	Concretely, in the construction of $\Gamma(Q,f,\mm, \cc, \tt,\mathfrak Z; \mm')$ in Proposition~\ref{prop lift weighted surface} we have that
	\begin{equation}
		A_0/\xoverline{(\mathfrak R(0), \mathfrak I)_{A_0}} \cong \Lambda(Q,f,\mm, \cc, 0, \emptyset)
	\end{equation}
	and
	\begin{equation}
		z=  \sum_{\alpha \in Q_1} \alpha g(\alpha)\cdots g^{m'_{\alpha}n_\alpha -1}(\alpha).
	\end{equation}
	Now we have the following cases, corresponding to the various possible assumptions from the statement:
	\begin{enumerate}
		\item
		      If $m'_\alpha > m_{\alpha}$ for all $\alpha \in Q_1$ then $z$ becomes zero in $\Lambda(Q,f,\mm, \cc, 0, \emptyset)$ since it is a sum of paths each properly containing a subword of the form $B_\alpha$ (with respect to this generalised weighted surface algebra). Hence $\varphi$ maps $X$ to $0$. If instead we only have $m'_\alpha\geq m_{\alpha}$ for all $\alpha \in Q_1$, then $\varphi$ maps $X$ to
		      \begin{equation}
			      -\sum_{\alpha\in Q_1 \textrm{ s.t. } m_{\alpha}=m_{\alpha'}} B_{\alpha},
		      \end{equation}
		      which clearly only depends on the set $\{\alpha \in Q_1 \ | \ m_{\alpha}=m_{\alpha}' \}$.
		\item If $m_\alpha n_\alpha \geq 3$ for all $\alpha\in Q_1$ with $t_{\bar{\alpha}}\equiv 1$ and $\{ \alpha f(\alpha) g(f(\alpha))  \ | \ \alpha \in Q_1\} \subseteq \mathfrak Z$, then the $k\sqq{X}$-span of $A/\xoverline{(\mathfrak R(X), \mathfrak I)_A}$, which is isomorphic to
		      \begin{equation}
			      \Lambda(Q,f,\mm, X\cdot \cc, \tt, \mathfrak Z)\quad \textrm{defined over $k\sqq{X}$,}
		      \end{equation}
		      is a generalised weighted surface algebra in the sense of Definition~\ref{def weighted surface}. Moreover, Proposition~\ref{prop weighted surface cond} applies to this algebra, which implies that elements of the form $\alpha g(\alpha) f(g(\alpha))$ and $\alpha f(\alpha) g(f(\alpha))$ for $\alpha \in Q_1$ are zero in this algebra, and therefore also in $A/\xoverline{(\mathfrak R(X), \mathfrak I)_A}$. From this it is easy to see that $z$ is actually central in  $A/\xoverline{(\mathfrak R(X), \mathfrak I)_A}$, and $r_{\alpha}(X+z)$ becomes zero in  $A/\xoverline{(\mathfrak R(X), \mathfrak I)_A}$ (since $zA_{\bar\alpha}t_\alpha$ becomes zero, as our assumptions ensure that $A_{\bar\alpha}t_{\alpha}$ has length $\geq 2$ if $t_{\alpha}\neq 0$). This implies that there is an automorphism of  $A/\xoverline{(\mathfrak R(X), \mathfrak I)_A}$ sending $X$ to $X+z$. Since $\varphi$ maps $X$ to $-z$, precomposing $\varphi$ with this automorphism yields the homomorphism which sends $X$ to $0$. That is, the pullback diagram~\eqref{eqn pullback rings} is isomorphic in this case to the pullback diagram where $\varphi$ is replaced by the map which sends $X$ to $0$ and all arrows to themselves.
		\item Using Proposition~\ref{prop weighted surface cond} we see that this case is identical to the previous case, since  elements of the form $\alpha g(\alpha) f(g(\alpha))$ and $\alpha f(\alpha) g(f(\alpha))$ for $\alpha \in Q_1$ become zero in ${\Lambda(Q,f,\mm, X\cdot \cc, \tt, \emptyset)}$ without being added as relations explicitly.
		\item In this case we define an element
		      \begin{equation}
			      \widehat z = \sum_{\alpha \in Q_1} (\alpha g(\alpha))^{m'_{\alpha}} - X^{2m'_{\alpha}} \cdot e_{s(\alpha)} = z - \sum_{\alpha\in Q_1}  X^{2m'_{\alpha}} \cdot e_{s(\alpha)},
		      \end{equation}
		      where $s(\alpha)$ denotes the source of the arrow $\alpha$, and $e_{s(\alpha)}$ is the corresponding idempotent. Proposition~\ref{prop semimimple q3k} gives an explicit isomorphism between the $k\sqq{X}$-span of $A/\xoverline{(\mathfrak R(X), \mathfrak I)_A}$, which is  $\Lambda(Q,f,\mm, X\cdot \cc, \tt, \emptyset)$ defined over $k\sqq{X}$, and the $k\sqq{X}$-algebra
		      \begin{equation}
			      k\sqq{X}\oplus k\sqq{X}\oplus k\sqq{X} \oplus M_3(k\sqq{X}).
		      \end{equation}
		      Under this isomorphism, the element $\widehat z$ gets mapped to an element with non-zero entries only in the first three components, while all arrows get mapped to elements with non-zero entries only in the last component. Hence $\widehat z$ is central in  $A/\xoverline{(\mathfrak R(X), \mathfrak I)_A}$, and it annihilates all arrows, which implies that $r_{\alpha}(X+\widehat z)$ is zero in $A/\xoverline{(\mathfrak R(X), \mathfrak I)_A}$ for all $\alpha \in Q_1$. Moreover, $X+\widehat z$ together with $Q_1$ and the idempotents generate a $k$-algebra whose completion is $A/\xoverline{(\mathfrak R(X), \mathfrak I)_A}$, since $\widehat z$ is contained in the square of the radical of this algebra. It follows that we get an automorphism on  $A/\xoverline{(\mathfrak R(X), \mathfrak I)_A}$ which sends $X$ to $X+\widehat z$, and the composition of this automorphism with $\varphi$ sends $X$ to $\sum_{\alpha\in Q_1} z^{2m'_{\alpha}} \cdot e_{s(\alpha)}$. Now this is clearly zero in $ A_0/\xoverline{(\mathfrak R(0), \mathfrak I)_{A_0}}\cong \Lambda(Q,f,\mm,\cc, 0,\emptyset)$, since $z$ already lies in the socle of this algebra. \qedhere
	\end{enumerate}
\end{proof}

In all except the first of the alternative assumptions of Proposition~\ref{prop multiplicity independence} the $k\sbb{X}$-orders $\Gamma(Q,f,\mm, \cc, \tt,\mathfrak Z; \mm')$ and $\Gamma(Q,f,\mm, \cc, \tt, \mathfrak Z; \mm'')$ end up being isomorphic as rings to a pullback as in \eqref{eqn pullback rings} where $\varphi$ sends $X$ to $0$. For future reference, we give this ring explicitly below.

\begin{defi}\label{defi gamma0}
	In the situation of Proposition~\ref{prop lift weighted surface} define $\Gamma_0=\Gamma_0(Q,f,\mm,\cc,\tt,\mathfrak Z)$ as the pullback of rings
	\begin{equation}
		\xymatrix{
		\frac{\widehat{k[X]Q}}{\phantom{.}\overline{(\alpha f(\alpha) - X\cdot c_{\bar{\alpha}} A_{\bar{\alpha}}t_\alpha ,\
		c_\alpha B_\alpha-c_{\bar{\alpha}}B_{\bar \alpha},\ \mathfrak Z \ | \ \alpha \in Q_1)}\phantom{.}} \ar@{->>}[rr]^{\varphi} && \frac{\widehat{kQ}}{\phantom{.}\overline{(\alpha f(\alpha),\  c_\alpha B_\alpha-c_{\bar{\alpha}}B_{\bar \alpha}\ | \ \alpha \in Q_1)} \phantom{.}}\\ \Gamma_0	 \ar[u]\ar[rr] && \frac{\widehat{kQ}}{\phantom{.}\overline{(\alpha f(\alpha) \ | \  \alpha \in Q_1)}{\phantom{.}}} \ar@{->>}[u]_{\nu} }
	\end{equation}
	where $\nu$ is the natural surjection and $\varphi$ is the $k$-algebra homomorphism sending $X$ to $0$, the arrows in $Q_1$ to themselves, extended to the completion.
\end{defi}

\section{Bijections of silting complexes}

We now have lifts of generalised weighted surface algebras to $k\sbb{X}$-orders in many cases. We would like to use these lifts to prove a statement about their silting complexes. To this end, we need to understand how silting complexes over a $k\sbb{X}$-order relate to silting complexes over its reduction modulo $X$. Let $R$ denote a complete discrete valuation ring with maximal ideal $\pi R$, and let $\Gamma$ be an $R$-order. It was observed in \cite{KoenigZimmermannTiltingHereditary} that the endomorphism ring of a tilting complex $T^\xbullet$ over $\Gamma$ is not necessarily an $R$-order, and by \cite[Theorem~2.1]{RickardDerEqDerFun} and \cite[Theorem 3.3]{RickardLiftTilting}  it is an $R$-order if and only if $T^\xbullet$ reduces to a tilting complex over $\Gamma/\pi \Gamma$. That is, tilting complexes do not necessarily reduce to tilting complexes. It turns out that silting complexes are better behaved in this regard.

\begin{prop}\label{prop reduction silting}
	Let $R$ be a complete discrete valuation ring with maximal ideal $\pi R$, and let $\Gamma$ be an $R$-order. Define  $\bar \Gamma = \Gamma/\pi\Gamma$, and use ``$\phantom{\cdot}\xoverline{\phantom{a}}$'' to denote application of the functor ``$R/\pi R \otimes_{R} -$''. Let $X^\xbullet$ and $Y^\xbullet$ be bounded complexes of finitely generated projective $\Gamma$-modules. We claim that
	\begin{equation}
		\Hom_{\mathcal K^b(\proj\Gamma)}(X^\xbullet, Y^\xbullet[i]) =0\textrm{ for all $i>0$}
	\end{equation}
	if and only if
	\begin{equation}
		\Hom_{\mathcal K^b(\proj\bar \Gamma)}(\bar X^\xbullet, \bar Y^\xbullet[i]) =0\textrm{ for all $i>0$}.
	\end{equation}
	In particular,  $X^\xbullet$ is a pre-silting complex over $\Gamma$ if and only if $\bar X^\xbullet$ is a pre-silting complex over~$\bar{\Gamma}$.
\end{prop}
\begin{proof}
	For the ``if''-direction assume $\Hom_{\mathcal K^b(\proj\bar \Gamma)}(\bar X^\xbullet, \bar Y^\xbullet[i]) =0$ for all $i>0$. Let ${f:\ X^\xbullet \longrightarrow Y^\xbullet[i]}$, for some $i>0$, be a map (of $\Z$-graded modules) that commutes with the differential. Then ${\bar f:\ \bar X^\xbullet \longrightarrow \bar Y^\xbullet[i]}$ must be null-homotopic by assumption, i.e. there is a map ${\bar h:\ \bar X^\xbullet \longrightarrow \bar Y^\xbullet[i-1]}$ such that ${\bar h\circ d_{\bar X} - d_{\bar Y[i-1]} \circ \bar h = \bar f}$ (note that $ d_{\bar Y[i-1]}=- d_{\bar Y[i]}$ by the usual sign conventions). Since homotopies are not required to commute with the differential, we can lift it to a map ${h:\  X^\xbullet \longrightarrow  Y^\xbullet[i-1]}$ such that $f'=f-h\circ d_{ X} + d_{Y[i-1]} \circ h \in \pi \Hom_{\mathcal K^b(\proj{\Gamma})}(X^\xbullet, Y^\xbullet[i])$. Hence we may represent the homotopy class of $f$ as $f'=\pi\cdot g$  for some other  map $g$ commuting with the differential. But the same argument applies to $g$, and can be repeated indefinitely thereafter. That is, $f\in \pi^n 	\Hom_{\mathcal K^b(\proj \Gamma)}(X^\xbullet, Y^\xbullet[i])$ for all $n>0$, which implies that $f$ is in fact null-homotopic.

	Now let us prove the ``only if''-direction. Assume that  $\Hom_{\mathcal K^b(\proj \Gamma)}( X^\xbullet,  Y^\xbullet[i]) =0$ for all ${i>0}$.  We want to show $\Hom_{\mathcal K^b(\proj\bar \Gamma)}(\bar X^\xbullet, \bar Y^\xbullet[j]) =0$ for arbitrary $j>0$.  To this end, consider a map ${\bar f: \bar X^\xbullet \longrightarrow \bar Y^\xbullet[j] }$ that commutes with the differential. Such a map can be lifted to a map ${f:\ X^\xbullet \longrightarrow Y^\xbullet[j] }$, but $f\circ d_{ X} - d_{Y[j]} \circ f$ is a potentially non-zero map from $X^\xbullet$ into $Y^{\xbullet}[j+1]$ commuting with the differential and reducing to zero modulo $\pi$. By assumption, we get a map ${h:\ X^\xbullet \longrightarrow Y^\xbullet[j]}$ such that 
	\begin{equation}
		\frac{1}{\pi} (f\circ d_{ X} - d_{Y[j]} \circ f) = h\circ d_{ X} - d_{Y[j]} \circ h.
	\end{equation}
	This tells us that $f-\pi \cdot h$ does commute with the differential, and is therefore null-homotopic by assumption. But since $f-\pi \cdot h$ reduces to $\bar f$ modulo $\pi$, it follows that $\bar f$ is null-homotopic, too. Hence~$\Hom_{\mathcal K^b(\proj\bar \Gamma)}(\bar X^\xbullet, \bar Y^\xbullet[j]) =0$.
\end{proof}

\begin{defi}[``torsion-silting'']\label{def tsilt}
	Let $\Gamma$ be a ring. We call  a pre-silting complex $X^\bullet\in \mathcal K^b(\proj{\Gamma})$ \emph{torsion-silting} if
	\begin{equation}
		\left\{ S^\bullet \in \mathcal D(\Gamma) \ | \ \Hom_{\mathcal D(\Gamma)}(X^\bullet[i], S^\bullet) =0\ \forall i \right\} \subseteq \left\{
		S^\bullet \in \mathcal D(\Gamma) \ | \ H^i(S^\bullet)=H^i(S^\bullet)\cdot \rad(\Gamma)\ \forall i \right\}.
	\end{equation}
	We denote the set of isomorphism classes of basic torsion-silting complexes by $\tsilt{\Gamma}$.
\end{defi}

\begin{prop}\label{prop silti to silt}
	In the situation of Proposition~\ref{prop reduction silting}, the complex $X^\xbullet$ is torsion-silting if and only if $\bar X^\bullet$ is silting.
\end{prop}
\begin{proof}
	Let $\mathcal T_\Gamma$ denote the smallest localising subcategory of $\mathcal K(\Gamma)=\mathcal K(\Mod{\Gamma})$ containing $\mathcal K^-(\Proj{\Gamma})$. Define $\mathcal T_{\bar \Gamma}$ analogously. By \cite[Proposition 2.12]{BokstedtNeeman} the canonical functors $\mathcal T_{\Gamma} \longrightarrow \mathcal D(\Gamma)$ and $\mathcal T_{\bar \Gamma} \longrightarrow \mathcal D(\bar \Gamma)$ are equivalences, and in particular any complex in $\mathcal D(\Gamma)$ and $\mathcal D(\bar{\Gamma})$ is isomorphic to a complex with projective terms. We will also repeatedly use that $\Hom_{\mathcal K(\Gamma)}(X^\bullet,-)\cong \Hom_{\mathcal D(\Gamma)}(X^\bullet,-)$, since $X^\bullet$ is a bounded complex of projectives (also, the analogous statement for $\bar{\Gamma}$ and $\bar X^\bullet$).

	Assume that $X^\bullet$ is torsion-silting. Let $T^\bullet\in \mathcal D(\bar{\Gamma})$ be an element such that $\Hom_{\mathcal D(\bar{\Gamma})}(\bar X^\bullet[i], T^\bullet)=0$ for all $i\in \Z$. Then
	\begin{equation}
		0=\Hom_{\mathcal D(\bar{\Gamma})}(\bar X^\bullet[i], T^\bullet)\cong \Hom_{\mathcal K(\bar{\Gamma})}(\bar X^\bullet[i], T^\bullet)=\Hom_{\mathcal K(\Gamma)}(X^\bullet[i], T^\bullet)\cong \Hom_{\mathcal D(\Gamma)}(X^\bullet[i], T^\bullet)
	\end{equation}
	for all $i$. Since $X^\bullet$ is torsion-silting this implies that $H^i(T^\bullet)\cdot \rad (\Gamma)= H^i(T^\bullet)$ for all $i$. By assumption, $\Gamma$ is an $R$-order and therefore there is a $j\geq 1$ such that $\rad(\Gamma)^j \subseteq \pi\Gamma$. It follows that  $\pi H^i(T^\bullet) = H^i(T^\bullet)$ for all $i$. But $T^\bullet$ is a complex of $\bar{\Gamma}$-modules, on which $\pi$ acts as zero. Hence  $H^i(T^\bullet) = 0$ for all $i\in \Z$, which implies that $T^\bullet \cong 0$ in $\mathcal D(\bar{\Gamma})$. We conclude that $\add (\bar X^\bullet)$ is a ``compact generating subcategory'' of $\mathcal D(\bar{\Gamma})$, which by \cite[Proposition~4.2]{AiharaIyamaSiltingMutation} implies that $\operatorname{thick}(\bar X^\bullet)=\mathcal K^b(\proj{\bar\Gamma})$.

	For the other direction, assume that $\operatorname{thick}(\bar X^\bullet)=\mathcal K^b(\proj{\bar{\Gamma}})$. Then
	\begin{equation}
		\{ T^\xbullet\in \mathcal D(\bar{\Gamma}) \ | \ \Hom_{D(\bar{\Gamma})}(\bar X^\xbullet[i], T^\bullet)=0 \textrm{ for all $i\in\Z$}\}=0.
	\end{equation}
	Assume that $S^\bullet\in\mathcal T_\Gamma$ has the property that $\Hom_{\mathcal D(\Gamma)}(X^\bullet[i], S^\bullet)=0$ for all $i\in \Z$. Consider the complex $\bar S^\bullet$, which lies in $\mathcal T_{\bar{\Gamma}}$. We can argue as in the second part of Proposition~\ref{prop reduction silting}. An $\bar f \in \Hom_{\mathcal K(\bar \Gamma)}(\bar X^\bullet[i], \bar S^\bullet)$, for some $i\in \Z$, lifts to a homomorphism of $\Z$-graded modules $f\in \Hom_\Gamma(X^\bullet[i], S^\bullet)$. Then  $f\circ d_{X[i]} - d_{S} \circ f$ commutes with the differential, and has image contained in $\pi S^\bullet$. We can therefore find a homotopy  $h:\ X^\bullet[i] \longrightarrow S^\bullet$ such that $\pi^{-1}(f\circ d_{X[i]} - d_{S} \circ f) = h\circ d_{X[i]} - d_{S} \circ h$, from which it follows that $f-\pi h$ is a homomorphism of chain complexes, which must be null-homotopic since $\Hom_{\mathcal D(\Gamma)}(X^\bullet[i], S^\bullet)=0$. But then $\bar f$ is null-homotopic. In particular, $ \Hom_{\mathcal D(\bar \Gamma)}(\bar X^\bullet[i], \bar S^\bullet)=0$ for all $i\in \Z$, which implies (using the assumption on $\bar X^\bullet$) that $\bar S^\bullet \cong 0$ in $\mathcal D(\bar{\Gamma})$. Now $R/\pi R\otimes_R H^i(S^\bullet) \hookrightarrow H^i(\bar S^\bullet)=0$, which implies $\pi H^i(S^\bullet)= H^i(S^\bullet)$. Since $\pi \in \rad(\Gamma)$ it follows that $H^i(S^\bullet)\cdot \rad(\Gamma)=H^i(S^\bullet)$, showing that $X^\bullet$ is torsion-silting.
\end{proof}

\begin{remark}
	By \cite[Theorem 3.3]{RickardLiftTilting} it is also true that if $\bar X^\bullet$ is tilting the $X^\bullet$ is tilting.
\end{remark}

As a consequence, we get bijections of silting complexes in many cases. Note that we do not know whether $\tsilt{\Gamma}$ is partially ordered, but we can still define the relation ``$\leq$'' as in equation~\eqref{eqn def partial order}. In Corollary~\ref{corollary silting bijection}, $\tsilt{\Gamma}$ being a poset is part of the assertion.  In general, we do not know to what extent $\tsilt{\Gamma}$ can differ from $\silt{\Gamma}$. A torsion-silting complex which is not silting would correspond to a pre-silting complex which reduces to a silting complex over $\Gamma/\pi\Gamma$, but tensoring with the field of fractions $K$ of $R$ does not give a generator of $\mathcal K^b(\proj{K\otimes_R\Gamma})$. If, for example, $K\otimes_R\Gamma$ is semisimple then this cannot happen.

\begin{corollary}\label{corollary silting bijection}
	Let $\Gamma$ be a ring and let
	\begin{equation}
		\iota_1,\iota_2:\ k\sbb{X} \longrightarrow Z(\Gamma)
	\end{equation}
	be two embeddings which both turn $\Gamma$ into a $k\sbb{X}$-order. Then the functor $-\otimes_{\Gamma} \Gamma/\iota_i(X)\Gamma$ for $i\in\{1,2\}$ induces a bijection between the isomorphism classes of pre-silting complexes over $\Gamma$ and those over $\Gamma/\iota_i(X)\Gamma$, and in particular induces isomorphisms of partially ordered sets
	\begin{equation}
		\silt\Gamma/\iota_1(X)\Gamma \stackrel{\sim}{\longleftrightarrow} \tsilt\Gamma 	 \stackrel{\sim}{\longleftrightarrow}
		\silt\Gamma/\iota_2(X)\Gamma.
	\end{equation}
\end{corollary}
\begin{proof}
	It suffices to prove the claim for the relationship between $\Gamma /\iota_1(X)\Gamma$ and $\Gamma$. So let us assume without loss that $\Gamma$ is given as a $k\sbb{X}$-order. By \cite[Proposition 3.1]{RickardLiftTilting}, for any pre-silting complex $\bar T^\xbullet$ in $\mathcal K^b(\proj\Gamma/X\Gamma)$ there is a complex $T^\xbullet$, unique up to isomorphism in $\mathcal K^b(\proj\Gamma)$, such that $T^\xbullet\otimes_{\Gamma} \Gamma/X\Gamma \cong \bar T^\xbullet$. Moreover, by Proposition~\ref{prop reduction silting}, the complex $T^\xbullet$ in $\mathcal K^b(\proj\Gamma)$ is pre-silting if and only if  $\bar T^\xbullet$ is pre-silting over $\Gamma/X\Gamma$, and by Proposition~\ref{prop silti to silt} it is torsion-silting if and only if $\bar T^\xbullet$  is silting. This shows that the functor $-\otimes_\Gamma \Gamma/X\Gamma$ does indeed induce a bijection between basic torsion-silting complexes over $\Gamma$ and basic silting complexes over $\Gamma/X\Gamma$. The relation ``$\geq$'' is defined in equation~\eqref{eqn def partial order}, and, with this definition in mind, the main assertion of Proposition~\ref{prop reduction silting} can be restated as saying that $X^\xbullet \geq Y^\xbullet$ if and only if $\bar X^\xbullet \geq \bar Y^\xbullet$ ($X$ and $Y$ being as in that proposition), which shows that the functor $-\otimes_\Gamma \Gamma/X\Gamma$ preserves order (which also shows that $\tsilt{\Gamma}$ is a poset).
\end{proof}

\begin{thm}[Lifting theorem \& silting bijection]\label{thm main general}
	Let $Q$, $f$, $\cc$, $\tt$ and  $\mathfrak Z$ be as in \S\ref{section gen weighted surf}. For $\alpha\in Q_1$ set
	\begin{equation}
		e_\alpha = \left\{ \begin{array}{ll} 1& \textrm{if $n_\alpha > 1$,}\\ 2&\textrm{if $n_\alpha=1$.} \end{array} \right.
	\end{equation}
	Assume that we are given two multiplicity functions $ 	\mm\upbr{1}, \mm\upbr{2}:\ Q_1/\langle g \rangle \longrightarrow \Z_{\geq 2} $.  Moreover, assume that for each $i\in \{1,2\}$ and all $\alpha \in Q_1$ with $t_{\bar{\alpha}}\equiv 1$ we have one of the following:
	\begin{enumerate}
		\item  $m\upbr{i}_\alpha \geq e_\alpha+ \frac{3}{n_\alpha}$ if $\{ \alpha f(\alpha) g(f(\alpha))  \ | \ \alpha \in Q_1\} \subseteq \mathcal Z$, or
		\item $m\upbr{i}_\alpha \geq e_\alpha+ \frac{4}{n_\alpha}$.
	\end{enumerate}
	Then there are a ring $\Gamma$ and embeddings
	\begin{equation}
		\iota_i:\ k\sbb{X} \longrightarrow Z(\Gamma) \quad \textrm{ for $i\in\{1,2\}$},
	\end{equation}
	each endowing $\Gamma$ with the structure of a  $k\sbb{X}$-order, such that
	\begin{equation}
		\Lambda(Q,f,\mm\upbr{i},\cc, \tt, \mathcal Z) \cong \Gamma/\iota_i(X)\cdot\Gamma  \quad \textrm{ for $i\in\{1,2\}$}.
	\end{equation}
	In particular, there are isomorphisms of partially ordered sets
	\begin{equation}
		\xymatrix{
			&\tsilt\Gamma \ar@{<->}[rd]^\sim \ar@{<->}[ld]_\sim \\
			\silt\Lambda(Q,f,\mm\upbr{1},\cc, \tt, \mathcal Z) \ar@{<->}[rr]^\sim && \silt\Lambda(Q,f,\mm\upbr{2},\cc, \tt, \mathcal Z)
		}
	\end{equation}
	where the arrows going down are induced by the functor ``$-\otimes_{\Gamma} \Gamma/\iota_i(X)\cdot\Gamma$''.
\end{thm}
\begin{proof}
	Define a map
	\begin{equation}
		\mm:\ Q_1/\langle g \rangle\longrightarrow \Z_{>0}:\ \alpha \mapsto \left\{ \begin{array}{ll}\min\{m_\alpha\upbr{1}, m_\alpha\upbr{2}\}-e_\alpha &\textrm{ if $t_{\bar{\alpha}}\equiv 1$,} \\ \min\{m_\alpha\upbr{1}, m_\alpha\upbr{2}\}-1 &\textrm{ if $t_{\bar{\alpha}}\equiv 0$.}  \end{array} \right.
	\end{equation}
	Pick an $i\in \{1,2\}$ and define $\mm' = \mm\upbr{i}-\mm$. The definition of $\mm$ ensures that $m'_\alpha \geq 1$ and $m_\alpha\geq 1$ for all $\alpha\in Q_1$. Furthermore, for any $\alpha\in Q_1$ with $t_{\bar{\alpha}}\equiv 1$ our assumptions ensure that $m'_\alpha \geq e_\alpha$, which in turn ensures $m'_{\alpha}n_\alpha \geq 2$.   Also, when $t_{\bar{\alpha}}\equiv 1$ we have
	\begin{equation}\label{eqn sjdjjd}
		m_{\alpha} n_{\alpha} = \min \{ (m_\alpha\upbr{1} - e_\alpha)n_\alpha,  (m_\alpha\upbr{2} - e_\alpha)n_\alpha \} \geq  \left\{ \begin{array}{ll} 3 & \textrm{ if $\{ \alpha f(\alpha) g(f(\alpha))  \ | \ \alpha \in Q_1\} \subseteq \mathfrak Z$,} \\ 4 & \textrm{otherwise.} \end{array}\right.
	\end{equation}
	It follows that the assumptions of Proposition~\ref{prop lift weighted surface} are satisfied. Specifically,  the inequality~\eqref{eqn sjdjjd} implies that we can use Proposition~\ref{prop weighted surface cond} to verify that the algebra $\Lambda(Q,f,\mm, X\cdot \cc, \tt, \mathfrak Z)$ in the statement of Proposition~\ref{prop lift weighted surface} is a generalised weighted surface algebra. We checked the other condition of Proposition~\ref{prop lift weighted surface} verbatim further above. Hence we get a $k\sbb{X}$-order  $\Gamma(Q,f,\mm, \cc, \tt, \mathfrak Z; \mm')$ reducing to  $\Lambda(Q,f,\mm+\mm', \cc, \tt, \mathfrak Z) = \Lambda(Q,f,\mm\upbr{i}, \cc, \tt, \mathfrak Z)$.

	The inequality~\eqref{eqn sjdjjd} tells us that one of the alternative assumptions of Proposition~\ref{prop multiplicity independence} is satisfied, which implies that the isomorphism type of $\Gamma(Q,f,\mm, \cc, \tt, \mathfrak Z; \mm')$ as a ring does not depend on $\mm'$. To be specific, by the remarks following Proposition~\ref{prop multiplicity independence}, we have
	\begin{equation}
		\Gamma(Q,f,\mm, \cc, \tt, \mathfrak Z; \mm') \cong \Gamma_0(Q,f,\mm, \cc, \tt, \mathfrak Z) \quad \textrm{(from Definition~\ref{defi gamma0})}
	\end{equation}
	as rings. Our claims on the correspondence of silting complexes now follow from Corollary~\ref{corollary silting bijection}.
\end{proof}

The preceding theorem shows that generalised weighted surface algebras defined for the same combinatorial data but different multiplicities have common lifts, and therefore common silting posets, provided the multiplicities are big enough. It is inevitable that in some cases we could, in theory, allow slightly smaller multiplicities. The problem with incorporating these cases in a general theorem is that checking when a $\Lambda(Q,f,\mm, \cc,\tt,\mathfrak Z)$ is a generalised weighted surface algebra in the sense of Definition~\ref{def weighted surface} is tricky for multiplicities smaller than those allowed by Proposition~\ref{prop weighted surface cond} (for instance, the answer may no longer be independent of the characteristic of $k$). That is, one will have to consider this case-by-case, and verify the assumptions of Propositions~\ref{prop lift weighted surface}~and~\ref{prop multiplicity independence} by hand. Below we do this for one case which is of interest in the modular representation theory of finite groups.

\begin{prop}\label{prop application q3k}
	Assume $\operatorname{char}(k)=2$. Let $Q$, $f$, $\cc$ and $\tt$ be as in Proposition~\ref{prop semimimple q3k}. Then the conclusion of Theorem~\ref{thm main general} holds for any two functions $\mm\upbr{1}, \mm\upbr{2}:\ Q_1/\langle g \rangle \longrightarrow \Z_{\geq 2}$ .
\end{prop}
\begin{proof}
	We can proceed as in the proof of Theorem~\ref{thm main general}. Define $m_{\alpha}=1$  for all $\alpha\in Q_1$ and, after picking an $i\in\{1,2\}$, set $\mm' =\mm\upbr{i}-\mm$.  All verifications made in the proof of Theorem~\ref{thm main general} carry over to this case, apart from inequality~\eqref{eqn sjdjjd}. But the inequality~\eqref{eqn sjdjjd} is only used to check the assumptions of Proposition~\ref{prop weighted surface cond}, which we can replace by Proposition~\ref{prop semimimple q3k}, and to check the assumptions of Proposition~\ref{prop multiplicity independence}, which are also satisfied in this case.
\end{proof}

The advantage of Theorem~\ref{thm main general} is that it can be applied to arbitrary Brauer graph algebras, but in cases where they coincide with twisted Brauer algebras we can allow arbitrary multiplicities by using the lifts from Proposition~\ref{prop biserial order} instead.
\begin{prop}\label{prop improvement biserial}
	In the situation of Theorem~\ref{thm main general}, if $\tt=0$, $\cc=1$ and either $\operatorname{char}(k)=2$ or $k=\bar k$ and the Brauer graph of $(Q,f)$ in the sense of Definition~\ref{def brauer graph} is bipartite, then the conclusion of Theorem~\ref{thm main general} holds for any two functions $\mm\upbr{1}, \mm\upbr{2}:\ Q_1/\langle g \rangle \longrightarrow \Z_{\geq 1}$.
\end{prop}
\begin{proof}
	By Proposition~\ref{prop tw bga eq bga} our assumptions imply that $\Lambda(Q,f,\mm\upbr{i},\cc,\tt,\mathcal Z)\cong \LambdaTw(Q,f,\mm\upbr{i})$ for $i\in\{1,2\}$. By Proposition~\ref{prop biserial order} we can therefore choose $\Gamma=\GammaTw(Q,f)$.
\end{proof}

Theorem~\ref{thm main general}, together with Proposition~\ref{prop improvement biserial}, implies Theorem~\ref{thm bga tilting correspondence}. Since it is well-known that blocks of group algebras of cyclic defect are Morita equivalent to Brauer tree algebras, Proposition~\ref{prop improvement biserial} also implies the part of Theorem~\ref{thm block tilting correspondence} concerned with cyclic defect. Note that, in all of these cases, the algebras involved are symmetric, which is why we get statements on tilting complexes.

\section{Tilting bijections for algebras of dihedral, semi-dihedral and quaternion~type}

The aim of this section is to give a quick summary of what Theorem~\ref{thm main general} implies for algebras of dihedral, semi-dihedral and quaternion type as classified by Erdmann in \cite{TameClass}. Not all of these algebras are generalised weighted surface algebras, but all of them are derived equivalent to one (bar one exceptional family of $20$-dimensional algebras), and all of them are symmetric. Apart from being a class of concrete examples to apply Theorem~\ref{thm main general} to, these algebras also contain all blocks of dihedral, semi-dihedral and quaternion defect, which are of interest in the modular representation theory of finite groups. The derived equivalence classification of the algebras classified in \cite{TameClass} was carried out in \cite{HolmDerivedEquivalentTameBlocks} and \cite{HolmDerivedClassErdmannAlg}. We exclude all local algebras from our considerations, since their tilting complexes are just shifts of progenerators by \cite[Theorem~2.11]{RouquierZimmermann}.

\begin{defi}
	Define quivers
	\begin{equation}
		Q_B = \vcenter{\vbox{ \xygraph{ !{<0cm,0cm>;<1.5cm,0cm>:<0cm,1cm>::} !{(0,0) }*+{\bullet_{1}}="a" !{(2,0) }*+{\bullet_{2}}="b" "b" :@/^/^{\alpha_2} "a" "a" :@/^/^{\alpha_1} "b" "a" :@(lu,ld)_{\beta_1} "a" "b" :@(ru,rd)^{\beta_2} "b" }}}
	\end{equation}
	\begin{equation}
		Q_K = \vcenter{\vbox{ \xygraph{ !{<0cm,0cm>;<1.5cm,0cm>:<0cm,1cm>::} !{(0,0) }*+{\bullet_{1}}="a" !{(2,0) }*+{\bullet_{2}}="b" !{(1,-2) }*+{\bullet_{3}}="c" "b" :@/^/^{\beta_2} "a" "a" :@/^/^{\alpha_1} "b" "b" :@/^/^{\alpha_2} "c" "c" :@/^/^{\beta_3} "b" "c" :@/^/^{\alpha_3} "a" "a" :@/^/^{\beta_1} "c" }}} \quad	 Q_R = \vcenter{\vbox{\xygraph{ !{<0cm,0cm>;<1.5cm,0cm>:<0cm,1cm>::} !{(0,0) }*+{\bullet_{1}}="a" !{(2,0) }*+{\bullet_{2}}="b" !{(1,-2) }*+{\bullet_{3}}="c" "a" :@/^/^{\alpha_1} "b" "b" :@/^/^{\alpha_2} "c" "c" :@/^/^{\alpha_3} "a" "a" :@(lu,ld)_{\beta_1} "a" "b" :@(ru,rd)^{\beta_2} "b" "c" :@(ld,rd)_{\beta_3} "c" }}}
	\end{equation}
	and corresponding permutations
	\begin{equation}
		\begin{gathered}
			f_B =  (\alpha_1\beta_2\alpha_2),\  f_K = (\alpha_1\alpha_2\alpha_3)(\beta_3\beta_2\beta_1),\ f_R=(\alpha_1\beta_2\alpha_2\beta_3\alpha_3\beta_1), \\ g_B =  (\beta_1\alpha_1\alpha_2),\  g_K = (\alpha_1\beta_2)(\alpha_2\beta_3)(\alpha_3\beta_1),\ g_R=(\alpha_1\alpha_2\alpha_3).
		\end{gathered}
	\end{equation}
\end{defi}

\begin{thm}[see \cite{HolmDerivedClassErdmannAlg}]
	Let $k$ be an algebraically closed field.  If $A$ is a non-local $k$-algebra of dihedral, semi-dihedral or quaternion type in the sense of Erdmann \cite{TameClass}, then $A$ is derived equivalent to one of the following algebras:
	\begin{enumerate}
		\item $\mathcal D(2B)^{a_1,a_2}(v) = \Lambda(Q_B, f_B, \mm,\cc,\tt, \emptyset)$, with $m_{\beta_i}=a_i$ for all $i$, $\cc=1$, $t_{\beta_1}=v\beta_1$ and $t_{\beta_2}=0$, where $a_1\geq a_2 \geq 1$ and $v\in\{0,1\}$.
		\item $\mathcal D(3K)^{a_1,a_2,a_3}=\Lambda(Q_K,f_K,\mm,\cc,\tt,\emptyset)$ with $m_{\alpha_i}=a_i$ for all $i$, $\cc= 1$ and  $\tt= 0$,
		      where $a_1\geq a_2\geq a_3\geq 1$.
		\item $\mathcal D(3R)^{a_1,b_1,b_2,b_3}=\Lambda(Q_R,f_R,\mm,\cc,\tt,\emptyset)$ with $m_{\alpha_1}=a_1$ and $m_{\beta_i}=b_{i}$ for all $i$, $\cc= 1$ and  $\tt= 0$, where $b_1\geq b_2\geq b_3 \geq a_1\geq 1$ and $b_2\geq 2$.

		\item $\mathcal{SD}(2B)_1^{a_1,a_2}(v)=\Lambda(Q_B,f_B,\mm,\cc,\tt,\emptyset)$,  with $m_{\beta_i}=a_i$ for all $i$, $\cc=1$, $t_{\beta_1}=1+v\beta_1$ and $t_{\beta_2}=0$, where $a_1,a_2\geq 1$, $(a_1,a_2)\neq (1,1)$ and $v\in\{0,1\}$.
		\item $\mathcal{SD}(2B)_2^{a_1,a_2}(v)=\Lambda(Q_B,f_B,\mm,\cc,\tt, \{ \alpha_1\beta_2^2, \beta_2^2\alpha_2 \})$,  with $m_{\beta_i}=a_i$ for all $i$, $\cc=1$, $t_{\beta_1}=v\beta_1$ and $t_{\beta_2}=1$, where $a_1\geq 1$, $a_2\geq 2$, $a_1+a_2\geq 4$ and $v\in\{0,1\}$.
		\item $\mathcal{SD}(3K)^{a_1,a_2,a_3}=\Lambda(Q_K,f_K,\mm,\cc,\tt,\emptyset)$ with $m_{\alpha_i}=a_i$ for all $i$, $\cc= 1$, $t_{\alpha_1}=1$ and $t_{\beta_1}=0$,
		      where $a_1\geq a_2\geq a_3\geq 1$ and $a_1\geq 2$.

		\item $\mathcal Q(2B)_1^{a_1,a_2}(u,v)=\Lambda(Q_B,f_B,\mm,\cc,\tt,\{\beta_1^2\alpha_1, \alpha_2\beta_1^2\})$ with $m_{\beta_i}=a_i$ for all $i$, ${c_{\beta_1}=u}$, ${c_{\beta_2}=u^{1-a_2}}$ and ${t_{\beta_1}=1+v\beta_1}$, ${t_{\beta_2}=1}$, where $a_1\geq 1$, $a_2\geq 3$, $u\in k^\times$ and $v\in k$.
		\item $\mathcal Q(3K)^{a_1,a_2,a_3}=\Lambda(Q_K,f_K,\mm,\cc,\tt,\{\beta_2\alpha_1\alpha_2, \alpha_2\beta_3\beta_2, \alpha_3\beta_1\beta_3\})$ with $m_{\alpha_i}=a_i$ for all~$i$, $\cc=1$ and  $\tt=1$, where $a_1\geq a_2\geq a_3\geq 1$, $a_2\geq 2$ and $(a_1,a_2,a_3)\neq (2,2,1)$.
		\item $\mathcal Q(3A)_1^{2,2}(d)$, where $d\in k\setminus\{0,1\}$. These are algebras of dimension $20$ which are excluded in~\cite{HolmDerivedClassErdmannAlg}.
	\end{enumerate}
	In the names of these algebras, the letters ``$\mathcal D$'', ``$\mathcal{SD}$'' and ``$\mathcal Q$'' indicate the type (dihedral, semi-dihedral or quaternion), and this type is preserved under derived equivalences. \qed
\end{thm}

We should mention that in the case of $\mathcal Q(2B)_1^{a_1,a_2}(u,v)$ the presentation given in \cite{HolmDerivedClassErdmannAlg} is different from ours, and our parameters $u$ and $v$ are not the same as Holm's $a$ and $c$. To recover Holm's presentation one needs to replace the generator $\beta_2$ by $\beta_2'=u^{-1}\beta_2$, and then the parameters correspond as $a=u$ and $c=uv$. We also corrected the range of allowed parameters for $\mathcal{SD}(2B)_1^{a_1,a_2}(v)$.

\begin{corollary}
	Let $k$ be an algebraically closed field. Consider the following families of algebras:
	\begin{equation}
		\begin{gathered}
			\{ \mathcal D(2B)^{a_1,a_2}(v) \ | \ a_1,a_2\geq 2\} \textrm{ for $v\in\{0,1\}$},\ \{\mathcal D(3K)^{a_1,a_2,a_3} \ | \ a_1,a_2,a_3\geq 2 \}, \\
			\{ D(3R)^{a_1,b_1,b_2,b_3} \ |\ a_1,b_1,b_2,b_3\geq 2\},\\
			\{ \mathcal{SD}(2B)_1^{a_1,a_2}(v) \ | \ a_1 \geq 3, a_2 \geq 2 \},\
			\{ \mathcal{SD}(2B)_2^{a_1,a_2}(v) \ | \ a_1 \geq 2, a_2 \geq 5 \} \textrm{ for $v\in\{0,1\}$}, \\
			\{ \mathcal{SD}(3K)^{a_1,a_2,a_3} \ | \ a_1, a_2, a_3 \geq 3 \}, \\
			\{ \mathcal Q(2B)_1^{a_1,a_2}(u,v) \ | \ a_1 \geq 2, a_2\geq 5 \} \textrm{ for $u\in k^\times$ and $v\in k$}, \
			\{ \mathcal Q(3K)^{a_1,a_2,a_3} \ | \ a_1, a_2, a_3 \geq 3 \}.
		\end{gathered}
	\end{equation}
	If $\operatorname{char}(k)=2$ then also consider the families
	\begin{equation}
		\begin{gathered}
			\{ \mathcal D(2B)^{a_1,a_2}(0) \ | \ a_1,a_2\geq 1\},\  \{ \mathcal D(3K)^{a_1,a_2,a_3} \ | \ a_1,a_2,a_3\geq 1\}, \\
			\{ \mathcal Q(3K)^{a_1,a_2,a_3} \ | \ a_1,a_2,a_3\geq 2\}.
		\end{gathered}
	\end{equation}

	If $\mathcal F$ is any of the families listed above, and $A,B\in\mathcal F$ are two algebras in this family, then the pre-tilting complexes over $A$ and $B$ are in bijection, and the bijection induces an isomorphism of partially ordered sets
	\begin{equation}
		\tilt{A}\stackrel{\sim}{\longleftrightarrow} \tilt{B}.
	\end{equation}
\end{corollary}
\begin{proof}
	For the families in arbitrary characteristic this is a straightforward application of Theorem~\ref{thm main general}. For the family $\{ \mathcal Q(3K)^{a_1,a_2,a_3} \ | \ a_1,a_2,a_3\geq 2\}$ in characteristic two this is an application of Proposition~\ref{prop application q3k}. For the two families of algebras of dihedral type in characteristic two it follows from Proposition~\ref{prop improvement biserial}.
\end{proof}

It was shown in \cite{HolmDerivedEquivalentTameBlocks} that a block having quaternion defect group $Q_{2^n}$ (for $n\geq 3$) with three simple modules is derived equivalent to $\mathcal Q(3K)^{2,2,2^n-2}$, and that a non-local block with dihedral defect group $D_{2^n}$ is derived equivalent to either  $\mathcal D(3K)^{2,2,2^n-2}$ or $\mathcal D(2B)^{2,2^n-2}(v)$ with $v\in\{0,1\}$. It was shown by the author in \cite{EiseleDerEq} that no algebra of the form $\mathcal D(2B)^{2,2^n-2}(1)$ is derived equivalent to a block. This proves the part of Theorem~\ref{thm block tilting correspondence} concerned with blocks of dihedral or quaternion defect.

We should note that the families defined above do not contain any blocks of semi-dihedral defect, since the multiplicities in such blocks are too small for Theorem~\ref{thm main general}. This cannot easily be fixed, and it was shown in \cite{TauRigid} that the poset of basic two-term tilting complexes over $\mathcal{SD}(3K)^{a_1,a_2,a_3} $ is not the same for all parameter choices (note that we allow $a_i=1$, and the algebras so-obtained technically have a different name in Erdmann's classification).

In the case of blocks of quaternion defect with two simple modules there are two problems. The first one is that we would need to allow ``$a_2\geq 4$'' in the family ${\{ \mathcal Q(2B)_1^{a_1,a_2}(u,v) \ | \ a_1 \geq 2, a_2\geq 5 \}}$ to cover blocks with defect groups of order $<32$ (such blocks do not exist for the defect group $Q_{8}$, but they do for $Q_{16}$). This is only a minor issue and one should be able to extend this family in characteristic two similar to Proposition~\ref{prop application q3k}. The second problem are the parameters $u$ and $v$, which are not fully determined for blocks of quaternion defect. This is a long-standing problem, and it means that Donovan's conjecture is not fully settled for these blocks. As long as this is not resolved we need to exclude these blocks from Theorem~\ref{thm block tilting correspondence}.

\section{Multiplicity-independence of derived Picard groups}

We have seen that if a ring $\Lambda$ carries two different $k\sbb{X}$-algebra structures, both turning it into a $k\sbb{X}$-order, then it has two different reductions modulo $X$, say $\bar \Lambda_1$ and $\bar{\Lambda}_2$. Isomorphism classes of silting complexes over $\bar \Lambda_1$ and $\bar{\Lambda}_2$ are in bijection. The derived Picard groups of $\bar\Lambda_1$ and $\bar\Lambda_2$ both act on the respective posets of basic silting complexes, and one might be tempted to ask what their relationship is.  In the classes of algebras we study we will identify large common subgroups of $\TrPic_k(\bar\Lambda_1)$ and $\TrPic_k(\bar\Lambda_2)$ and show that $\Picent(\bar{\Lambda}_1)\cong\Picent(\bar{\Lambda}_2)$. This was motivated by the results on self-injective Nakayama algebras in \cite{VolkovZvonareva}. The results of this section, Theorems~\ref{thm trpic q3k}~and~\ref{thm Brauer graph}, imply Theorems~\ref{thm intro trpic bga}~and~\ref{thm intro trpic blocks}.

Let us now quickly sketch the method used in this section. Denote the two copies of $k\sbb{X}$ in $Z(\Lambda)$ by $R_1$ and $R_2$ (although later we will go with $R$ and $S$). The group $\TrPic_k(\bar\Lambda_i)$ (where $i\in\{1,2\}$) acts on $\tilt{\bar\Lambda_i}$, and $\Pic_k(\bar\Lambda_i)$ is the stabiliser of the trivial basic tilting complex (the stalk complex of a progenerator) under this action. In principle, each one-sided tilting complex $\bar T^\xbullet$ over $\bar{\Lambda}_i$ can be lifted to a tilting complex $T^\xbullet$ over $\Lambda$. However, there is no guarantee that tilting complexes with endomorphism ring $\bar{\Lambda}_i$ lift to tilting complexes with endomorphism ring $\Lambda$. This would only follow if $\Lambda$ was the \emph{unique} $R_i$-order reducing to $\bar{\Lambda}_i$. We will generalise the ideas of \cite{EiseleQuaternion} to show that the algebras we are interested in do lift uniquely up to technicalities. This is proved in Lemma~\ref{lemma unique lifting}, which is bespoke for the algebras we are interested in (but since it is separate from the rest of the proof one could easily replace it if one can prove the same statement for a different class of algebras). Assuming one has such a unique lifting property, one can always find an element $Y^\xbullet \in \TrPic_{R_i}(\Lambda)$ which restricts to the one-sided complex~$T^\xbullet$. This shows that $\TrPic_{R_i}(\Lambda)$ and $\Pic_k(\bar{\Lambda}_i)$ taken together generate $\TrPic_k(\bar{\Lambda}_i)$. However, the group $\TrPic_{R_i}(\Lambda)$  still depends on the $R_i$-algebra structure of $\Lambda$. We will formulate conditions on the Grothendieck groups and the centre of $\Lambda$ which ensure that we can even pick a lift in  $Y^\xbullet \in \TrPic_{R_1}(\Lambda)\cap \TrPic_{R_2}(\Lambda)$. The group $\TrPic_{R_1}(\Lambda)\cap \TrPic_{R_2}(\Lambda)$ is the same for $\bar{\Lambda}_1$ and $\bar{\Lambda}_2$, so it independent of $R_i$ in the sense in which we need it to be, but of course some further work is required to see how it interacts with $\Pic_k(\bar{\Lambda}_i)$.

\begin{lemma}[Unique lifting]\label{lemma unique lifting}
	Let $R$ be a complete discrete valuation ring with algebraically closed residue field $k=R/\pi R$ and field of fractions $K$. Moreover, let $\Lambda$ be an $R$-order in a semisimple $K$-algebra $A$ and set $\bar \Lambda = \Lambda /\pi\Lambda$. Denote by $e_1,\ldots,e_n$ ($n\in \N$) a full set of orthogonal primitive idempotents in $\Lambda$. Now suppose that all of the following hold:
	\begin{enumerate}[start=1,label={\bfseries($\mathbf \Lambda$\arabic*)}]
		\item\label{assumption lambda 1}  As a $Z(\bar\Lambda)$-module,  $e_i\bar \Lambda e_j$ is generated by a single element, for all $i,j\in \{1,\ldots,n\}$.
		\item\label{assumption lambda 2} Either $Z(\Lambda) \longrightarrow Z(\bar \Lambda)$ is surjective or $Z(\Lambda)e_i\subset Z(A)e_i$ is not properly contained in any local $R$-order with the same $K$-span for all $i\in\{1,\ldots,n\}$.
		\item\label{assumption lambda 3}  For all $i,j,l\in \{1,\ldots,n\}$ with $i\neq j$ and $j\neq l$
		      \begin{equation}
			      \dim_K(e_i A e_j A e_l) = \max \{ \dim_k(e_i\bar \Lambda e_j \bar{\Lambda} e_l),\ \dim_k(e_l\bar \Lambda e_j \bar{\Lambda} e_i) \}.
		      \end{equation}
		\item\label{assumption lambda 4} To each pair $i\neq j\in \{1,\ldots,n\}$ we can assign a simple $Z(A)$-module $Z_{i,j}$ in such a way that $Z_{i,j}= Z_{j,i}$ and for all pair-wise distinct $i,j,l\in \{1,\ldots,n\}$ the $Z(A)$-module  $e_i A e_j A e_l$ is either zero or
		      \begin{equation}
			      e_i A e_j A e_l\cong Z_{i,j}= Z_{j,l}= Z_{i,l}.
		      \end{equation}
		\item\label{assumption lambda 5}  For every $\sigma\in S_n$ such that
		      \begin{equation}
			      \dim_K(e_{\sigma(i)}Ae_{\sigma(j)}) =\dim_K(e_iAe_j) \quad \textrm{for all $i,j\in\{1,\ldots,n\}$}
		      \end{equation}
		      and
		      \begin{equation}
			      \dim_k(e_{i_1}\bar{\Lambda}e_{i_2}\bar{\Lambda}\cdots \bar{\Lambda}e_{i_r})\neq 0  \Longrightarrow \dim_k(e_{\sigma(i_1)}Ae_{\sigma(i_2)}A\cdots Ae_{\sigma(i_r)})\neq 0
		      \end{equation}
		      for all $r\geq 2$ and $i_1,\ldots,i_r\in\{1,\ldots,n\}$, there is a $\gamma \in \Aut_K(A)$ such that $\gamma(e_i)=e_{\sigma(i)}$ and $\gamma(Z(\Lambda))=Z(\Lambda)$.
	\end{enumerate}
	If $\Gamma$ is an $R$-order in $A$ such that
	\begin{enumerate}[start=1,label={\bfseries($\mathbf \Gamma$\arabic*)}]
		\item\label{assumption gamma 1} there is a full set of orthogonal primitive idempotents $f_1,\ldots,f_n\in \Gamma$ such that $f_iA\cong e_i A$ as $A$-modules for all $i\in\{1,\ldots,n\}$, and
		\item\label{assumption gamma 2} $\bar{\Gamma}=\Gamma/\pi\Gamma \cong \bar \Lambda$, and
		\item\label{assumption gamma 3} $Z(\Lambda)=Z(\Gamma)$,
	\end{enumerate}
	then $\Lambda\cong\Gamma$. If there is an isomorphism between $\Gamma$ and $\Lambda$ sending $f_i$ to $e_i$ for all $i$, then $\Gamma$ is even conjugate to $\Lambda$ within the units of $A$.
\end{lemma}
\begin{proof}
	First of all, the idempotents $f_1,\ldots,f_n$ are conjugate to $e_1,\ldots,e_n$ (in that order) within the units of $A$, so we can replace $\Gamma$ by a conjugate and assume $e_i=f_i$ for all $i$. If we fix an isomorphism between $\bar{\Lambda}$ and $\bar{\Gamma}$ sending $e_i+\pi \Lambda$ to $e_{\sigma(i)}+\pi\Gamma$  for some $\sigma\in S_n$ then $\dim_K e_iAe_j=\dim_k e_i\bar{\Lambda} e_j = \dim_k e_{\sigma(i)}\bar{\Gamma} e_{\sigma(j)}=  \dim_K e_{\sigma(i)}A e_{\sigma(j)}$, and  given $i_1,\ldots,i_r$ such that $e_{i_1}\bar{\Lambda} e_{i_2}\bar{\Lambda}\cdots \bar{\Lambda} e_{i_r} \neq 0$ we have $e_{\sigma(i_1)}\bar{\Gamma} e_{\sigma(i_2)}\bar{\Gamma}\cdots \bar{\Gamma} e_{\sigma(i_r)} \neq 0$, which implies $e_{\sigma(i_1)}A e_{\sigma(i_2)}A\cdots A e_{\sigma(i_r)} \neq 0$. Hence there is a $\gamma\in\Aut_K(A)$ as in assumption~\ref{assumption lambda 5}. If we replace $\Gamma$ by $\gamma^{-1}(\Gamma)$ then we can assume without loss of generality that there is an isomorphism between $\bar{\Lambda}$ and $\bar{\Gamma}$ sending $e_i+\pi\Lambda$ to $e_i+\pi\Gamma$ for all $i\in \{1,\ldots,n\}$. If, as in the very last sentence of the assertion, there is an isomorphism between $\Gamma$ and $\Lambda$ sending $f_i$ to $e_i$ for all $i$ then we can assume this without replacing $\Gamma$ by $\gamma^{-1}(\Gamma)$. We will now show that $\Gamma$ is conjugate to $\Lambda$, which will imply both parts of the claim.

	Since $Z(\Lambda)=Z(\Gamma)$ we have $Z(\Gamma)e_i  =  Z(\Lambda)e_i$ for all~$i$.  If some $Z(\Lambda) e_i$ is not contained in any larger local order in $Z(A)e_i$, then we must have $Z(\Lambda)e_i=e_i\Lambda e_i = e_i\Gamma e_i$. If $Z(\Lambda)\longrightarrow Z(\bar{\Lambda})$ is surjective, then $Z(\Gamma)\longrightarrow Z(\bar{\Gamma})$ must be surjective as well by virtue of dimensions. To see this note that $Z(\Gamma)$ is a pure sublattice of $\Gamma$ (this  holds for any $R$-order), which implies that the rank of $Z(\Gamma)$ equals the dimension of the image of $Z(\Gamma)$ in $\bar \Gamma$. Now it follows that each $e_i Z(\Gamma)$ maps surjectively onto $e_iZ(\bar \Gamma)=e_i\bar{\Gamma}e_i$. Any proper sublattice of $e_i\Gamma e_i$ maps to a proper subspace of $e_i\bar{\Gamma} e_i$ (again, by purity), whence $e_i\Gamma e_i=e_i Z(\Gamma)$.

	Regardless of which of the two options given in \ref{assumption lambda 2} holds, we now know that $e_i\Gamma e_i=e_i\Lambda e_i$ for all $i\in\{1,\ldots,n\}$, and all of these rings are commutative. Since the $e_i\bar{\Gamma} e_j$ are generated by a single element as an $e_i\bar \Gamma e_i$-module, the $e_i\Gamma e_i$-lattice $e_i\Gamma e_j$ must also be generated by a single element. Because the analogous statement is true for $e_i\Lambda e_j$, and we have $e_i\Gamma e_i=e_i\Lambda e_i$, we must have $e_i\Gamma e_j = a_{i,j} \cdot e_i \Lambda e_j$ for certain $a_{i,j}$ in $Z(A)$. Of course only the image of $a_{i,j}$ in $e_iZ(A) = e_iAe_i$ matters, and we can ask without loss of generality that the $a_{i,j}$ should lie in $Z(A)^\times$. Also, set $a_{i,i}=1$ for all $i\in\{1,\ldots,n\}$. We are free to modify the $a_{i,j}$ by units of $R$. Now consider the simple $Z(A)$-modules $Z_{i,j}$ provided by \ref{assumption lambda 4}. The endomorphism ring of $Z_{i,j}$ is a finite extension $K_{i,j}$ of $K$,  in which the integral closure of $R$ is a totally ramified (since $k=\bar k$) extension $R_{i,j}$ of $R$ with uniformiser $\pi_{i,j}$. Make a choice such that $\pi_{i,j}=\pi_{i',j'}$ whenever $Z_{i,j}= Z_{i',j'}$. Now we can write any element of $K_{i,j}$ as $\pi_{i,j}^z \cdot r$ for some $z\in \Z$ and $r\in R$. Hence we can stipulate that the $a_{i,j}$ should act on $Z_{i,j}$ by multiplication by a power of~$\pi_{i,j}$ for all $i\neq j$.

	Our  assumption \ref{assumption lambda 3} implies that for any $i,j,l\in\{1,\ldots,n\}$ with $i\neq j$ and $j\neq l$  either  $e_i\Lambda e_j\Lambda e_l$ is a pure sublattice of $e_i\Lambda e_l$ or  $e_l\Lambda e_j\Lambda e_i$ is a pure sublattice of $e_l\Lambda e_i$ (or both; note that $\dim_K e_iAe_jAe_l=\dim_K e_lAe_jAe_i$ follows from \ref{assumption lambda 4}). Until further notice let us only consider triples $i,j,l$ such that   $e_i\Lambda e_j\Lambda e_l$ is a pure sublattice of $e_i\Lambda e_l$. Since we have seen that there is an isomorphism between $\bar{\Gamma}$ and $\bar{\Lambda}$ sending $e_r+\pi\Gamma$ to $e_r+\pi\Lambda$ for all $r$, it follows that  $\Gamma$ also satisfies (the analogue of) assumption~\ref{assumption lambda 3}, which means that $e_i\Gamma e_j\Gamma e_l$ is also a pure sublattice of $e_i\Gamma e_l$. This is the same as saying that $a_{i,j} a_{j,l} a_{i,l}^{-1}\cdot e_i\Lambda e_j\Lambda e_l$ is a pure sublattice of $e_i\Lambda e_l$, and therefore
	\begin{equation}\label{eqn ijciih}
		a_{i,j} a_{j,l} a_{i,l}^{-1}\cdot e_i\Lambda e_j\Lambda e_l =  e_i\Lambda e_j\Lambda e_l
	\end{equation} since pure sublattices are determined by their $K$-span.

	For $i=l$ the $Z(\Lambda)$-modules $e_i\Lambda e_j\Lambda e_i$ , $e_i\Lambda e_j$ and $e_j\Lambda e_i$ are isomorphic (they are generated by a single element, which means they are determined by their annihilator, which is the simultaneous annihilator of $e_i$ and $e_j$ in all cases), and therefore $a_{i,j}\cdot a_{j,i} = z$ for some element $z\in Z(\Lambda)$ which acts invertibly on $e_i\Lambda e_j\Lambda e_i$ (w.l.o.g. $z\in Z(\Lambda)^\times$). We can replace $a_{i,j}$ by $z^{-1}\cdot a_{i,j}$ whenever $j>i$, thus making $a_{i,j}\cdot a_{j,i}$ act as the identity. This is compatible with the action on $Z_{i,j}$ that was fixed earlier.

	Now assume that $i,j$ and $l$ are pair-wise distinct. By our assumption, whenever $e_i\Lambda e_j\Lambda e_l$ is non-zero, the element $a_{i,j} a_{j,l} a_{i,l}^{-1}$ acts by multiplication by a power of $\pi_{i,j}$ on it, which in light of equality~\eqref{eqn ijciih} above is only possible if $a_{i,j} a_{j,l} a_{i,l}^{-1}$ acts as the identity. The element $a_{l,j}a_{j,i}a_{l,i}^{-1}$ acts like the inverse of  $a_{i,j} a_{j,l} a_{i,l}^{-1}$, and therefore also acts as the identity.

	The upshot of the above is that
	\begin{equation}
		e_i\Gamma e_j = a_{i,j} \cdot e_i\Lambda e_j \quad \textrm{ and } \quad a_{i,j} a_{j,l} a_{i,l}^{-1}  \textrm{ acts as the identity on $e_i\Lambda e_j \Lambda e_l$}
	\end{equation}
	for all $i,j,l\in\{1,\ldots,n\}$ (not subject to any restrictions), and therefore the map
	\begin{equation}
		\Lambda \longrightarrow \Gamma:\  e_ixe_j\mapsto a_{i,j} \cdot e_ixe_j \quad \textrm{ for all $i,j\in\{1,\ldots,n\}$}
	\end{equation}
	defines an isomorphism which restricts to the identity on $Z(\Lambda)=Z(\Gamma)$, thus showing that $\Lambda$ and $\Gamma$ are conjugate.
\end{proof}

\begin{lemma}\label{lemma pic}
	Let $R$ be a complete discrete valuation ring with algebraically closed residue field $k=R/\pi R$ and field of fractions $K$. Let $\Lambda$  be an $R$-order in a semisimple $K$-algebra $A$,  set $\bar{\Lambda} = \Lambda/\pi\Lambda$ and assume that $\bar{\Lambda}$ is basic. 	Suppose furthermore that $S\subset Z(\Lambda)$ is a complete discrete valuation ring with field of fractions ${L\subset Z(A)}$ such that $\Lambda$ is also an $S$-order. Let $V_1,\ldots, V_n$  ($n\in\N$)  denote representatives for the simple $A$-modules   and let ${\eps_1,\ldots,\eps_n\in Z(A)}$ be the corresponding central primitive idempotents.  Assume that all of the following hold:
	\begin{enumerate}[start=1,label={\bfseries(A\arabic*)}]
		\item\label{assumprion a1} $\Lambda$ fulfils the assumptions \ref{assumption lambda 1}--\ref{assumption lambda 5} of Lemma~\ref{lemma unique lifting}.
		\item\label{assumption centre aut lifts} Every $K$-algebra automorphism of $Z(A)$ lifts to a Morita auto-equivalence of $A$.
		\item\label{assumption decomp determined}
		      If $\Gamma$ is an $R$-order in a $K$-algebra $B$ such that $\Gamma$ is derived equivalent to $\Lambda$, $\bar{\Gamma}=\Gamma/\pi\Gamma\cong \bar{\Lambda}$,   and there is an isometry between $\Gro(B)$ and $\Gro(A)$ sending $\Im D_\Gamma$ to $\Im D_\Lambda$, then there exist isometries $\iota_1:\ \Gro(\bar\Gamma) \xrightarrow{\sim} \Gro(\bar{\Lambda})$ and $\iota_2:\ \Gro(B) \xrightarrow{\sim} \Gro({A})$ sending distinguished bases to distinguished bases such that the following diagram commutes
		      \begin{equation}
			      \xymatrix{
				      \Gro(\bar \Gamma) \ar[r]^{\iota_1} \ar[d]^{D_\Gamma} & \Gro(\bar \Lambda) \ar[d]^{D_\Lambda}\\
				      \Gro(B) \ar[r]^{\iota_2} & \Gro(A),
			      }
		      \end{equation}
		      where $D_\Lambda$ and $D_\Gamma$ are as in Definition~\ref{defi decomp}.
		\item \label{assumption ex gamma} If
		      \begin{equation}
			      \varphi_{A}:\ \Gro(A) \longrightarrow \Gro(A):\ [V_i]\mapsto (-1)^{\tau(i)} \cdot [V_{\sigma(i)}]
		      \end{equation}
		      is a self-isometry, where $\tau:\ \{1,\ldots,n\} \longrightarrow \{\pm 1\}$ and $\sigma \in S_n$,    such that   
		      \begin{equation}
			      \varphi_A(\Im  D_\Lambda) \subseteq \Im D_\Lambda 
		      \end{equation}
		      then there is a  $\gamma \in \Aut_K(Z(A))$ such that $\gamma(\eps_i)=\eps_{\sigma(i)}$ for all ${i\in\{1,\ldots,n\}}$ and $\gamma(Z(\Lambda)) = Z(\Lambda)$.
		\item\label{assumption center s stable} If, for a $\varphi_A$ as in \ref{assumption ex gamma}, there is a $\gamma\in \Aut_L(Z(A))$ such that $\gamma(\eps_i)=\eps_{\sigma(i)}$ for all $i$, then
		      $\gamma(Z(\Lambda))=Z(\Lambda)$.
		\item\label{assumption aut intersect} The images of $\Aut_K(Z(A))\cap \Aut_L(Z(A))$ and $\Aut_K(Z(A))$ in $\Aut(\Gro(A))$ are equal.
		\item\label{assumption autl} The kernel of $\Aut_L(Z(A)) \longrightarrow \Aut(\Gro(A))$ is contained in $\Aut_K(Z(A))$.
	\end{enumerate}
	Define $\mathcal{H}$ as the full preimage of $\Isom(\Gro(A))$ (the group of self-isometries, which depends on $K$)    under the group homomorphism $ 	\TrPic_S(\Lambda) \longrightarrow \Aut(\Gro(A))$. Then $\mathcal H \subseteq \TrPic _R(\Lambda)$, and if we let $\bar{\mathcal H}$ denote the image of $\mathcal H$ in $\TrPic_k(\bar \Lambda)$, we have
	\begin{equation}\label{eqn decomp h trpic}
		\TrPic_k(\bar \Lambda) =  \Pic_k(\bar \Lambda) \cdot \bar{\mathcal H}.
	\end{equation}
\end{lemma}
\begin{proof}
	Let $\bar X^\xbullet$ be an arbitrary element of $\TrPic_k(\bar{\Lambda})$, and let $\bar T^\xbullet \in \mathcal K^b(\bar{\Lambda} \projL)$ be its restriction to the left. By \cite[Proposition 3.1]{RickardLiftTilting} there is a tilting complex $T^\xbullet \in \mathcal K^b(\Lambda\projL)$ such that $k\otimes_R T^\xbullet \cong \bar T^\xbullet$. By \cite[Theorem 3.3]{RickardLiftTilting} the endomorphism algebra  $\Gamma= \End_{\mathcal D^b(\Lambda)}( T^\xbullet)^{\opp}$ is again an $R$-order, and $k\otimes_R\Gamma \cong \bar \Lambda$. By \cite[Proposition 3.1]{RickardDerEqDerFun} there is a two-sided tilting complex $X^\xbullet$ in $\mathcal D^b(\Lambda^{\opp}\otimes_R \Gamma)$ whose restriction to the left is $T^\xbullet$. We get a commutative diagram
	\begin{equation}\label{eqn double sq}
		\xymatrix{
		\Gro(\bar \Lambda) \ar[r]^{\varphi_{\bar X}} \ar[d]^{D_\Lambda} & \Gro(\bar \Gamma) \ar[d]^{D_\Gamma}\ar[r]^{\iota_1}
		& \Gro(\bar \Lambda) \ar[d]^{D_\Lambda} \\
		\Gro(A) \ar[r]^{\varphi_{KX}} & \Gro(B)\ar[r]^{\iota_2} & \Gro(A),
		}
	\end{equation}
	where the left hand square comes from Proposition~\ref{prop compat center decomp}   and the right hand square comes from assumption~\ref{assumption decomp determined}. 

	Now assumption~\ref{assumption ex gamma} implies that there is an automorphism of $Z(A)$ inducing the permutation on central primitive idempotents which corresponds to $\iota_2\circ \varphi_{KX}$, and mapping $Z(\Lambda)$ into itself. By Proposition~\ref{prop compat center decomp} there exists an algebra homomorphism $\gamma_{KX}$ corresponding to $\varphi_{KX}$ mapping $Z(\Lambda)$ into $Z(\Gamma)$, which tells us that there is an algebra homomorphism $\gamma_2:\ Z(B) \longrightarrow Z(A)$ which maps the central primitive idempotent corresponding to $[S]$ (for some simple $B$-module $S$), to the central primitive idempotent corresponding to $\iota_2([S])$, and $\gamma_2(Z(\Gamma))=Z(\Lambda)$. By considering a Morita equivalence between $A$ and $B$, and applying assumption~\ref{assumption centre aut lifts} to the isomorphism of centres induced by the Morita equivalence followed by $\gamma_2$, we see that there is in fact a Morita equivalence between $B$ and $A$, given by a bimodule $M$, such that $[W\otimes_B M] = \iota_2([W])$ for all $B$-modules $W$, and the induced isomorphism of centres $\gamma_{KM}$ is~equal~to~$\gamma_2$.

	Now note that $\iota_1([\bar{\Gamma}])=[\bar \Lambda]$, since $\bar \Lambda$ is assumed to be basic which means that $[\bar\Lambda]$ and $[\bar\Gamma]$ are just the sums over the distinguished bases. Moreover, $D_\Gamma([\bar{\Gamma}])=[B]$ and $D_\Lambda([\bar{\Lambda}])=[A]$. Hence commutativity of the rightmost square in \eqref{eqn double sq} implies that $[B\otimes_B M] =[A]$, which shows that $M$ is induced by an algebra isomorphism $\alpha:\ B \longrightarrow A$. That is, $M\cong {_\alpha A}$. Set $\Gamma'=\alpha(\Gamma)\subseteq A$. We have $\alpha|_{Z(B)}=\gamma_{KM}=\gamma_2$, which shows that $\alpha(Z(\Gamma))=Z(\Lambda)$. It follows that $Z(\Gamma')=Z(\Lambda)$. Applying Proposition~\ref{prop compat center decomp} to the functor $-\otimes_{\Gamma} {_\alpha \Gamma'}$   (and taking into account that it sends simple modules to simple modules) gives us an isometry $\iota_3:\ \Gro(\bar \Gamma)\longrightarrow \Gro(\Gamma'/\pi\Gamma')$ mapping the distinguished basis to the distinguished basis such that $\iota_2 \circ D_\Gamma =  D_{\Gamma'} \circ \iota_3$. Combining this with rightmost square in \eqref{eqn double sq} gives $D_{\Gamma'}=D_\Lambda\circ \iota_1\circ  \iota_3^{-1}$.  This means that the image of the distinguished basis of $\Gro(\Gamma'/\pi\Gamma')$ under $D_{\Gamma'}$ is the same (up to reordering) as the image of the distinguished basis of $\Gro(\bar{\Lambda})$ under $D_\Lambda$. It follows that the $K$-spans of the projective indecomposable $\Gamma'$-modules fulfil condition \ref{assumption gamma 1} in Lemma~\ref{lemma unique lifting}. Since we have checked  \ref{assumption gamma 2} and \ref{assumption gamma 3} already, we can conclude by Lemma~\ref{lemma unique lifting} that $\Gamma'\cong \Lambda$. We can hence compose our original two-sided tilting complex $X^\xbullet$ with a bimodule corresponding to an isomorphism between $\Gamma$ and $\Lambda$, to obtain an element of $\TrPic_R(\Lambda)$ whose restriction to the left is $T^\xbullet$. It is well-known that any two two-sided tilting complexes over $\bar{\Lambda}$ restricting to $\bar T^\xbullet$ differ from one another only by an element of $\Pic_k(\bar\Lambda)$ (see \cite[Proposition 2.3]{RouquierZimmermann}). We have thus shown that
	\begin{equation}
		\TrPic_k(\bar \Lambda) = \xoverline{\TrPic_R(\Lambda)} \cdot \Pic_k(\bar \Lambda).
	\end{equation}

	Now suppose that $Y^\xbullet \in \mathcal H \subseteq \TrPic_S(\Lambda)$. By Proposition~\ref{prop compat center decomp} $Y^\xbullet$ induces maps ${\varphi_{LY}:\ \Gro(A) \longrightarrow \Gro(A)}$ and $\gamma_{LY}\in \Aut_L(Z(A))$.  By definition of $\mathcal H$ the  map $\varphi_{LY}$ is an isometry of $\Gro(A)$ (equipped with the bilinear form coming from the $K$-algebra structure of $A$, rather than the $L$-algebra structure for which this would be trivial), and it maps $\Im D_\Lambda$ into itself. Therefore, by assumption \ref{assumption ex gamma}, there is a map $\gamma'\in \Aut_K(Z(A))$ which induces the same permutation of the $\eps_i$'s as $\gamma_{LY}$. By assumption~\ref{assumption aut intersect} there must be a $\gamma''\in \Aut_K(Z(A))\cap \Aut_L(Z(A))$ which also induces the same permutation, meaning that $\gamma''\circ \gamma_{LY}^{-1}$ is $L$-linear and acts trivially on $\Gro(A)$. By assumption~\ref{assumption autl} it follows that $\gamma''\circ \gamma_{LY}^{-1}$ is $K$-linear, but then $\gamma_{LY}$ must be $K$-linear as well. By Remark~\ref{remark detect trpic linearity} it follows that  $Y^\xbullet$ lies in (or, is isomorphic to an element of)~$\TrPic_R(\Lambda)$.

	Now let us consider an arbitrary $X^\xbullet \in \TrPic_R(\Lambda)$. By assumption~\ref{assumption aut intersect} there must be a $K$- and $L$-linear automorphism $\gamma':\ Z(A)\longrightarrow Z(A)$ inducing the same permutation of idempotents as $\gamma_{KX}$. By assumption~\ref{assumption center s stable} it follows that $\gamma'(Z(\Lambda))=Z(\Lambda)$. Set $\gamma''=\gamma_{KX}^{-1} \circ \gamma'$. Then $\gamma''$ is $K$-linear, acts trivially on $\Gro(A)$, and $\gamma''(Z(\Lambda))=Z(\Lambda)$. By assumption~\ref{assumption centre aut lifts} it follows that there is a Morita auto-equivalence of $A$ extending $\gamma''$, and since $\gamma''$ acts trivially on $\Gro(A)$ this must actually be induced by an $\alpha \in \Aut_K(A)$. That is, $\alpha$ induces the identity on $\Gro(A)$ and restricts to $\gamma''$. So, if we set $\Gamma=\alpha(\Lambda)$, the assumptions \ref{assumption gamma 1}--\ref{assumption gamma 3} of Lemma~\ref{lemma unique lifting} are satisfied, as well as the one needed for conjugacy instead of isomorphism, and therefore there is a central automorphism $\beta$ of $A$ such that $\beta(\alpha(\Lambda))=\Lambda$. Hence $Y^\xbullet = X^\xbullet \otimes_{\Lambda} {_{\beta\circ\alpha}\Lambda}$ induces the automorphism $\gamma_{KY}=\gamma'$ on $Z(A)$, which is $L$-linear. Hence $Y^\xbullet \in \TrPic_S(\Lambda)$  by Remark~\ref{remark detect trpic linearity}, and $Y^\xbullet$ clearly satisfies the condition defining $\mathcal H$. Therefore $X^\xbullet \in \mathcal H \cdot \Pic_R(\Lambda)$. We have now established that $\TrPic_k(\bar \Lambda) = \bar{\mathcal H} \cdot \Pic_k(\bar \Lambda)$, and the equality~\eqref{eqn decomp h trpic} follows by inverting.
\end{proof}

\begin{lemma}[Additional structure]\label{lemma extra structure}
	Assume that all assumptions of Lemma~\ref{lemma pic} hold.
	\begin{enumerate}
		\item If
		      \begin{enumerate}[start=1,label={\bfseries(B\arabic*)}]
			      \item\label{assumption B1} $\bar{\Lambda}$ is silting-connected
		      \end{enumerate}
		      then $\PicS_k(\bar{\Lambda})$ is normal in $\TrPic_k(\bar{\Lambda})$.
		\item If \ref{assumption B1} holds and in addition
		      \begin{enumerate}[start=2,label={\bfseries(B\arabic*)}]
			      \item\label{assumption B2} the kernel of the action of $\Pic_S({\Lambda})$ on $\Gro(\Lambda) \cong \Gro(\bar{\Lambda})$ is trivial, and
			      \item\label{assumption B3} the image of $\Pic_k(\bar \Lambda)$  in $\Aut(\Gro(\bar{\Lambda}))$ is contained in the image of $\Pic_S({\Lambda})\cap \Pic_R(\Lambda)$ in $\Aut(\Gro(\bar{\Lambda}))$,
		      \end{enumerate}
		      then $\bar{ \mathcal H }\cong \mathcal H$ and we can write
		      \begin{equation}
			      \TrPic_k(\bar \Lambda) = \PicS_k(\bar \Lambda) \rtimes \mathcal H.
		      \end{equation}
		      Note that if $Z(\bar\Lambda)=Z(\Lambda)/\pi Z(\Lambda)$ and $\Picent(\bar{\Lambda})\cap \PicS_k(\bar \Lambda) =1$, then $\TrPicent(\Lambda)\unlhd \mathcal H$ acts trivially on $ \PicS_k(\bar \Lambda)$.
		\item
		      If $\TrPic_k(\bar \Lambda) = \PicS_k(\bar \Lambda) \cdot \bar{\mathcal H}$ (e.g. due to \ref{assumption B1}--\ref{assumption B3} holding), and in addition
		      \begin{enumerate}[start=4	,label={\bfseries(B\arabic*)}]
			      \item\label{assumption b4} $\PicS_k(\bar \Lambda) \cap \Picent(\bar{\Lambda})\subseteq \bar{\mathcal H}$, and
			      \item\label{assumption b5} the images of $\bar{\mathcal H}$ and $\PicS_k(\bar{\Lambda})$ in $\Aut_k(Z(\bar{\Lambda}))$ intersect trivially,
		      \end{enumerate}
		      then
		      \begin{equation}\label{eqn dudfhidh}
			      \TrPicent(\bar \Lambda)=  \Ker(\bar{\mathcal H} \longrightarrow \Aut_k(Z(\bar{\Lambda}))).
		      \end{equation}
	\end{enumerate}
\end{lemma}
\begin{proof}
	If \ref{assumption B1} holds then any two silting complexes over $\bar{\Lambda}$ are linked by a finite sequence of irreducible mutations. It is clear from the definition of irreducible silting mutations  by minimal left or right approximations (see \cite[Definition-Theorem 2.3]{AiharaTiltingConnSymm}) that if  a Morita auto-equivalence fixes the isomorphism classes of all indecomposable summands of a silting complex, then it also fixes all indecomposable summands of an irreducible mutation of that complex. Hence, in the connected case, the group $\PicS_k(\bar{\Lambda})$ fixes all isomorphism classes of silting complexes, since it fixes the projective indecomposable modules and every silting complex is connected via mutation to their direct sum. In particular, the restriction  of $X^\xbullet \otimes_\Lambda M$ to the right is isomorphic to the restriction of $X^\xbullet$ to the right, for any $X^\xbullet \in \TrPic_k(\bar{\Lambda})$ and $M\in\PicS_k(\bar \Lambda)$. It follows that  $X^\xbullet \otimes_\Lambda M\otimes_\Lambda (X^{-1})^\xbullet$ lies in $\Pic_k(\bar \Lambda)$. Since $M$ acts trivially on $\Gro(\bar{\Lambda})$ its conjugate will also act trivially on $\Gro(\bar \Lambda)$, that is,  $X^\xbullet \otimes_\Lambda M\otimes_\Lambda (X^{-1})^\xbullet$ lies in $\PicS_k(\bar \Lambda)$. This proves the first part of the assertion.

	The kernel of the map  ${\TrPic_R(\Lambda) \longrightarrow \TrPic_k(\bar{\Lambda})}$ is contained in $\Pic_R(\Lambda)$ (see \cite[Lemma 3.4]{RouquierZimmermann}). Since, by the assumption \ref{assumption B2}, the group $\Pic_S(\Lambda)$ acts faithfully on $\Gro({\Lambda})$, and ${\mathcal H \cap \Pic_R(\Lambda) \subseteq \Pic_S(\Lambda)}$, it follows that $\mathcal H$ embeds into $\TrPic_k(\bar \Lambda)$, that is, $\mathcal H \cong \bar{\mathcal H}$. By the assumption \ref{assumption B3} it follows that $\TrPic_k(\bar{\Lambda})= \PicS_k(\bar{\Lambda})\cdot \mathcal H$, and by faithfulness of $\mathcal H\cap \Pic_R(\Lambda)$ on $\Gro(\bar{\Lambda})$ also $\PicS_k(\bar{\Lambda})\cap \mathcal H=\{1\}$.  This proves $	\TrPic_k(\bar \Lambda) = \PicS_k(\bar \Lambda) \rtimes \mathcal H$.  If $\Picent(\bar{\Lambda})\cap \PicS_k(\bar \Lambda) =1$ then an element of $\PicS_k(\bar \Lambda)$ is determined by the automorphism of $Z(\bar{\Lambda})$ it induces, which implies the last bit of the assertion of the second part.

	Let us now prove the third part.  Assumption~\ref{assumption b5} implies that
	\begin{equation}
		\TrPicent(\bar{\Lambda}) = {(\PicS_k(\bar{\Lambda}) \cap \TrPicent(\bar{\Lambda})) \cdot (\bar{\mathcal H} \cap \TrPicent(\bar{\Lambda}))},
	\end{equation}
	and assumption~\ref{assumption b4} ensures that $\TrPicent(\bar{\Lambda})\cap \PicS_k(\bar{\Lambda}) \subseteq \bar{\mathcal H}$. It follows that $\TrPicent(\bar{\Lambda}) \subseteq \bar{\mathcal H}$, from which one infers \eqref{eqn dudfhidh}.
\end{proof}

\begin{prop}\label{prop silting discrete}
	Assume $k$ is an algebraically closed field, and let $\bar \Lambda$ be a $k$-algebra of dihedral, semi-dihedral or quaternion type in the sense of Erdmann~\cite{TameClass}. Then $\bar \Lambda$ is silting-discrete.
\end{prop}
\begin{proof}
	The class of algebras  of dihedral, semi-dihedral or quaternion type is, by definition, closed under derived (and even stable) equivalences. All of these algebras are symmetric as well, which implies that tilting and silting complexes coincide. By \cite[Theorem 2.4]{AiharaMizuno} it suffices to show that all algebras in the derived equivalence class of $\bar{\Lambda}$ are $2$-tilting finite, and this was verified in \cite[Theorem 16]{TauRigid}.
\end{proof}

We will now verify the assumptions of Lemmas~\ref{lemma pic}~and~\ref{lemma extra structure} for the lifts of the algebras $\mathcal Q(3K)^{a_1,a_2,a_3}$ in characteristic two constructed in Proposition~\ref{prop lift weighted surface}, as well as the orders lifting twisted Brauer graph algebras from Proposition~\ref{prop biserial order} (subject to mild assumptions on the Brauer graph). This is straight-forward in principle, but requires us to describe these $k\sbb{X}$-orders more explicitly. A large chunk of the verification for the lifts of $\mathcal Q(3K)^{a_1,a_2,a_3}$ is already contained in Remark~\ref{remark structure q3k} below. This remark also highlights that these $k\sbb{X}$-orders are equicharacteristic analogues of the lifts over an extension $\OO$ of the $2$-adic integers coming from blocks of quaternion defect over $\OO$ (e.g. they share the same decomposition matrix). The same is true for Brauer tree algebras. However, for other Brauer graph algebras like $\mathcal D(3K)^{a_1,a_2,a_3}$ our lifts over $k\sbb{X}$ do have a different decomposition matrix and a centre of a different dimension than the lifts over $\OO$ coming from modular representation theory.

\begin{remark}[Structure of the lifts of $\mathcal Q(3K)^{a_1,a_2,a_3}$]\label{remark structure q3k}
	Assume $k$ is an algebraically closed field of characteristic two. Let $Q$, $f$, $\mm$, $\tt$  be as in Proposition~\ref{prop semimimple q3k}, set $c_\alpha=1$ for all $\alpha\in Q_1$, and let  $\mm':\ Q_1/\langle g \rangle \longrightarrow \Z_{>0}$ be arbitrary. Consider the $k\sbb{X}$-order $\Lambda=\Gamma(Q,f,\mm, \cc, \tt, \emptyset; \mm')$ constructed in Proposition~\ref{prop lift weighted surface}. Let $A$ denote the $k\sqq{X}$-span of~$\Lambda$. By definition, $\Lambda$ is a pullback
	\begin{equation}\label{eqn pullback gamma}
		\xymatrix{
			\frac{\widehat{k[X]Q}}{\phantom{.}\overline{(\alpha f(\alpha) - X\bar \alpha,\
				\alpha g(\alpha)-\bar \alpha g(\bar{\alpha})\ | \ \alpha \in Q_1)}\phantom{.} \ar@{->>}[rr]^{\varphi}} && \frac{\widehat{k[X]Q}}{\phantom{.}\overline{(\alpha f(\alpha),\
				\alpha g(\alpha)-\bar \alpha g(\bar{\alpha}),\ X+z \ | \ \alpha \in Q_1)} \phantom{.}}\\
			\Lambda	 \ar[u]\ar[rr] && \frac{\widehat{k[X]Q}}{\phantom{.}\overline{(\alpha f(\alpha),\ X+z \ | \  \alpha \in Q_1)}{\phantom{.}}} \ar@{->>}[u]_{\nu}
		}
	\end{equation}
	where $z=\sum_{i=1}^3 (\alpha_i g(\alpha_i) + g(\alpha_i)\alpha_i	)^{m'_{\alpha_i}}$ and $\varphi,\nu$ are the natural surjections. Let us denote the orders in the top left and bottom right corners of this diagram by $\Lambda_1$ and $\Lambda_2$, respectively, and the $k\sbb{X}$-algebra in the top right corner by $\bar \Lambda_0$. We will view $\Lambda$ as a subset of $\Lambda_1\oplus \Lambda_2$. Note that $\Lambda_1$ and $\Lambda_2$ can be described pictorially as
	\begin{equation}
		\vcenter{\hbox{
				\begin{tikzpicture}[font=\tiny,scale=0.55, every left delimiter/.style={xshift=.5em}, every right delimiter/.style={xshift=-.5em},]
					\matrix (mat0) [matrix of math nodes, left delimiter=., right delimiter=.] {
						k\sbb{X} \\
					};

					\matrix (mat1) [matrix of math nodes, left delimiter=., right delimiter=.] at (1.5,0) {
						k\sbb{X} \\
					};

					\matrix (mat2) [matrix of math nodes, left delimiter=., right delimiter=.] at (3,0) {
						k\sbb{X} \\
					};

					\matrix (mat3) [matrix of math nodes, left delimiter=(, right delimiter=)] at (6.8,0) {
						k\sbb{X} & (X)                      & (X)                      \\
						(X)                      & k\sbb{X} & (X)                      \\
						(X)                      & (X)                      & k\sbb{X} \\
					};

					\draw[thick] (mat0-1-1)   to [bend left] node [midway,above]{$\scriptscriptstyle k\sbb{X}/X^2$} (mat3-1-1);
					\draw[thick] (mat1-1-1) to [bend left=20] node [pos=0.1, above]{$\scriptscriptstyle k\sbb{X}/X^2$} (mat3-2-2);
					\draw[thick] (mat2-1-1) to [bend right] node [midway,below]{$\scriptscriptstyle k\sbb{X}/X^2$} (mat3-3-3);
				\end{tikzpicture}}}
		\textrm{ and }
		\vcenter{\hbox{
				\begin{tikzpicture}[font=\tiny,scale=1.2,  every left delimiter/.style={xshift=.5em}, every right delimiter/.style={xshift=-.5em},]
					\matrix (mat0) [matrix of math nodes, left delimiter=(, right delimiter=)] {
						k\sbb{X^{\frac{1}{u}}}          & k\sbb{X^{\frac{1}{u}}} \\
						(X^{\frac{1}{u}}) & k\sbb{X^{\frac{1}{u}}} \\
					};

					\matrix (mat1) [matrix of math nodes, left delimiter=(, right delimiter=)] at (2.5,0) {
						k\sbb{X^{\frac{1}{v}}}          & k\sbb{X^{\frac{1}{v}}} \\
						(X^{\frac{1}{v}}) & k\sbb{X^{\frac{1}{v}}} \\
					};

					\matrix (mat2) [matrix of math nodes, left delimiter=(, right delimiter=)] at (5,0) {
						k\sbb{X^{\frac{1}{w}}}          & k\sbb{X^{\frac{1}{w}}} \\
						(X^{\frac{1}{w}}) & k\sbb{X^{\frac{1}{w}}} \\
					};

					\draw[thick] (mat0-1-1) to [bend left=23] node [midway,above]{$\scriptscriptstyle k$} (mat1-1-1);
					\draw[thick] (mat1-2-2) to [bend right] node [midway,above]{$\scriptscriptstyle k$} (mat2-1-1);
					\draw[thick] (mat2-2-2) to [bend left=17] node [midway,below]{$\scriptscriptstyle k$} (mat0-2-2);
				\end{tikzpicture}}},
	\end{equation}
	where $u,v,w$ are the multiplicities $m'_{\alpha_1}$, $m'_{\alpha_2}$ and $m'_{\alpha_3}$. The arcs indicate that the entries linked by them must have the same image in the ring labelling the arc (either $k\sbb{X}/(X^2)$ or $k$). This information can be extracted from Proposition~\ref{prop semimimple q3k} and Proposition~\ref{prop biserial order}, respectively. Of course, the description above mainly serves to illustrate, the facts we are actually going to use are the following (easily obtained from Propositions~\ref{prop semimimple q3k}~and~\ref{prop biserial order} and the fact that $\bar \Lambda=\Lambda/X\Lambda$ is isomorphic to $\Lambda(Q,f,\mm + \mm', \cc, \tt, \emptyset)$):
	\begin{enumerate}
		\item $(e_1,e_1)$, $(e_2,e_2)$ and $(e_3,e_3)$ form a full set of orthogonal primitive idempotents~in~ $\Lambda\leq \Lambda_1\oplus\Lambda_2$. When it is unambiguous we will write ``$e_i$'' instead of $(e_i,e_i)$.
		\item $\Lambda_1$ and $\Lambda_2$ are both orders in semisimple $k\sqq{X}$-algebras which are Morita equivalent to their centres (i.e. no division algebras occur), and therefore so is $\Lambda$. It follows that $\Lambda$ satisfies assumption~\ref{assumption centre aut lifts} of Lemma~\ref{lemma pic} with $R=k\sbb{X}$.
		\item The central primitive idempotents in the $k\sqq{X}$-span of $\Lambda_1$ are $\eps_i =e_i-X^{-2}\alpha_i\beta_{i+1}$ for $i\in \{1,2,3\}$  (set $\beta_4=\beta_1$) and $\eps_4 = 1-\eps_1-\eps_2-\eps_3$. Those in the $k\sqq{X}$-span of $\Lambda_2$ are $\eps_{4+i}=X^{-1} (\alpha_i\beta_{i+1}+\beta_{i+1}\alpha_{i})^{m'_{\alpha_i}}$ for $i\in\{1,2,3\}$.
		\item\label{point iso centre} The elements $\eta_i=(\eps_i,0)$ for $1\leq i \leq 4$ and $\eta_i=(0,\eps_i)$ for $5\leq i \leq 7$ form a full set of primitive idempotents in $A$. We have
		      \begin{equation}
			      \eta_i Z(A)\cong k\sqq{X} \textrm{ for $1\leq i \leq 4$}, \quad  \eta_{4+i}Z(A) \cong k\sqq{X^{1/m'_{\alpha_i}}} \textrm{ for $1\leq i \leq 3$}
		      \end{equation}
		      as $k\sqq{X}$-algebras, where $\eta_{4+i}X^{1/m_{\alpha_i}}$ has preimage $ \alpha_i\beta_{i+1}+\beta_{i+1}\alpha_{i}$ in $Z(\Lambda_2)$. In particular,
		      \begin{equation}
			      \dim_{k\sqq{X}}(Z(A))=4+m_{\alpha_1}'+m_{\alpha_2}'+m_{\alpha_3}'=\dim_k(Z(\bar{\Lambda})),
		      \end{equation}
		      where the second equality is obtained by counting a basis of $Z(\bar{\Lambda})$ (alternatively see the appendix of \cite{TameClass}). It follows that $Z(\Lambda)$ surjects onto $Z(\bar{\Lambda})$, proving assumption \ref{assumption lambda 2} of Lemma~\ref{lemma unique lifting}.
		\item The bilinear form on $\Gro(A)$ has Gram-matrix
		      \begin{equation}
			      \operatorname{diag}(1,1,1,1,m'_{\alpha_1},m'_{\alpha_2},m'_{\alpha_3})
		      \end{equation}
		      and the map  $D_\Lambda:\ \Gro(\bar\Lambda) \longrightarrow \Gro(A)$ has matrix
		      \begin{equation}\label{eqn matrix D}
			      D=\begin{pmatrix}
				      1 & 0 & 0 \\
				      0 & 1 & 0 \\
				      0 & 0 & 1 \\
				      1 & 1 & 1 \\
				      1 & 1 & 0 \\
				      0 & 1 & 1 \\
				      1 & 0 & 1 \\
			      \end{pmatrix}^{\top} \quad\textrm{(acting on row vectors)}
		      \end{equation}
		      with respect to the distinguished bases (note that the ordering of the idempotents $\eta_1,\ldots, \eta_7$ induces an order on the basis of $\Gro(A)$ as well). In fact, $e_iA$ is the direct sum of $k\sqq{X}\otimes_{k\sbb{X}}e_i\Lambda_1$ and $k\sqq{X}\otimes_{k\sbb{X}}e_i\Lambda_2$, so this follows from the description of these two algebras.
		\item Each $e_i\bar{\Lambda}e_j$ for $i\neq j$ is generated by the unique arrow $\alpha \in Q_1$ pointing from $e_i$ to $e_j$. In fact, all other non-zero paths in $\bar{\Lambda}$ with the same source and target are  $(\alpha g(\alpha))^l\alpha = (\alpha g(\alpha) + \bar \alpha g(\bar{\alpha}))^l\alpha $ for some $l\in \N$, and the element $\alpha g(\alpha) + \bar \alpha g(\bar{\alpha})$ is central. That is, $\bar \Lambda$ satisfies assumption~\ref{assumption lambda 1} of Lemma~\ref{lemma unique lifting}.
		\item If $i,j,l\in\{1,2,3\}$ are pair-wise distinct, the $\dim_k e_i\bar{\Lambda}e_j\bar{\Lambda}e_l=1$ (by considering a basis), which is equal to $\dim_{k\sqq{X}}(e_iAe_jAe_l)$ by \eqref{eqn matrix D}. It also follows that $\Lambda$ satisfies assumption~\ref{assumption lambda 4} if we set all $Z_{i,j}$ equal to the simple $A$-module corresponding to $\eta_4$.

		      Moreover, if $i\neq j\in \{1,2,3\}$, then $e_i\bar{\Lambda}e_j\bar{\Lambda}e_i$ has the same dimension as $e_i\bar{\Lambda}e_j$ (just by checking that multiplying this by the unique arrow from $e_j$ to $e_i$ does not annihilate any elements), which shows that $\dim_k e_i\bar{\Lambda}e_j\bar{\Lambda}e_i=\dim_{k\sqq{X}}e_iAe_jAe_i$. In particular, $\Lambda$ satisfies  assumption~\ref{assumption lambda 3} of Lemma~\ref{lemma unique lifting}.
		\item Assumption~\ref{assumption lambda 5} of Lemma~\ref{lemma unique lifting} is trivially satisfied for $\Lambda$ since there is even an automorphism of $\Lambda$ (instead of $A$) inducing the desired permutation of idempotents.
		\item\label{point basis centre} The diagram \eqref{eqn pullback gamma} restricts to a  pullback diagram on the centres, from which one sees that $Z(\Lambda)\subset Z(A)$ is spanned as a  $k\sbb{X}$-lattice by
		      \begin{equation}\label{eqn basis centre}
			      \begin{gathered}
				      \displaystyle (1,1),\\ \displaystyle \phantom{X^{X^{X^X}}} \left(X^2(\eps_1+\eps_2), \eps_{5}X^{1/m'_{\alpha_1}}\right),
				      \left(X^2(\eps_2+\eps_3), \eps_{6}X^{1/m'_{\alpha_2}}\right), \left(X^2(\eps_1+\eps_3), \eps_{7}X^{1/m'_{\alpha_3}}\right),\\
				      \displaystyle \phantom{X^{X^{X^X}}} \left(X^3 \eps_1,0\right),  \left(X^3 \eps_2,0\right), \left(X^3 \eps_3,0\right),\\
				      \displaystyle\phantom{\left(X^{X^{X^X}}\right)} \left(0,\eps_5 X^{i_1/m'_{\alpha_1}}+\eps_6 X^{i_1/m'_{\alpha_2}}+\eps_7 X^{i_1/m'_{\alpha_3}}\right), \\
				      \displaystyle\phantom{\left(X^{X^{X^X}}\right)}
				      \left(0,\eps_6 X^{(i_2+1)/m'_{\alpha_2}}\right), \left(0,\eps_7 X^{(i_3+1)/m'_{\alpha_3}}\right) \quad \textrm{ for $1\leq i_j \leq m'_{\alpha_j}$}.
			      \end{gathered}
		      \end{equation}
		      One then verifies that the $\xoverline{k\sqq{X}}$-algebra automorphism $\alpha_{\sigma}$ of $\bigoplus_{i=1}^7 \xoverline{k\sqq{X}} \eta_i$ sending $\eta_i$ to $\eta_{\sigma(i)}$ restricts to $Z(\Lambda)$ for all
		      \begin{equation}\label{eqn group sigma}
			      \sigma \in \langle (1,2,3)(5,6,7),\ (1,2)(6,7),\ (1,2)(3,4) \rangle_{(1,1,1,1,m'_{\alpha_1},m'_{\alpha_2}, m'_{\alpha_3})},
		      \end{equation}
		      where the subscript indicates a stabiliser ($S_7$ acting on $\Z^7$ simply by permutation).

		      The only part of this that is not immediate from the definition is that $\alpha_{(1,2)(3,4)}$ maps $Z(\Lambda)$ into itself, specifically the third and fourth given generators (the others just get permuted). To verify this one just needs the fact that $X^2(\eps_2+\eps_3) = X^2(\eps_1+\eps_4) +X^2$, and and $X^2(\eps_1+\eps_3) = X^2(\eps_2+\eps_4) +X^2$.   
		\item\label{point self isom} In the case where  $m_{\alpha_i}'=1$ for all $i\in\{1,2,3\}$ the group of self-isometries of $\Gro(A)$ which map
		      $\Im D_\Lambda$ into itself was determined in \cite[Proposition 1.1]{HolmKessarLinckelmann}  to be $\langle (1,2,3)(5,6,7),\ (1,2)(6,7),\ (3,-4)(6,-7),\ -\id \rangle$ (this is adapted to our labelling of simple $A$-modules)  . The corresponding group of self-isometries of $\Gro(A)$ for an arbitrary choice of  $m_{\alpha_i}'$'s must be contained in this group, and it is characterised by the fact that its image in $S_7$ stabilises $(1,1,1,1,m'_{\alpha_1},m'_{\alpha_2}, m'_{\alpha_3})\in \Z^7$. Therefore the group of self-isometries of $\Gro(A)$ which preserve $\Im D_\Lambda$  maps onto the group of which $\sigma$ is an element in equation~\eqref{eqn group sigma}. It follows that \ref{assumption ex gamma} of Lemma~\ref{lemma pic} is satisfied for this choice of $\Lambda$ and $R=k\sbb{X}$.
		\item Assume $\Gamma$ is a $k\sbb{X}$-order in a $k\sqq{X}$-algebra $B$ Morita-equivalent to $A$ such that $\bar{\Gamma}=\Gamma/X\Gamma \cong \bar{\Lambda}$. Since $\dim_{k\sqq{X}} Z(B)=\dim_{k\sqq{X}} Z(A)=\dim_{k} Z(\bar{\Lambda})$ it follows that $Z(\Gamma)$ surjects onto $Z(\bar{\Gamma})$, and since $e\bar{\Gamma}e=eZ(\bar{\Gamma})$ for any primitive idempotent $e\in\bar{\Gamma}$ it follows that  $f\Gamma f=Z(\Gamma)f$ and $fBf=fZ(B)$ for any primitive idempotent $f\in \Gamma$. What this means is that the corresponding projective indecomposable module $f\Gamma$ spans a multiplicity-free $B$-module. That is, the entries of the matrix $D\in M_{3\times 7}(\Z)$ representing $D_\Gamma$ (with respect to the distinguished bases) has entries bounded by one. Now consider the equation
		      \begin{equation}\label{eqn yyrtt5}
			      D\cdot \operatorname{diag}(1,1,1,1,m_{\alpha_1}',m_{\alpha_2}',m_{\alpha_3}') \cdot D^\top = C_{\bar{\Gamma}}
		      \end{equation}
		      where $C_{\bar{\Gamma}}$ denotes the Cartan matrix of $\bar{\Gamma}$ (affording the bilinear form on $\Gro(\bar{\Gamma})$). If we define a set $I=I_1\uplus\ldots\uplus I_7$, where $|I_1|=\ldots=|I_4|=1$ and $|I_{4+i}|=m'_{\alpha_i}$ for $i\in\{1,2,3\}$, then we can define sets $R_{i}=\bigcup \{ I_j \ | \ 1\leq j \leq 7 \textrm{ and } D_{i,j}=1 \}$ for $i\in\{1,2,3\}$. Clearly these sets determine $D$, and we can interpret the entries on the left hand side of equation~\eqref{eqn yyrtt5} as the cardinalities $|R_i\cap R_j|$ for $i,j\in\{1,2,3\}$. Note that $D$ has no zero-columns, which implies that $|R_1\cup R_2\cup R_3|=|I|$. By the inclusion-exclusion principle the cardinalities of all possible intersections of $R_i$'s and complements of $R_i$'s are now determined. In particular, by comparing to the matrix in \eqref{eqn matrix D} (which also satisfies \eqref{eqn yyrtt5}) we easily see that $|I_1\cap I_2 \cap I_3|=1$ and $|I_i \cap (I\setminus \bigcup_{j\neq i} I_j)|=1$, which shows that, up to a permutation $\sigma\in S_7$ such that $|I_i|=|I_{\sigma(i)}|$ applied to the columns of $D$ (which corresponds to composing with an isometry of $\Gro(B)$ stabilising the distinguished basis), the first  four columns of $D$ are as in the matrix given in~\eqref{eqn matrix D}, that is, $R_i \subseteq I_i\cup I_4\cup I_5\cup I_6 \cup I_7$ for all $i$.

		      Now, for $(i,j)\in \{(1,2),(2,3),(3,1)\}$ we have $|R_i\cap R_j|=1+m_{\alpha_i}'$. If we choose $i$ such that $m_{\alpha_i}'$ is minimal amongst the values of $\mm'$, then, after potentially applying a permutation which fixes the cardinalities of the $I_i$'s,  $R_i\cap R_j$ must be equal to $I_4\cup I_{4+i}$. But then, again up to an admissible permutation,  $R_i=I_1\cup I_4\cup I_{4+i} \cup I_{4+\sigma(i)}$ and $R_j=I_1\cup I_4\cup I_{4+j} \cup I_{4+\sigma(j)}$ (where $\sigma=(3,2,1)$), since that is again the only union of $I_i$'s with the right cardinality. If $i'$ is the unique element of $\{1,2,3\}\setminus\{i,j\}$, then we just have to consider the intersections $R_i\cap R_{i'}$ and $R_j\cap R_{i'}$ (whose cardinalities we know) to infer that  $R_{i'}=I_1\cup I_4\cup I_{4+i'} \cup I_{4+\sigma(i')}$, showing that $D$ is the same matrix as the one given in \eqref{eqn matrix D}, up to the isometries of $\Gro(B)$ preserving the distinguished bases which we applied. Hence $\Lambda$ satisfies assumption~\ref{assumption decomp determined} of Lemma~\ref{lemma pic}.
	\end{enumerate}
\end{remark}

\begin{prop}\label{prop picent}
	Assume $k$ is an algebraically closed field of characteristic two. Let $Q$, $f$, $\mm$, $\tt$  be as in Proposition~\ref{prop semimimple q3k}, set $c_\alpha=1$ for all $\alpha\in Q_1$, and let  $\mm':\ Q_1/\langle g \rangle \longrightarrow \Z_{>0}$ be arbitrary. Then
	\begin{equation}
		\Picent(\Gamma(Q,f,\mm,\cc,\tt; \mm'))=1.
	\end{equation}
\end{prop}
\begin{proof}
	By Proposition~\ref{prop multiplicity independence} the isomorphism type of $\Gamma(Q,f,\mm,\cc,\tt; \mm')$ as a ring is independent of $\mm'$. Since $\Picent$ depends exclusively on the ring structure, we can assume without loss of generality that $m'_\alpha=1$ for all $\alpha\in Q_1$. Set $\Lambda=\Gamma(Q,f,\mm,\cc,\tt; \mm')$ and $A=k\sqq{X}\otimes_{k\sbb{X}} \Lambda$. Let $e_1,e_2,e_3$ be the orthogonal primitive idempotents in $\Lambda$ as in Remark~\ref{remark structure q3k}, and let $\eta_1,\ldots,\eta_7$ denote the primitive idempotents in $Z(A)$. By the shape of the matrix $D$ given in equation~\eqref{eqn matrix D} of Remark~\ref{remark structure q3k} it is clear that $e_i \eta_j$ is either zero or a primitive idempotent  in $A$ for all $i,j$. To be specific, $e_i \eta_j$ is non-zero if and only if the $(i,j)$-entry of the matrix $D$ in \eqref{eqn matrix D}	is non-zero. Moreover, by Remark~\ref{remark structure q3k}~\eqref{point iso centre} it also follows that $e_i \eta_j A e_i\eta_j \cong k\sqq{X}$ as $k\sqq{X}$-algebras whenever $e_i \eta_j\neq 0$ (using the assumption that  $m'_\alpha=1$ for all $\alpha\in Q_1$).

	Now let us consider an arbitrary $\gamma\in\Autcent(\Lambda)$ (note that $\Picent(\Lambda)=\Outcent(\Lambda)$, since the simple $\Lambda$-modules are the reduction modulo $X$ of irreducible lattices).  The idea of the remainder of the proof is to modify $\gamma$ be inner automorphisms until we reach the identity automorphism.  After modifying $\gamma$ by an inner automorphism we can assume $\gamma(e_i)=e_i$ for all $i\in\{1,2,3\}$, and therefore $\gamma(e_i\eta_j)=e_i\eta_j$ for all $i,j$. Since  $\Outcent(A)=1$ it follows that $\gamma(x)=uxu^{-1}$ for some unit $u\in A$, and  $\gamma(e_i\eta_j)=e_i\eta_j$  implies that $u=\sum_{i,j} u_{ij} e_i\eta_j$ for certain $u_{ij}\in k\sqq{X}^\times$, where $i,j$ run over all tuples indexing non-zero entries of the matrix given in equation~\eqref{eqn matrix D}.  Multiplying $u$ by $\sum_{i=1}^3 c_ie_i$, where the $c_i$ are units in $k\sqq{X}$, only changes $u$ by an inner automorphism. We can therefore assume that $u_{i4}=1$ for all $i\in\{1,2,3\}$.  Multiplying $u$ by a unit in $Z(A)$ does not change the automorphism it induces. Therefore we can also assume without loss of generality that $u_{1j}=1$ for all $1\leq j\leq 7$, and  $u_{22}=1$, $u_{33}=1$ as well as $u_{26}=1$. The only potentially non-trivial entries of $u$ are therefore  $u_{25}$, $u_{36}$ and $u_{37}$.

	Now we must have $u e_2\Lambda e_1 u^{-1}=(u_{25} \eta_5 + \eta_4) \cdot e_2 \Lambda e_1\subseteq e_2 \Lambda e_1$. Hence $u_{25} \eta_5 + \eta_4$ lies in the endomorphism ring of $e_2\Lambda e_1$ as a $Z(\Lambda)$-module. Since $e_2\Lambda e_1$ is generated by a single element as a $Z(\Lambda)$-module, this endomorphism ring is a quotient of $Z(\Lambda)$, given by  $\langle \eta_4+\eta_5, X\eta_5\rangle_{k\sbb{X}}$ (seen by projecting the generators given in Remark~\ref{remark structure q3k}~\eqref{point basis centre} to the fourth and fifth component). We have $e_2\Lambda e_2 = e_2 Z(\Lambda)$, and by projecting the generators given in Remark~\ref{remark structure q3k}~\eqref{point basis centre}  we see that $(X^2 \eta_2 + X\eta_5)e_2\in e_2\Lambda e_2$. Hence we can find a unit $v\in e_2\Lambda e_2$ such that  $\eta_5 v = u_{25}$ and $\eta_i v=1$ for all $i\not\in \{2,5\}$. We can then modify $u$ by $e_1+v+e_3$ to obtain a new $u$ where $u_{25}$ is equal to one, and all other $u_{ij}$ are unchanged except possibly $u_{22}$. We can multiply by an appropriate unit in $Z(A)$ to get $u_{22}=1$ as well.

	We can repeat the process above to modify $u$ in such a way that $u_{36}$ and $u_{37}$ also become equal to one, which means $u=1$. Hence we can modify an arbitrary central automorphism of $\Lambda$ by inner automorphisms to obtain the identity automorphism, which implies $\Picent(\Lambda)=\Outcent(\Lambda)=1$.
\end{proof}

\begin{thm}[Quaternion-type algebras]\label{thm trpic q3k}
	Assume $k$ is an algebraically closed field of characteristic two. Define an equivalence relation ``$\sim$'' on $\Z_{\geq 2}^3$ where  $ 	(a_1,a_2,a_3) \sim (b_1,b_2,b_3) $  precisely when the following holds: $a_i=a_j$ if and only if $b_i=b_j$ and $a_i=2$ if and only if $b_i=2$, for all $1\leq i,j\leq 3$.
	\begin{enumerate}
		\item For every $I\in \Z_{\geq 2}^3/\sim$ there is a group $\mathcal H_I$ equipped with a homomorphism $\mathcal H_I \longrightarrow S_7$ such that
		      \begin{equation}
			      \TrPic_k(\mathcal Q(3K)^{a_1,a_2,a_3}) \cong \PicS_k (\mathcal Q(3K)^{a_1,a_2,a_3}) \rtimes  \mathcal H_I
		      \end{equation}
		      for any choice of $(a_1,a_2,a_3)\in I$. The  action of $\mathcal H_I$ on $ \PicS_k (\mathcal Q(3K)^{a_1,a_2,a_3})$ factors through $S_7$.
		\item There is a group $\mathcal G$ such that
		      \begin{equation}
			      \TrPicent(\mathcal Q(3K)^{a_1,a_2,a_3}) \cong \mathcal G
		      \end{equation}
		      for all $a_1,a_2,a_3\geq 2$.
	\end{enumerate}
\end{thm}
\begin{proof}
	Let $Q$, $f$, $\mm$, $\tt$  be as in Proposition~\ref{prop semimimple q3k}, set $c_\alpha=1$ for all $\alpha\in Q_1$, and let  $\mm':\ Q_1/\langle g \rangle \longrightarrow \Z_{>0}$ take values $a_1-1$, $a_2-1$, and $a_3-1$ on the $g$-orbits of $\alpha_1$, $\alpha_2$ and $\alpha_3$. Consider the $k\sbb{X}$-order $\Lambda=\Gamma(Q,f,\mm, \cc, \tt, \emptyset; \mm')$ constructed in Proposition~\ref{prop lift weighted surface}. We know that $\bar{\Lambda}=\Lambda/X\Lambda \cong \mathcal Q(3K)^{a_1,a_2,a_3}$. On top of that, define  $I_2=\{ 1\leq i \leq 3 \ | \ a_i=2 \}$ and a function $\mm'':\ Q_1/\langle g \rangle \longrightarrow \Z_{>0}$ such that $m''_{\alpha_i}=1$ if $i\in I_2$ and  $m''_{\alpha_i}=2$ if $i\not\in I_2$. We set $\Lambda'=\Gamma(Q,f,\mm, \cc, \tt, \emptyset; \mm'')$. Note that by Proposition~\ref{prop multiplicity independence} we have $\Lambda\cong \Lambda'$ as rings. What is even more, we are in the case of Proposition~\ref{prop multiplicity independence}~\eqref{prop muly indep part 1}, which means that both $\Lambda$ and $\Lambda'$ are isomorphic as rings to the pullback
	\begin{equation}\label{eqn pullback q3k ring}
		\xymatrix{
		\frac{\widehat{k[X]Q}}{\phantom{.}\overline{(\alpha f(\alpha) - X\bar \alpha,\
			\alpha g(\alpha)-\bar \alpha g(\bar{\alpha}) | \ \alpha \in Q_1)}\phantom{.}} \ar@{->>}[rr]^{X \mapsto -\sum_{i\in I_2} \alpha_i g(\alpha_i) + g(\alpha_i) \alpha_i} && \frac{\widehat{kQ}}{\phantom{.}\overline{(\alpha f(\alpha),\
			\alpha g(\alpha)-\bar \alpha g(\bar{\alpha}) \ | \ \alpha \in Q_1)} \phantom{.}}\\
		\Gamma	 \ar[u]\ar[rr] && \frac{\widehat{kQ}}{\phantom{.}\overline{(\alpha f(\alpha) \ | \  \alpha \in Q_1)}{\phantom{.}}} \ar@{->>}[u]_{\nu}
		}
	\end{equation}
	and both for $\Lambda$ and $\Lambda'$ the pullback diagram of $k\sbb{X}$-algebras given in equation~\eqref{eqn pullback gamma} is isomorphic to the pullback diagram above by applying the identity to the top left entry, and mapping paths in $\widehat{kQ}$ to themselves in the top and bottom right corners of the diagram. That fixes a ring isomorphism between $\Lambda'$ and $\Lambda$. If we set $A=k\sqq{X}\otimes_{k\sbb{X}} \Lambda$, we can identify (as rings only)
	\begin{equation}\label{eqn identify ssss}
		Z(A) \cong  k\sqq{X}^{\oplus 7}
	\end{equation}
	similar to  Remark~\ref{remark structure q3k}~\eqref{point iso centre}. Then the image of  $X\in Z(\Lambda)$ in the right hand side of equation~\eqref{eqn identify ssss} is $X_\Lambda = (X,X,X,X,X^{m'_{\alpha_1}}, X^{m'_{\alpha_2}}, X^{m'_{\alpha_3}})$, and that of $X\in Z(\Lambda')$ is $X_\Lambda'=(X,X,X,X,X^{m''_{\alpha_1}}, X^{m''_{\alpha_2}}, X^{m''_{\alpha_3}})$. Set $R=k\sbb{X_\Lambda}$ and  $S=k\sbb{X_{\Lambda'}}$, both contained in the centre of $\Lambda$, and define $K$ and $L$ as their respective fields of fractions. Instead of working with $\Lambda$ and $\Lambda'$, we can now simply consider $\Lambda$ either as an $R$-order or as an $S$-order.  We will now check that the assumptions of Lemma~\ref{lemma pic} are satisfied for $\Lambda$, $R$ and $S$. Note that we have already checked the assumptions~\ref{assumprion a1}--\ref{assumption ex gamma} in Remark~\ref{remark structure q3k}.

	By the definition of $\mm''$ we have
	\begin{equation}
		\Aut_L(Z(A)) = \Aut_{k\sqq{X}}(k\sqq{X}^{\oplus (4+|I_2|)}) \times  \Aut_{k\sqq{X^2}}(k\sqq{X}^{\oplus (3-|I_2|)}) \cong S_{4+|I_2|} \times S_{3-|I_2|},
	\end{equation}
	where we are using the fact that $\operatorname{char}(k)=2$ and therefore $\Aut_{k\sqq{X^2}}(k\sqq{X})=1$. In particular, assumption~\ref{assumption autl} of Lemma~\ref{lemma pic} is satisfied.

	Any $\gamma\in \Aut_K(Z(A))$ induces a permutation $\sigma \in S_7$ which fixes $X_{\Lambda}$. This is because any two entries of $X_\Lambda\in Z(A)=k\sqq{X}^{\oplus 7}$  are either equal or they generate (complete) subfields of different index in $k\sqq{X}$. In particular, $\sigma$ acting by permutation on the components of  $Z(A)$ also induces an element of $\Aut_K(Z(A))$, which in turn induces the same automorphism of $\Gro(A)$ as $\gamma$.  Now, by definition of $\mm''$, any $\sigma \in S_7$ which fixes $X_{\Lambda}$ also fixes $X_{\Lambda'}$, and therefore the $\sigma$ from before acting on $Z(A)$ by permutation is also an $L$-linear automorphism. We have thus found an element of $\Aut_L(Z(A))\cap \Aut_K(Z(A))$ inducing the same automorphism of $\Gro(A)$ as $\gamma$, which proves that assumption~\ref{assumption aut intersect} of Lemma~\ref{lemma pic} is satisfied.

	Since any self-isometry of $\Gro(A)$, where $A$ is considered as a $K$-algebra, remains a self-isometry when we consider $A$ as an $L$-algebra, and since in Remark~\ref{remark structure q3k} we  have effectively also verified assumption~\ref{assumption ex gamma} for $\Lambda$ considered as an $S$-algebra, it follows that in the situation of assumption~\ref{assumption center s stable} there always exists a $\gamma' \in \Aut_L(Z(A))$ inducing the desired permutation of central primitive idempotents such that $\gamma'(Z(\Lambda)) \subseteq Z(\Lambda)$. Now  assumption~\ref{assumption center s stable}  asks for any  $\gamma \in \Aut_L(Z(A))$ inducing the same permutation of central primitive idempotents as $\gamma'$ to also satisfy $\gamma(Z(\Lambda)) \subseteq Z(\Lambda)$. But from our verification of assumption~\ref{assumption autl} it follows that $\gamma$ is fully determined by the permutation it induces on the central primitive idempotents, which means that $\gamma=\gamma'$. This shows that  assumption~\ref{assumption center s stable}  of Lemma~\ref{lemma pic} is satisfied.

	Let us now also verify assumptions~\ref{assumption B1}--\ref{assumption b5} of Lemma~\ref{lemma extra structure}. The algebra $\bar{\Lambda} \cong\mathcal Q(3K)^{a_1,a_2,a_3}$ is silting-discrete by Proposition~\ref{prop silting discrete}, which implies assumption~\ref{assumption B1}. If an element of $\Pic_S(\Lambda)$ acts trivially on $\Gro({\Lambda})$, then it must also act trivially on  $\Gro(A)$ by Remark~\ref{remark structure q3k}~\eqref{point self isom}, since no self-isometry fixes the first three simple $A$-modules and induces a non-trivial permutation on the others. Hence an element of $\Pic_S(\Lambda)$ acting trivially on $\Gro({\Lambda})$ lies in $\Picent(\Lambda)$, which is trivial by Proposition~\ref{prop picent}. Assumption~\ref{assumption B2} follows. By definition of $\Lambda$ there is an automorphism (both $R$-linear and $S$-linear) inducing the same permutation on $\Gro(\bar \Lambda)$ as some $\sigma\in S_3$ if and only if $\sigma$ stabilises $(a_1,a_2,a_3)$, and the same is true for automorphisms of $\bar{\Lambda}$. This implies assumption~\ref{assumption B3}.

	Since the $S$-algebra structure on $\Lambda$ depends only on $I_2$ as defined at the beginning of the proof, the group ${\mathcal H\leq \TrPic_S(\Lambda)}$ from Lemma~\ref{lemma pic} depends only on $I$ (as $I$ determines which elements of $\Aut(\Gro(A))$ are isometries, and it determines $I_2$). It therefore makes sense to denote $\mathcal H$ by $\mathcal H_I$. Now, any element of $\PicS_k (\bar{\Lambda}) $ can be represented by an automorphism  $\gamma$ of $\bar{\Lambda}$ which fixes $e_1$, $e_2$ and $e_3$. Since each $e_i\bar{\Lambda}e_j$ is spanned by the paths along the $g$-orbit of the unique arrow from $e_i$ to $e_j$, it follows that each $\alpha \in Q_1$ gets mapped to $\gamma(\alpha)=q(\alpha) \alpha$ for some unit $q(\alpha) \in k[\alpha g(\alpha)]$.  Conjugation by $u=q(\alpha_1)e_1+q(\alpha_2)e_2+q(\alpha_3)e_3$ maps the arrow $\alpha_i$ to $r_i q(\alpha_i)\alpha_i$ for $i\in\{1,2,3\}$, where $r_1,r_2,r_3\in k^\times$ are constants. By further conjugation we can find an inner automorphism which maps $\alpha_1$ and $\alpha_2$ to $\gamma(\alpha_1)$ and $\gamma(\alpha_2)$, respectively, and $\alpha_3$ to $r\gamma(\alpha_3)$, where $r=r_1r_2r_3$. That is, we may assume without loss of generality that  $\gamma(\alpha_1)=\alpha_1$, $\gamma(\alpha_2)=\alpha_2$ and $\gamma(\alpha_3)=r^{-1}\alpha_3$ for some constant $r\in k^\times$. If we now assume that $\gamma$ restricts to the trivial automorphism of $Z(\bar{\Lambda})$, then  $\gamma(\alpha_1\alpha_2\alpha_3)=\alpha_1\alpha_2\alpha_3$, and therefore $r=1$. That is, $\gamma(\alpha_i)=\alpha_i$ for all $i\in\{1,2,3\}$. Moreover, again assuming that $\gamma$ is trivial on the centre, for an $\alpha\in\{\alpha_1,\alpha_2,\alpha_3\}$ we have $\gamma(\alpha g(\alpha) +  g(\alpha)\alpha)=\alpha g(\alpha) +  g(\alpha)\alpha$ and therefore  $\gamma(\alpha g(\alpha))=\alpha g(\alpha)$. Given that $\gamma(\alpha)=\alpha$, it follows that $g(\alpha)-\gamma(g(\alpha))$ is a linear combination of paths of the form $(g(\alpha)\alpha)^ig(\alpha)$  for $i\in \Z_{\geq 0}$ which annihilates $\alpha$. This is only possible if $g(\alpha)-\gamma(g(\alpha))$ is zero, since otherwise there would be an $i$ such that $(g(\alpha)\alpha)^ig(\alpha)\neq 0$ but $(\alpha g(\alpha))^{i+1}= 0$, which we know is not the case in generalised weighted surface algebras. Hence we have seen that if $\gamma\in \PicS_k(\bar{\Lambda})$ restricts to the identity automorphism on $Z(\bar{\Lambda})$, then $\gamma =1$. That is, the map $\PicS_k(\bar{\Lambda})\longrightarrow \Aut_k(Z(\bar{\Lambda}))$ is injective, so in particular condition~\ref{assumption b4} holds.

	To verify condition~\ref{assumption b5}  first recall that by Remark~\ref{remark structure q3k}~\eqref{point basis centre}~and~\eqref{point self isom} the image of $\mathcal H_I$  in $\Aut_S(Z(\Lambda))\subseteq \Aut_L(Z(A))\subseteq S_7$ is contained in
	\begin{equation}
		\langle (1,2,3)(5,6,7),\ (1,2)(6,7),\ (1,2)(3,4) \rangle_{(1,1,1,1,m'_{\alpha_1},m'_{\alpha_2}, m'_{\alpha_3})}.
	\end{equation}
	One can also compute the images of the elements of $\mathcal H_I$ in $\Aut(\F_2\otimes_\Z \Gro(\bar{\Lambda}))$, since these are determined by  the self-isometries of $\Gro(A)$ these elements induce (and Remark~\ref{remark structure q3k}~\eqref{point self isom} gives a list of possible self-isometries). One can check that an element of $\mathcal H_I$ acts trivially on $Z(\Lambda)$ if and only if it acts trivially  on $\F_2\otimes_\Z \Gro(\bar{\Lambda})$. The assumptions of Lemma~\ref{lemma action centre} are satisfied, and therefore it follows that if an element of $\mathcal H_I$ acts trivially on $\F_2\otimes_\Z \Gro(\bar{\Lambda})$ then it stabilises the subspaces $e_i\soc(\bar{\Lambda})e_i\subset Z(\bar{\Lambda})$ for all $i\in\{1,2,3\}$. Conversely, if an element of $\mathcal H_I$ stabilises these subspaces, then it has no choice but to act trivially on $\F_2\otimes_\Z \Gro(\bar{\Lambda})$, and therefore also on $Z(\Lambda)$. This is again by Lemma~\ref{lemma action centre}, but note that in characteristic $p$ it would only follow that each distinguished basis element of $\F_p\otimes_\Z \Gro(\bar{\Lambda})$ gets mapped to a non-zero multiple of itself, which in characteristic $p=2$ happens to be sufficient.

	An element of $\PicS_k(\bar\Lambda)$ is induced by an automorphism $\gamma\in \Aut_k(\bar{\Lambda})$ which fixes the idempotents $e_1$, $e_2$, and $e_3$, which shows that $\gamma(e_i\soc(\bar{\Lambda})e_i)=e_i\soc(\bar{\Lambda})e_i$ for all $i\in\{1,2,3\}$. It follows that if an element of $\PicS_k(\bar{\Lambda})$, represented by some $\gamma\in \Aut_k(\bar{\Lambda})$ which fixes all $e_i$, induces the same automorphism of $Z(\bar{\Lambda})$ as an element of $\mathcal H_I$, then that element of $\mathcal H_I$ induces the identity on $Z(\bar{\Lambda})$ by the discussion in the previous paragraph. It follows that condition~\ref{assumption b5}  holds.

	We now know by Lemma~\ref{lemma extra structure} that
	\begin{equation}
		\TrPic_k(\bar{\Lambda}) = \PicS_k(\bar{\Lambda}) \rtimes \mathcal H_I,
	\end{equation}
	where the action of $\mathcal H_I$ on $\PicS_k(\bar{\Lambda})$ has  $\TrPicent(\Lambda)$ in its kernel and therefore factors through the natural map  $\mathcal H_I \longrightarrow \Aut_S(Z(\Lambda)) \leq \Aut_L(Z(A))\leq S_7$.  We also know by Lemma~\ref{lemma extra structure} that $\TrPicent(\bar{\Lambda}) = \Ker(\mathcal H_I \longrightarrow \Aut_k(Z(\bar{\Lambda})))$, and by the preceding discussion regarding \ref{assumption b5} this kernel is equal to $\TrPicent(\Lambda)$. So we can set $\mathcal G=\TrPicent(\Lambda)$. This completes the proof.
\end{proof}

\begin{remark}
	In the situation of the preceding theorem, the action of $\mathcal H_I$ on $\PicS_k(\bar{\Lambda})$ can be determined explicitly. Namely, the image of the map $\mathcal H_I\longrightarrow \Aut_S(Z(\Lambda))\leq S_7$ is contained in $\langle (1,2,3)(5,6,7),\ (1,2)(5,6),\ (1,2)(3,4) \rangle$ by Remark~\ref{remark structure q3k}~\eqref{point self isom}. Clearly $\Pic_S(\Lambda)\cap\Pic_R(\Lambda)$ is contained in $\mathcal H_I$, which acts on $\bar \Lambda$ by permutation of vertices and arrows. This explains how the image of $\mathcal H_I$ in $\langle (1,2,3)(5,6,7), (1,2)(5,6) \rangle\cong S_3$ acts on $\PicS_k(\bar{\Lambda})$. But then one has to figure out how the permutation $(12)(34)$ acts on $Z(\bar{\Lambda})$. Of course Remark~\ref{remark structure q3k}~\eqref{point basis centre}  helps with that, but this is less straight-forward.
\end{remark}

\begin{thm}[Twisted Brauer graph algebras]\label{thm Brauer graph}
	Let $k$ be an algebraically closed field, and let $Q$ and $f$ be as in \S\ref{section gen weighted surf}.  Assume that the Brauer graph of $(Q,f)$  in the sense of Definition~\ref{def brauer graph}  is a connected simple graph (i.e. it has no loops and no double edges). Define an equivalence relation ``$\sim$'' on the set $\mathcal I = \{ \mm:\ Q/\langle g \rangle \longrightarrow \Z_{>0} \}$, where $\mm\sim \mm'$ means that $m_\alpha=m_\beta$ if and only if $m'_\alpha=m'_\beta$ for all $\alpha,\beta\in Q_1$. 

	\begin{enumerate}
		\item For each $I\in \mathcal I/\sim$ there is a group $\mathcal H_I$ such that
		      \begin{equation}
			      \TrPic_k (\Lambda_{\rm tw}(Q, f, \mm)) \cong \PicS_k (\Lambda_{\rm tw}(Q, f, \mm)) \cdot  \mathcal H_I
		      \end{equation}
		      for all $\mm \in I$.
		\item There is a group $\mathcal G$ such that
		      \begin{equation}
			      \TrPicent (\Lambda_{\rm tw}(Q, f, \mm)) \cong  \mathcal G
		      \end{equation}
		      for all $\mm \in \mathcal I$ for which $\Lambda_{\rm tw}(Q, f, \mm)$ is symmetric.
	\end{enumerate}
\end{thm}
\begin{proof}
	Define $\mm':\ Q_1/\langle g \rangle \longrightarrow \Z_{>0}$ by setting $m'_\alpha=1$ for all $\alpha\in Q_1$. Then $\GammaTw(Q,f,\mm)$ and $\GammaTw(Q,f,\mm')$ are isomorphic as rings, but they carry different  $k\sbb{X}$-algebra structures (this comes directly from Proposition~\ref{prop biserial order}). Set $\Lambda=\GammaTw(Q,f,\mm)$, where $R=k\sbb{X}$ acts as defined in Proposition~\ref{prop biserial order}. This becomes an $S=k\sbb{X}$-order by letting $X$ act as $\sum_{\alpha\in Q_1} \alpha g(\alpha)\cdots g^{n_\alpha-1}(\alpha)$, and  $\Lambda$ is isomorphic to $\GammaTw(Q,f,\mm')$ as an $S$-order.  Let us also define $K$ and $L$ as in Lemma~\ref{lemma pic}, and set $A= K\otimes _R \Lambda$. We should mention that we will exclude the case where the Brauer graph of $(Q,f)$ consists of only a single edge whenever necessary (in that case the assertion follows easily from \cite[Proposition~3.3]{RouquierZimmermann}).

	Let us first verify that $\Lambda$ satisfies the assumptions of Lemma~\ref{lemma unique lifting} as an $R$-order. Note that $\bar{\Lambda}=\Lambda/X\Lambda \cong \LambdaTw(Q,f,\mm)$. Condition~\ref{assumption lambda 1} is clear from the presentation of $\LambdaTw(Q,f,\mm)$. For condition~\ref{assumption lambda 2} we can use the second option: if $e$ is a primitive idempotent in $\Lambda$, and $\alpha$ and $\bar{\alpha}$ are the two arrows whose source is the corresponding vertex in $Q_0$, then $\alpha$ and $\bar \alpha$ lie in different $g$-orbits by our ``no loops''-assumption on the Brauer graph, which implies that  $e\Lambda e$ is isomorphic to $\{ (p,q) \in k\sbb{X} \oplus k\sbb{X} \ | \ p(0)=q(0) \}$ as an $S$-algebra (directly from the presentation of $\GammaTw(Q,f,\mm')$). This is a subspace of $k$-codimension one in $k\sbb{X}\oplus k\sbb{X}$, which is the unique maximal order in $k\sqq{X}\oplus k\sqq{X}$, irrespective of whether we view it as an $R$-order or an $S$-order. This shows that there is no local $R$-order  in $eAe$ properly containing $e\Lambda e$.

	For condition~\ref{assumption lambda 3} pick vertices $e_1,e_2,e_3\in Q_0$ such that  ${e_1\neq e_2}$ and ${e_2\neq e_3}$ (these correspond to edges in the Brauer graph and idempotents in $\Lambda$). If there is no vertex in the Brauer graph of $(Q,f)$ such that the three edges associated with $e_1$, $e_2$ and $e_3$ are incident to that vertex, then $e_1\Lambda e_2 \Lambda e_3=0$ and  condition~\ref{assumption lambda 3} holds for these three idempotents. If there is such a vertex, then it is unique by our assumption on the Brauer graph. By swapping $e_1$ and $e_3$ if necessary we can assume that $e_2$ is between $e_1$ and $e_3$ in the cyclic order around that vertex. Let $\alpha\in Q_1$ be the arrow whose $g$-orbit corresponds to the aforementioned vertex in the Brauer graph and whose source is $e_1$.  It then follows from the presentation of $\Lambda$ that $e_1 \Lambda e_2 \Lambda e_3$ is the completed span of all paths of positive length  from  $e_1$ to $e_3$ along the $g$-orbit of $\alpha$, since any such path passes through $e_2$ anyway. This is a pure sublattice of $\Lambda$, since if $zw$ lies in it for some $w\in \Lambda$ (with $z$ as in Proposition~\ref{prop biserial order}), then $w$ lies in it too. To see this one just has to note that if $w$ involves any paths along other $g$-orbits, even of length zero, then so does $zw$. Condition~\ref{assumption lambda 3} now follows from purity.

	To check condition~\ref{assumption lambda 4}, consider $e_1\neq e_2 \in Q_0$. If the corresponding edges in the Brauer graph do not meet in a vertex, then $e_1A e_2$ is zero and there is nothing to show. If these edges do meet in a vertex, then this vertex is unique by our assumptions on the Brauer graph, and therefore we can assign to $e_1$ and $e_2$ the simple $Z(A)$-module corresponding to that vertex in the Brauer graph (recall that we  parametrised the simple $A$-modules in Proposition~\ref{prop biserial order}).  If we now take pair-wise distinct $e_1, e_2, e_3\in Q_0$, then $e_1Ae_2Ae_3$ is non-zero if and only if the edges in the Brauer graph corresponding to $e_1$, $e_2$ and $e_3$ meet in a single vertex, and the simple $Z(A)$-module belonging to that vertex is by definition the simple $Z(A)$-module we attached to any pair selected from $e_1$, $e_2$ and $e_3$, which verifies \ref{assumption lambda 4}.

	We will deal with \ref{assumption lambda 5} further below, but first we need to have a closer look at the centre of $A$. The assumption \ref{assumption centre aut lifts} is trivially satisfied, since no matrix rings over (proper) skew-fields occur in the Wedderburn decomposition of $A$. The centre of $A$ can be described as
	\begin{equation}\label{eqn centre comp}
		Z(A) = \bigoplus _{\alpha \langle g\rangle\in Q_1/\langle g\rangle} k\sqq{X^{1/m_\alpha}}
	\end{equation}
	where $X\in R$ acts as $X$ on each component, and $X\in S$ acts as $X^{1/m_\alpha}$ on the component labelled by $\alpha\langle g \rangle$. So clearly $\Aut_L(Z(A))$ is the whole symmetric group on the components of the direct sum above, which are labelled by $Q_1/\langle g\rangle$. If, for some $L$-linear automorphism of $Z(A)$, the corresponding permutation of $Q_1/\langle g\rangle$ fixes $\mm$, then this $S$-linear automorphism is also $R$-linear.  This already implies \ref{assumption autl}. Similarly, an element of the group $\Aut_K(Z(A))$ induces a permutation of those components of \eqref{eqn centre comp} which share  the same multiplicity $m_\alpha$,  possibly followed by automorphisms of $k\sqq{X^{1/m_\alpha}}$ as a $k\sqq{X}$-algebra. In particular condition  \ref{assumption aut intersect} holds.

	Now note that the centre of $\Lambda$ is embedded in $Z(A)$ as described in \eqref{eqn centre comp} as follows:
	\begin{equation}\label{eqn z gamma}
		Z(\Lambda) = \left\{  p = (p_\alpha)_\alpha\in  \bigoplus_{\alpha \langle g\rangle\in Q_1/\langle g\rangle} k\sbb{X^{1/m_\alpha}} \subseteq Z(A) \ \bigg| \ p_{\alpha}(0)=p_\beta(0) \textrm{ for all $\alpha,\beta\in Q_1$} \right\}.
	\end{equation}
	Clearly this order is fixed by all elements of $\Aut_K(Z(A))$ and $\Aut_S(Z(A))$, which shows that conditions~\ref{assumption ex gamma}~and~\ref{assumption center s stable} hold.

	A map $\sigma$ as in  \ref{assumption lambda 5} can be interpreted as a permutation of the edges of the Brauer graph of $(Q,f)$, ignoring multiplicities for the time being. To simplify notation, let us regard such a $\sigma$ as a map from $Q_0$ into itself. If $e_1,\ldots,e_r\in Q_0$ (for $r\geq 2$) are pair-wise distinct and correspond to $r$ edges in the Brauer graph incident to some vertex $v$, and without loss of generality are ordered with respect to the cyclic order around $v$, then $e_{1}\bar{\Lambda} e_{2} \bar{\Lambda} \cdots \bar{\Lambda} e_{r}\neq 0$. By assumption on $\sigma$ we then have $\sigma(e_{1})A \sigma(e_{2})A \cdots A \sigma(e_{r})\neq 0$, which is only possible if the edges corresponding to $\sigma(e_{1}),\ldots,\sigma(e_{r})$ are all incident to some vertex $v'$ in the Brauer graph. That is, $\sigma$ induces an automorphism of the \emph{line graph} of the Brauer graph, which moreover preserves \emph{stars} (that is, collections of edges that meet in a single vertex). Since $\sigma$ is an automorphism of a line graph, the map $\sigma^{-1}$ needs to preserve stars as well (as the number of stars in the range of $\sigma$ is equal to the number of stars in the domain). By \cite[Theorem 1]{HemmingerLineGraph} (a variation of Whitney's line graph theorem) $\sigma$ is induced by a graph automorphism of the Brauer graph. The vertices of the Brauer graph correspond to the distinguished basis of $\Gro(A)$, and therefore there is a $\widehat \gamma \in \Aut(\Gro(A))$ permuting the distinguished basis such that $\widehat \gamma([eA]) =[\sigma(e)A]$ for all $e\in Q_0$. In particular $\widehat{\gamma}([A])=[A]$, so if $\widehat \gamma$ is induced by a Morita auto-equivalence, then it is induced by an automorphism. To check that $\widehat \gamma$ comes from a $\gamma \in \Aut_K(A)$ one only needs to check that there is a corresponding  automorphism of $Z(A)$ (since $A$ is Morita equivalent to its centre), which reduces to showing that $\widehat \gamma$ is an isometry. That is, we need to show that $\sigma$ preserves the multiplicities of vertices in the Brauer graph of $(Q,f)$. If a vertex in the Brauer graph is incident to at least two edges, corresponding to idempotents $e_1,e_2\in \Lambda$, then the multiplicity of the vertex is $\dim_K e_1Ae_2$, which is preserved by $\sigma$. If the vertex is incident only to a single edge, corresponding to an idempotent $e_1$, then this edge must be incident to another edge, corresponding to some idempotent $e_2$ (unless the Brauer graph has only one edge, a case we have excluded earlier). In this case the multiplicity of the vertex in question is $\dim_K  e_1Ae_1 - \dim_K e_1Ae_2$, which again is preserved by $\sigma$. It follows that $\sigma$ preserves multiplicities, which gives us an automorphism $\gamma$ inducing $\widehat \gamma$. Since $\gamma(Z(\Lambda)) \subseteq Z(\Lambda)$ is trivially satisfied, by the description of $Z(\Lambda)$ given in equation~\eqref{eqn z gamma} and the subsequent remarks, the condition \ref{assumption lambda 5} holds.

	Let us now show condition~\ref{assumption decomp determined}. To this end, fix an $R$-order $\Gamma$ derived equivalent to $\Lambda$ in a $K$-algebra~$B$ Morita equivalent to $A$, and assume $\bar{\Gamma} =\Gamma/X\Gamma\cong \bar\Lambda$. In particular, $Z(\Gamma)\cong Z(\Lambda)$. Fix an isomorphism $\varphi:\ \bar{\Lambda} \longrightarrow \bar{\Gamma}$. If $e_1,\ldots,e_r\in Q_0$ (for some $r\geq 1$) are pair-wise distinct, and $e=e_1+\ldots+e_r$, then $e\bar{\Lambda} e$ is again a twisted Brauer graph algebra whose Brauer graph is the subgraph of the Brauer graph of $(Q,f)$ retaining only the edges corresponding to $e_1,\ldots,e_r$ (removing all orphaned vertices and keeping the multiplicities of all other vertices unchanged). In the same vein, $e\Lambda e$ is the corresponding order of the form $\GammaTw(\ldots)$ as defined in Proposition~\ref{prop biserial order}. This can be verified using the presentations given in Definition~\ref{def twisted bga} and Proposition~\ref{prop biserial order} (refer to \cite[Lemma 1.12]{GnedinRibbonOrders} for a proof).

	Let us choose $e_1,\ldots,e_r\in Q_0$ corresponding to a spanning tree of the Brauer graph of $(Q,f)$.  Let ${e_1',\ldots,e'_r\in \Gamma}$ be pair-wise orthogonal lifts of $\varphi(e_1),\ldots,\varphi(e_r)$, and $e'=e'_1+\ldots+e'_r$. Then $e\bar \Lambda e$ and $e'\bar{\Gamma} e'$ are isomorphic Brauer graph algebras whose graph is a tree with $r$ edges. We know furthermore that
	\begin{equation}
		\rank_\Z \Gro(eAe) \leq \rank_\Z \Gro(A) = r+1 = \rank_\Z \Gro(B) \geq  \rank_\Z \Gro(e'Be').
	\end{equation}
	We also know that the leftmost ``$\leq$'' is actually an equality, and we have a description of the map $D_{e\Lambda e}$ from Proposition~\ref{prop biserial order}. The problem is that we do not have the corresponding information for $e'\Gamma e'$.

	By \cite[Theorem 7.4]{BrauerTreeDerEq} the algebra $e\bar{\Lambda}e$ is derived equivalent to a Brauer graph algebra $\Omega$ whose Brauer graph is a star, with the same multiplicities as those occurring in $e\bar{\Lambda} e$,  up to permutation. By \cite[Theorem 8.3]{BrauerTreeDerEq} we may actually assume that the multiplicity of the central vertex is minimal among all multiplicities, that is, it is equal to $m_0=\min\{m_\alpha \ | \ \alpha\in Q_1\}$. Now we can choose  a one-sided tilting complex $\bar T^\bullet$ over $e\bar{\Lambda}e$ such that $\End_{\mathcal D^b(e\bar{\Lambda}e)}(\bar T^\bullet)^{\opp} \cong \Omega$. Then  $\bar T'^\bullet = {e'\bar{\Gamma} e'_\varphi}\otimes_{e\bar{\Lambda}e} \bar T^\bullet$ is a one-sided tilting complex over $e'\bar{\Gamma} e'$ also with endomorphism ring $\Omega$. We can choose lifts $T^\bullet$ and $T'^\bullet$  of $\bar T^\bullet$ and $\bar T'^\bullet$ to $e\Lambda e$ and $e'\Gamma e'$, whose endomorphism rings are $R$-orders $\Omega_1$ and $\Omega_2$ reducing to $\Omega$. We can furthermore pick two-sided tilting complexes $Y^\bullet$ and $Y'^\bullet$ whose restrictions to the left are $T^\bullet$ and $T'^\bullet$. If $\bar Y^\bullet$ and $\bar Y'^\bullet$ denote the reductions of $Y^\bullet$ and $Y'^\bullet$ modulo $X$, then their restrictions to the right differ only by an automorphism of $\Omega$, which acts on $\Gro(\Omega)$ as some permutation $\tau$ of the distinguished basis. Hence we get a diagram
	\begin{equation}\label{eqn rho brtree}
		\xymatrix{
		\Gro(e\bar \Lambda e) \ar[r]^{\varphi_{\bar Y}} \ar[d]^{D_{e\Lambda e}} & \Gro(\Omega) \ar[d]^{D_{\Omega_1}}\ar[r]^{\tau}
		& \Gro(\Omega) \ar[d]^{D_{{\Omega_2}}} & 	\Gro(e'\bar \Gamma e') \ar[d]^{D_{e'\Gamma e'}} \ar[l]_{\varphi_{\bar Y'}} \\
		\Gro(eAe) \ar[r]^{\varphi_{KY}} & \Gro(K\otimes_R \Omega_1)\ar@{-->}[r]^{\exists\rho?} & \Gro(K\otimes_R \Omega_2) & \ar[l]_{\varphi_{KY'}} \Gro(e'Be')
		}
	\end{equation}
	where the leftmost and the rightmost squares are commutative. The composition ${\varphi_{\bar Y'}^{-1}\circ\tau\circ\varphi_{\bar Y}}$ corresponds to an isomorphism between $e\bar{\Lambda}e$ and $e'\bar{\Gamma}e'$, and therefore maps the distinguished basis to the distinguished basis. We will now show that we can find an isometry ${\rho:\ \Gro(K\otimes_R \Omega_1)\longrightarrow \Gro(K\otimes_R \Omega_2)}$ making the middle square commute. If such a map exists, then $\varphi_{KY'}^{-1}\circ \rho \circ \varphi_{KY}$ must actually map the distinguished basis of $\Gro(eAe)$ to that of  $\Gro(e'Be')$, since  $D_{e\Lambda e}$, $D_{e'\Gamma e'}$ and $\varphi_{\bar Y'}^{-1}\circ\tau\circ\varphi_{\bar Y}$ all preserve the $\Z_{\geq 0}$-span of the distinguished bases (and together these form a commutative square). It then follows that for a projective indecomposable $e'\bar{\Gamma}e'$-module $P$ the image $D_{e'\Gamma e'}([P])$  is the sum of exactly two distinguished basis elements of $\Gro(e'Be')$, since the analogous statement is true for $D_{e\Lambda e}$.

	To find $\rho$ first note that if $f_1,\ldots,f_r$ are a full set of orthogonal primitive idempotents in $\Omega$ such that the corresponding edges in the star-shaped Brauer graph follow the cyclic order around the central vertex, then $f_1\Omega f_2\Omega \cdots \Omega f_r\neq 0$. If $f\upbr{i}_1,\ldots,f\upbr{i}_r$ (for $i\in \{1,2\}$) denote mutually orthogonal lifts of $f_1,\ldots,f_r$ to $\Omega_i$, then  $f\upbr{i}_1\Omega_i f\upbr{i}_2\Omega_i \cdots \Omega_i f\upbr{i}_r\neq 0$. It follows that for each $i\in \{1,2\}$ there is a distinguished basis element $[V_i]$ in $\Gro(K\otimes_R \Omega_i)$ such that for each $1\leq j \leq r$ we have $D_{\Omega_i}([f_j\Omega]) = [V_i]+[W_{i,j}]$ for some $K\otimes_R\Omega_i$-module $W_{i,j}$. We know that $(D_{\Omega_i}([f_j\Omega]), D_{\Omega_i}([f_l\Omega]))=([f_j\Omega], [f_l\Omega])= \dim_k f_j\Omega f_l= m_0$ for all $j\neq l$, and $m_0$ happens to be the minimal length of an element of $\Gro(K\otimes_R\Omega_i)$.   Hence $[V_1]$ and $[V_2]$ are actually of length $m_0$, and the $[W_{i,j}]$ must be both mutually orthogonal and orthogonal to $[V_i]$. Moreover, the $[W_{i,j}]$ are all non-zero, since $(D_{\Omega_i}([f_j\Omega]),D_{\Omega_i}([f_j\Omega]))=\dim_k f_j\Omega f_j > m_0$ for all $j$. Since $\rank_\Z (\Gro(K\otimes_R\Omega_i))\leq r+1$ it follows (essentially by the ``pigeonhole principle'') that all $[W_{i,j}]$ must be distinguished basis elements, pair-wise distinct and distinct from $[V_i]$ (and $\rank_\Z (\Gro(K\otimes_R\Omega_i))= r+1$ also follows). To define $\rho$ we can now simply map $[V_1]$ to $[V_2]$ and $[W_{1,j}]$ to $[W_{2,j}]$ for $1\leq j \leq r$.

	Let us now assign a graph $G_\Omega$ to any $R$-order $\Omega$ in a semisimple $K$-algebra which has the property that for each primitive idempotent $e\in \Omega$ the $K\otimes_R\Omega$-module $K\otimes_Re\Omega$ has exactly two simple constituents, non-isomorphic to each other. Write $\bar\Omega=\Omega/X\Omega$, as usual. We define the vertices of $G_\Omega$ to be in bijection with the elements of the distinguished basis of $\Gro(K\otimes_R\Omega)$, and the edges of $G_\Omega$ to be in bijection with the elements of the distinguished basis of $\Gro(\bar\Omega)$. We want the edge labelled by $[e \bar\Omega]$, for a primitive idempotent $e\in\Omega$, to link the two vertices for which the sum of the corresponding basis elements is equal to $D_\Omega([e \bar\Omega])$. By this definition, two edges labelled by $[e_1\bar\Omega]$ and $[e_2\bar\Omega]$ are adjacent if and only if $e_1\Omega e_2\neq 0$, which happens if and only if $e_1\bar{\Omega} e_2\neq 0$.  If $\dim_k e_1 \bar\Omega e_1 > \dim_k e_1\bar\Omega e_2$ for all $[e_1\bar \Omega]\neq [e_2\bar \Omega]\in \Gro(\bar{\Omega})$, then $G_\Omega$ is a simple graph. By the discussion above (applied to all spanning trees of the Brauer graph) we get simple graphs $G_\Lambda$ and $G_\Gamma$, and $G_\Lambda$ is just the Brauer graph of $(Q,f)$.

	The isomorphism $\varphi:\ \bar\Lambda \longrightarrow \bar{\Gamma}$ induces a bijection between the edges of $G_\Lambda$ and $G_\Gamma$ by sending $[e\bar\Lambda]$, for a primitive idempotent $e\in \Lambda$, to $[\varphi(e)\bar{\Gamma}]$. By the discussion in the previous paragraph, this assignment preserves adjacency of edges.  That is, $\varphi$ induces an isomorphism $\widehat \varphi$ between the line graphs of $G_\Lambda$ and $G_\Gamma$. It follows from Whitney's line graph theorem \cite[Corollary]{HemmingerLineGraph} that either $G_\Lambda$ is isomorphic to $G_\Gamma$, or one of them is the complete graph $K_3$ and the other one is the complete bipartite graph $K_{3,1}$. The latter case cannot occur since both $G_\Lambda$ and $G_\Gamma$ have the same number of vertices, as their number is equal to the rank of $\Gro(A)\cong \Gro(B)$.

	The discussion following equation~\eqref{eqn rho brtree} implies that if $e=e_1+\ldots+e_r$ is a sum of orthogonal primitive idempotents in $\Lambda$ corresponding to a spanning tree of $G_\Lambda$, and $e'$ is a lift to $\Gamma$ of $\varphi(e)$, then $G_{e\Lambda e}$ is isomorphic to $G_{e'\Gamma e'}$. Note that, by definition, $G_{e\Lambda e}$ is a subgraph of $G_\Lambda$, $G_{e'\Gamma e'}$ is a subgraph of $G_{\Gamma}$, and $G_{e'\Gamma e'}$ is the image of $G_{e\Lambda e}$ under $\widehat{\varphi}$. That is, $\widehat \varphi$ maps spanning trees to  trees. Since every star in $G_\Lambda$ is contained in a maximal subtree, which is the same as a spanning tree, the map $\widehat \varphi$ maps stars to trees as well. But if a line graph isomorphism maps a star to a tree, then that tree must again be a star (all other possible images of a star contain a triangle). It follows that $\widehat\varphi$ maps stars to stars, and because it is a line graph isomorphism between two isomorphic graphs (with the same number of stars), the inverse of $\widehat \varphi$ preserves stars as well. By \cite[Theorem 1]{HemmingerLineGraph} the map $\widehat \varphi$ is induced by a graph isomorphism.

	We can attach the multiplicity $([V],[V])_{B}$ to the vertex of $G_\Gamma$ belonging to the distinguished basis element $[V]\in \Gro(B)$. The analogously defined multiplicities on $G_\Lambda$ coincide with the multiplicities already defined on it (as a Brauer graph). The Cartan matrices of $\bar{\Lambda}$ and $\bar{\Gamma}$ determine the multiplicities of the vertices of $G_\Lambda$ and $G_\Gamma$ by the same argument that was used in proving condition~\ref{assumption lambda 5}. In particular, $\widehat \varphi$ induces an isomorphism between $G_\Lambda$ and $G_\Gamma$ as graphs with multiplicities. By definition, these graphs determine the matrices of the maps $D_\Lambda$ and $D_\Gamma$ with respect to the respective distinguished bases (remember that both vertices and edges are labelled by such basis elements). The graph isomorphism $\widehat \varphi$ induces bijections between the sets of vertices and the sets of edges, which produces maps $\iota_1$ and $\iota_2$ as required in assumption~\ref{assumption decomp determined} (the fact that these are isometries follows from the fact that $\widehat \varphi$ preserves multiplicities). This finishes the verification of the conditions of Lemma~\ref{lemma pic}, which implies that there is a group  $\TrPicent(\Lambda)\leq \mathcal H_I \leq \TrPic_S(\Lambda)$ (which is the same for any $\mm\in I$) such that $\Pic_k(\bar{\Lambda})\cdot \bar{\mathcal H}_I$ is equal to $\TrPic_k(\bar{\Lambda})$.

	To finish the proof we will need to look at automorphism groups. Note that all algebras we consider are basic, so Picard groups and outer automorphism groups coincide. Let ${\gamma\in \Aut_S(\Lambda)\cap \Aut_R(\Lambda)}$ represent an element which lies in the kernel of the natural map ${\Out_R(\Lambda)\longrightarrow \Out_k(\bar{\Lambda})}$. Then $\gamma$ must act trivially on $\Gro(\bar{\Lambda})$. So we can assume that $\gamma(e)=e$ for all $e\in Q_0$, and by further modifying $\gamma$ by an inner automorphism we can assume that $\gamma$ is trivial on $\bar{\Lambda}$. In particular, $\gamma$ induces an automorphism of the Brauer graph of $(Q,f)$ which fixes all edges, and unless the Brauer graph consists of a single edge (which we assume it does not) such an automorphism must fix all vertices of the Brauer graph as well. So $\gamma$ acts trivially on $\Gro(A)$ and therefore becomes inner in $A$ (since $\gamma$ is $S$-linear and $A$ is a split semisimple $L$-algebra). Hence we can assume that $\gamma$ is induced by conjugation by
	\begin{equation}
		u=\sum_{e\in Q_0} \sum_{\alpha\langle g \rangle \in Q_1/\langle g \rangle} u_{e,\alpha\langle g\rangle}\cdot \eps_{\alpha\langle g\rangle}e,
	\end{equation}
	where $\eps_{\alpha\langle g\rangle}$ denotes the primitive idempotent in $A$ belonging to the $g$-orbit $\alpha\langle g \rangle$, and ${u_{e,\alpha\langle g\rangle} \in k\sqq{X^{1/m_\alpha}}}^\times$. We can multiply $u$ by a unit of $Z(A)$ and assume without loss of generality that for each $\alpha\in Q_1$ there is an $e(\alpha)\in Q_0$ such that $u_{e(\alpha),\alpha\langle g\rangle}=1$. Fix an orbit $\alpha\langle g\rangle$, and pick the representative $\alpha$ such that $e(\alpha)$ is the source of $\alpha$. Then conjugation by $u$ maps $\alpha g(\alpha) \cdots g^i(\alpha)$ (for any $i$) to $u_{t(g^i(\alpha)), \alpha\langle g \rangle}\cdot \alpha g(\alpha) \cdots g^i(\alpha)$, where $t(g^i(\alpha))$ denotes the target of $g^i(\alpha)$. At the same time we know that  the image of this element in $\bar \Lambda$ must be equal to  $\alpha g(\alpha) \cdots g^i(\alpha)$, which implies that $u_{t(g^i(\alpha)), \alpha\langle g \rangle}$ must lie in $1+X^{1/m_{\alpha}} k\sbb{X^{1/m_{\alpha}}}$. It follows that $u_{e,\alpha\langle g\rangle}$  lies in $1+X^{1/m_{\alpha}} k\sbb{X^{1/m_{\alpha}}}$ for all $\alpha\in Q_1$ and $e\in Q_0$. But from the description of $Z(\Lambda)$ in equation~\eqref{eqn z gamma} it is clear that then $u\in \bigoplus_{e\in Q_0} Z(\Lambda)e\subseteq \Lambda$, that is, $\gamma$ is inner. It follows that $\Out_S(\Lambda)\cap \Out_R(\Lambda)$ embeds into $\Out_k(\Lambda)$, and in particular $\bar{\mathcal H}_I\cong \mathcal H_I$.

	A permutation of the distinguished basis of $\Gro(\bar \Lambda)$ is induced by an automorphism of $\bar{\Lambda}$ if and only if it corresponds to a multiplicity preserving automorphism of the Brauer graph of $(Q,f)$. Similarly, there is an element of $\Aut_S(\Lambda)$ acting isometrically on $\Gro(A)$ (carrying the bilinear form coming from the $K$-algebra structure of $A$) which induces a given permutation of the distinguished basis of $\Gro(\Lambda)\cong \Gro(\bar{\Lambda})$ if and only if the permutation corresponds to a multiplicity preserving automorphism of the Brauer graph. Such an element of $\Aut_S(\Lambda)$  then gives rise to an element of $\mathcal H_I$. It follows that $\Pic_k(\bar{\Lambda}) \cdot \mathcal H_I 	= \PicS_k(\bar{\Lambda}) \cdot \mathcal H_I$, proving the first assertion.

	For the second assertion we will verify conditions~\ref{assumption b4}~and~\ref{assumption b5} of Lemma~\ref{lemma extra structure}. First let us show \ref{assumption b4}, that is, $\Picent(\bar\Lambda)\cap \PicS_k(\bar{\Lambda})\subseteq \mathcal H_I$. Let ${\gamma\in \Autcent(\bar \Lambda)}$ be an automorphism fixing the all $e\in Q_0$. For each $\alpha \in Q_1$ the element
	\begin{equation}
		z_{\alpha\langle g \rangle} = \sum_{\beta \in \alpha\langle g\rangle} \beta g(\beta)\cdots g^{n_\alpha-1}(\beta) \in \bar{\Lambda}
	\end{equation}
	is central and therefore fixed by $\gamma$, which implies $\gamma(\alpha g(\alpha)\cdots g^{n_\alpha-1}(\alpha))=\alpha g(\alpha)\cdots g^{n_\alpha-1}(\alpha)$. For each $0\leq i \leq n_\alpha-1$ we have $\gamma(g^i(\alpha))=u_i\cdot  g^i(\alpha)$ for some unit $u_i \in k[z_{\alpha\langle g \rangle}]$, and $u_0u_1\cdots u_{n_\alpha-1}z_{\alpha\langle g \rangle}=z_{\alpha\langle g \rangle}$. Write $u_i= (1+r_i)\cdot a_i$ for some $r_i\in z_{\alpha\langle g \rangle} k[z_{\alpha\langle g \rangle}]$ and $a_i\in k^\times$. Since $z_{\alpha\langle g \rangle}$ annihilates all arrows not contained in the $g$-orbit of $\alpha$ we can construct a unit $v\in\bar\Lambda$ such that $v \beta v^{-1}=\beta$ whenever $\beta\not\in \alpha\langle g\rangle$ and  $vg^i(\alpha)v^{-1} = (1+r_i)^{-1}g^i(\alpha)$ for all $1\leq i \leq n_\alpha-1$. If we compose $\gamma$ and conjugation by $v$ we can assume without loss of generality that $u_i\in k^\times$ for all $1\leq i \leq n_\alpha - 1$, which by $u_0u_1\cdots u_{n_\alpha-1}z_{\alpha\langle g \rangle}=z_{\alpha\langle g \rangle}$ implies that all $u_i$ lie in $k^\times$ (to see this note that if an element in $k[z_{\alpha\langle g \rangle}]$ does not annihilate the arrow $\alpha$, then it also does not annihilate $z_{\alpha\langle g \rangle}$). Note that we did not alter the image of any $\beta\not\in \alpha\langle g \rangle$. We can therefore assume without loss of generality that there are $u_\alpha\in k^\times$ such that $\gamma(\alpha)=u_\alpha\alpha$ for all $\alpha\in Q_1$, and moreover $\prod_{\beta\in \alpha \langle g \rangle}u_\beta =1$ for all $\alpha$. Now define a central automorphism of the order $\Lambda=\GammaTw(Q,f,\mm)$ that maps any $\alpha\in Q_1$ to $u_\alpha\alpha$ (this would not have been possible with the original $u_i$'s, since their product was equal to one only in $\bar{\Lambda}$ but not necessarily in $kQ$). Using the presentation given in Proposition~\ref{prop biserial order} one sees that this automorphism is well-defined and one checks that it is indeed trivial on $Z(\Lambda)$. This shows that $\Picent(\bar\Lambda)\cap \PicS_k(\bar{\Lambda})\subseteq \xoverline{\Picent(\Lambda)}\cong \Picent(\Lambda) \subseteq \mathcal H_I$, proving \ref{assumption b4}.

	For condition~\ref{assumption b5} we note that an element of $\mathcal H_I$ induces a permutation  on the elements of the form $z_{\alpha\langle g \rangle}$ defined earlier, as these are the reductions modulo $X$ of the elements ${X^{1/m_\alpha} \eps_{\alpha\langle g \rangle}\in Z(\Lambda)}$. An element of $\PicS_k(\bar{\Lambda})$ is induced by an automorphism $\gamma$ of $\bar{\Lambda}$ which fixes all $e\in Q_0$, and therefore can only map $z_{\alpha\langle g \rangle}$ to $z_{\beta\langle g \rangle}$ for $\alpha\langle g \rangle \neq \beta \langle g\rangle$ if the exact same vertices appear as sources of arrows in the orbits $\alpha\langle g \rangle$ and $\beta \langle g\rangle$. If two distinct  $g$-orbits of arrows have more than one vertex in common (as a source of an arrow), then by definition there is a double edge in the Brauer graph, which we do not allow. Hence $\gamma$ can only map $z_{\alpha\langle g \rangle}$ to $z_{\beta\langle g \rangle}$ if both $\alpha$ and $\beta$ are loops attached to the same vertex. In that case the Brauer graph has only a single edge, a case we exclude. So, if $\gamma$ induces the same automorphism of $Z(\bar{\Lambda})$ as some element of $\mathcal H_I$, then  $\gamma$ fixes  $z_{\alpha\langle g \rangle}$ for all $\alpha\in Q_1$.   But then $\gamma$ also fixes $ez_{\alpha\langle g \rangle}^{m_\alpha}$ for all $e\in Q_0$ and $\alpha\in Q_1$. The latter elements together with the $z_{\alpha\langle g \rangle}$ generate $Z(\bar{\Lambda})$. That is, $\gamma$ must be trivial on $Z(\bar{\Lambda})$, which implies~\ref{assumption b5}.

	It follows by Lemma~\ref{lemma extra structure} that $\TrPicent(\bar \Lambda) = \Ker(\mathcal H_I \longrightarrow \Aut_k(Z(\bar{\Lambda})))$. As already discussed above, an element of $\mathcal H_I$ induces a permutation of the $z_{\alpha\langle g \rangle}$ for $\alpha\in Q_1$, and since we are assuming that there is more than one edge in the Brauer graph we have $z_{\alpha\langle g \rangle}\neq z_{\beta\langle g \rangle}$ in $\bar{\Lambda}$ whenever $\alpha\langle g \rangle \neq \beta \langle g \rangle$.  In particular $\TrPicent(\bar\Lambda)\subseteq \xoverline{\TrPicent(\Lambda)}\cong \TrPicent(\Lambda)$, and we can therefore write $\TrPicent(\bar \Lambda) = \Ker(\TrPicent(\Lambda) \longrightarrow \Aut_k(Z(\bar{\Lambda})))$. Now $Z(\bar{\Lambda})$ is generated by the $z_{\alpha\langle g \rangle}$ for $\alpha \in Q_1$ and the elements of $\soc(\bar{\Lambda})=\langle  ez_{\alpha\langle g\rangle}^{m_\alpha}\ |\ e\in Q_0,\ \alpha\in Q_1\rangle_k$ (one obtains this from the presentation of $\bar{\Lambda}$).  In particular, $\soc(\bar{\Lambda})=\soc(Z(\bar{\Lambda}))$, and an element of $\TrPicent(\Lambda)$ induces the identity on $Z(\bar{\Lambda})$ if and only if it maps the elements $ez_{\alpha\langle g\rangle}^{m_\alpha}$ to multiples of themselves, since it fixes the elements $z_{\alpha\langle g \rangle}$ anyway.  In this part of the proof we can also assume that $\bar{\Lambda}$ is symmetric, so Lemma~\ref{lemma action centre} applies and it follows that
	\begin{equation}
		\TrPicent(\bar \Lambda) = \Ker(\TrPicent(\Lambda) \longrightarrow \Aut_k(k\otimes_{\Z} \Gro({\Lambda}))),
	\end{equation}
	which is independent of the $R$-algebra structure on $\Lambda$. If we let $\mathcal G$ denote the right hand side of the expression above then  the second assertion follows.
\end{proof}

\begin{remark}
	\begin{enumerate}
		\item While there may be some twisted Brauer graph algebras of independent interest, the main intended application of Theorem~\ref{thm Brauer graph} are twisted Brauer graph algebras which are isomorphic to their ordinary counterparts (e.g. in characteristic two, or when the Brauer graph is bipartite). In those cases the symmetry condition in the second part of Theorem~\ref{thm Brauer graph} is automatically met.
		\item One case we are particularly interested in are the algebras of dihedral type $\mathcal D(3K)^{a_1,a_2,a_3}$ from Erdmann's classification \cite{TameClass} in characteristic two, where $a_1,a_2,a_3\geq 1$. These are Brauer graph algebras, where the graph is a triangle (i.e. a complete graph on three vertices). Obviously Theorem~\ref{thm Brauer graph} applies, but these algebras are also silting-connected, which by Lemma~\ref{lemma extra structure} implies that
		      \begin{equation}
			      \PicS_k(\mathcal D(3K)^{a_1,a_2,a_3}) \unlhd \TrPic_k(\mathcal D(3K)^{a_1,a_2,a_3}).
		      \end{equation}
		      However, condition~\ref{assumption B2} of Lemma~\ref{lemma extra structure} fails to hold, and we do not get a semi-direct product decomposition as we did for $\mathcal Q(3K)^{a_1,a_2,a_3}$. Specifically, the quiver $Q$ for this algebra is an in Proposition~\ref{prop semimimple q3k}, and for each $c \in k^\times$ there is an automorphism sending $\alpha_1$ to $c\alpha_1$, $\beta_2$ to $c^{-1}\beta_2$, and fixing all other arrows. These automorphisms are trivial on the centre of the algebra, and lift to the $k\sbb{X}$-order $\Lambda$ used in the proof of Theorem~\ref{thm Brauer graph}, that is, they lie in $\mathcal H_I\cap \PicS_k(\bar{\Lambda})$ (in fact, this intersection consists exactly of the automorphisms we just described).
	\end{enumerate}
\end{remark}

Of course Theorem~\ref{thm Brauer graph} also applies to Brauer tree algebras, which are Brauer graph algebras whose graph is a tree and only a single vertex may have multiplicity bigger than one. Their derived Picard groups were already described in \cite{ZvonarevaBrauerStar, VolkovZvonareva} (which to some extent motivated Theorems~\ref{thm trpic q3k}~and~\ref{thm Brauer graph}). We can recover the fact that the derived Picard group decomposes as a direct product of $\PicS_k(\bar{\Lambda})$ and a group whose isomorphism type is mostly independent of multiplicities. If more than one vertex has multiplicity bigger than one, this becomes a semidirect product.

\begin{prop}\label{prop semidirect brauer tree}
	Let $k$ be algebraically closed and let $A(T, \mm)$ denote a Brauer graph algebra whose graph $T$ is a star with multiplicities $\mm$. Define $\mathcal I$ and ``$\sim$'' as in Theorem~\ref{thm Brauer graph}. For any $I\in \mathcal I /\sim$ there is a group $\mathcal H_I$ such that
	\begin{equation}
		\TrPic_k(A(T, \mm)) \cong \PicS_k(A(T, \mm))\rtimes \mathcal H_I.
	\end{equation}
	for all $\mm\in I$.

	Given a multiplicity function $\mm$ which assigns multiplicity one to all except the central vertex, there are exactly two possibilities for the equivalence class $I\in\mathcal I / \sim$  containing $\mm$ (one in which the multiplicity of the central vertex is also equal to one, and one in which it is bigger than one), and
	\begin{equation}
		\TrPic_k(A(T, \mm)) \cong \PicS_k(A(T, \mm))\times \mathcal H_I.
	\end{equation}
\end{prop}
\begin{proof}
	The first part of the assertion follows from Lemma~\ref{lemma extra structure}. Clearly \ref{assumption B1} holds since $A(T,\mm)$ is silting-discrete (see \cite[Theorem 6.7]{BrauerGraphTauTilt}). An argument like the one in Proposition~\ref{prop picent} shows that  $\Picent(\GammaTw(Q,f))=1$, where $\LambdaTw(Q,f,\mm)$ is  the twisted Brauer graph algebra isomorphic to $A(T, \mm)$ (note that $Q$ is not a circular quiver, but rather a circular quiver with a loop attached to each vertex). This implies condition~\ref{assumption B2} of Lemma~\ref{lemma extra structure}. Condition~\ref{assumption B3} is also satisfied, as both images in $\Aut(\Gro(A(T,\mm)))$ correspond precisely to the automorphisms of the tree  $T$. By applying Lemma~\ref{lemma extra structure} on top of Theorem~\ref{thm Brauer graph} we get a semidirect product decomposition $\PicS_k(A(T, \mm))\rtimes \mathcal H_I$. For the second part of the assertion we should note that $\Picent(A(T,\mm))=1$ irrespective of $\mm$, and the centre of $A(T,\mm)\cong \LambdaTw(Q,f,\mm)$ is the reduction of the centre of $\GammaTw(Q,f,\mm)$. In particular, the action of $\mathcal H_I$ on $\PicS_k(A(T,\mm))$ factors through $\Aut(\Gro(k\sqq{X} \otimes_{k\sbb{X}} \GammaTw(Q,f,\mm)))$. Hence one only needs to check that the automorphisms of $Z(A(T,\mm))$ coming from automorphisms of $Z(\GammaTw(Q,f,\mm))$ commute with those which are induced by automorphisms of $A(T,\mm)$. This is true if only the central vertex is allowed to have multiplicity bigger than one, and false otherwise (this requires a computation).
\end{proof}

\bibliographystyle{halpha}
\bibliography{refs}

\end{document}